\documentclass[10pt]{amsart}
\usepackage{amssymb}
\usepackage{amsbsy}
\usepackage{amscd}
\usepackage[mathscr]{eucal}
\usepackage{verbatim}
\usepackage{version}
\usepackage{color}
\usepackage[matrix,arrow]{xy}

\def\cal{\mathcal}
\def\Bbb{\mathbb}
\def\frak{\mathfrak}

\newenvironment{NB}{
\color{red}{\bf NB}. \footnotesize 
}{}
\excludeversion{NB}
\newenvironment{NB2}{
\color{blue}{\bf NB}. \footnotesize
}{}

\newcommand{\bb} {\mathbb}
\newcommand{\bl} {\mathbf}
\newcommand{\frk}{\mathfrak}

\newcommand{\ep}{\epsilon}

\newcommand{\wt}{\widetilde}

\newcommand{\simto}{\xrightarrow{\ \sim\ }}

\newcommand{\ev}{\operatorname{ev}}

\newcommand{\Eq}  {\operatorname{Eq}}
\newcommand{\FM}  {\operatorname{FM}}
\newcommand{\GL}  {\operatorname{GL}}
\newcommand{\LieO}{\operatorname{O}}
\newcommand{\Ker} {\operatorname{Ker}}

\newcommand{\SL}  {\operatorname{SL}}

\newcommand{\mpr}    [1]{\langle #1\rangle}

\newcommand{\Piczero}[1]{\operatorname{Pic}^0(#1)}
\newcommand{\Hilb}   [2]{\operatorname{Hilb}^{#1}(#2)}

\newcommand{\ch}{\operatorname{ch}}

\newcommand{\Coh}{\operatorname{Coh}}

\newcommand{\Hom}{\operatorname{Hom}}

\newcommand{\codim}{\operatorname{codim}}

\newcommand{\rk}{\operatorname{rk}}

\newcommand{\NS}{\operatorname{NS}}

\newcommand{\td}{\operatorname{td}}

\newcommand{\Amp}{\operatorname{Amp}}

\newcommand{\tr}{\operatorname{tr}}
\newcommand{\alg}{\operatorname{alg}}

\newcommand{\Sym}{\operatorname{Sym}}

\font\b=cmr10 scaled \magstep5
\def\bigzerou{\smash{\lower1.7ex\hbox{\b 0}}}

\numberwithin{equation}{section}

\theoremstyle{plain}
 \newtheorem{thm}{Theorem}[section]
 \newtheorem{lem}[thm]{Lemma}
 \newtheorem{prop}[thm]{Proposition}
 \newtheorem{cor}[thm]{Corollary}
 
\theoremstyle{definition}
 \newtheorem{defn}[thm]{Definition}
 
\theoremstyle{remark}
 \newtheorem{rem}[thm]{Remark}

\begin{document}


\title{Bridgeland's stabilities on abelian surfaces}
\author{Shintarou Yanagida, K\={o}ta Yoshioka}
\address{Department of Mathematics, Faculty of Science,
Kobe University,
Kobe, 657, Japan
}
\email{yanagida@math.kobe-u.ac.jp,
yoshioka@math.kobe-u.ac.jp}

\thanks{The first author is supported by JSPS Fellowships 
for Young Scientists (No.\ 21-2241). 
The second author is supported by the Grant-in-aid for 
Scientific Research (No.\ 22340010), JSPS}

\subjclass[2010]{14D20}

\begin{abstract}
In this paper, we shall study the structure of walls
for Bridgeland's stability conditions
on abelian surfaces.
In particular, we shall study the structure of 
walls for the moduli spaces of rank 1 complexes
on an abelian surface with the Picard number 1.
\end{abstract}
\maketitle


\section{Introduction}

Let $X$ be an abelian surface over a field ${\frak k}$.
Denote by $\Coh(X)$ the category of coherent sheaves on $X$,
by $\bl{D}(X)$ the bounded derived category of $\Coh(X)$
and by $K(X)$ the Grothendieck group of $\bl{D}(X)$.

For $\beta \in \NS(X)_\bb{Q}$ and an ample divisor
$\omega \in \Amp(X)_{\Bbb Q}$,
Bridgeland \cite{Br:3} constructed a stability condition
$\sigma_{\beta,\omega}=(\frk{A}_{(\beta,\omega)},Z_{(\beta,\omega)})$ 
on $\bl{D}(X)$.
Here $\frk{A}_{(\beta,\omega)}$ is a tilting of $\Coh(X)$, 
and $Z_{(\beta,\omega)}:K(X) \to \bb{C}$ is a group homomorphism 
called the stability function.
In terms of the Mukai lattice 
$(H^*(X,\bb{Z})_{\alg}, \langle \cdot,\cdot \rangle)$, 
$Z_{(\beta,\omega)}$ is given by 
\begin{align*}
 Z_{(\beta,\omega)}(E)=\langle e^{\beta+\sqrt{-1}\omega},v(E) \rangle,
 \quad
 E \in K(X).
\end{align*}
Here $v(E):=\ch(E)$ is the Mukai vector of $E$.
Hereafter for an object $E \in \bl{D}(X)$, we abbreviately write 
$Z_{(\beta,\omega)}(E) := Z_{(\beta,\omega)}([E])$, 
where $[E]$ is the class of $E$ in $K(X)$.
For abelian surfaces, these kind of stability conditions forms
a connected component of the space of stability conditions
up to the action of the universal cover
$\widetilde{\GL}^+(2,{\Bbb R})$ of $\GL^+(2,{\Bbb R})$
as stated in \cite[sect. 15]{Br:3}.

Let 
$\phi_{(\beta,\omega)}
 :\frk{A}_{(\beta,\omega)} \setminus \{ 0\} \to (0,1]$ 
be the phase function, which is defined by 
$$
 Z_{(\beta,\omega)}(E)=|Z_{(\beta,\omega)}(E)|
 e^{ \sqrt{-1}\pi \phi_{(\beta,\omega)}(E)}
$$ 
for $0 \ne E \in {\frak A}_{(\beta,\omega)}$.
Let $M_{(\beta,\omega)}(v)$ be the moduli space of
$\sigma_{(\beta,\omega)}$-semi-stable objects $E$ with $v(E)=v$.
It is a projective scheme, if $(\beta,\omega)$ is general
(\cite[Thm. 0.0.2]{MYY:2011:2}).
The stability of objects depends only on a chamber of the parameter
space $\NS(X)_{\Bbb R} \times \Amp(X)_{\Bbb R}$.
For a special chamber, 
the stability coincides with Gieseker stability
(\cite{Br:3}, \cite{Toda}, \cite{Bayer}, \cite{MYY:2011:1}).
For the analysis of Gieseker stability, Fourier-Mukai transforms
are very useful tool. However 
the Fourier-Mukai transform
does not preserve Gieseker stability in general,
since the category of coherent sheaves is not preserved. 
On the other hand, the Fourier-Mukai transform induces an isomorphism
of the moduli spaces of Bridgeland stable
objects as a consequence of \cite[Prop. 10.3]{Br:3}.
So it is very natural to study the moduli of Bridgeland stable
objects and its dependence on the parameter
even for the study of Gieseker stability. 

For abelian surfaces, the studies on the dependence of the parameters,
i.e., those of wall and chamber structures, are started by
two groups Maciocia and Meachan \cite{MM},
and Minamide, Yanagida and Yoshioka 
\cite{MYY:2011:1}, \cite{MYY:2011:2}. 
In this paper, we continue our study on the wall and chamber structures.
In particular, we shall study 
the behavior of the structures under Fourier-Mukai
transforms.
We are mainly interested in the most simple case, that is,
the case where $\NS(X)={\Bbb Z}H$.
In this case, 
the parameter space is the upper half plane,
the action of Fourier-Mukai transforms
can be described by an arithmetic subgroup
of $\SL(2,{\Bbb R})$ and the action on the upper half plane
is the natural one.
So we can study the structure in detail.

Let us explain the contents of this paper.
We first define the wall and chamber by using the characterization of walls
in \cite[Prop. 4.2.2]{MYY:2011:2}.
Then we present a few basic properties of walls and chambers.
As we explained,
by the paper of Bridgeland (\cite[Prop. 10.3]{Br:3}), 
Fourier-Mukai transforms preserve the stability 
of objects, that is, it induces an isomorphism of the moduli spaces. 
We shall show that the taking dual functor 
${\cal D}_X$ also preserves 
the stability of objects (Theorem \ref{thm:B:3-10.3}). 
These are done in sections \ref{sect:walls} and \ref{sect:FM} 

In \cite{YY}, we introduced a useful notion 
\emph{semi-homogeneous presentations}.
It is a presentation of a coherent sheaf
as the kernel or the cokernel
of a homomorphism $V_{-1} \to V_0$
of semi-homogeneous sheaves $V_{-1}, V_0$ with
some numerical conditions.
Extracting the numerical conditions
from $V_{-1}$ and $V_0$, we also introduced the notion
\emph{numerical solutions}
and constructed moduli spaces
of simple complexes $V_{-1} \to V_0$ associated to numerical solutions. 
In \cite{MYY:2011:1}, we found a relation 
between these moduli spaces and the wall crossing behavior
for Bridgeland stability conditions.
In this paper, we give a supplement
of this relation.
Thus we shall relate a particular
wall called a {\it codimension 0 wall}
for every numerical solution.
If $(\beta,\omega)$ belongs to a codimension 0 wall, then
for any neighborhood $U$ of $(\beta,\omega)$,
there is no $\sigma_{(\beta,\omega)}$-semi-stable object
which is $\sigma_{(\beta',\omega')}$-semi-stable
for all $(\beta',\omega') \in U$.

In section \ref{sect:restricted-parameter},
we fix an ample divisor $H$ and 
study special stability conditions
$\sigma_{(\beta+sH,tH)}$
parametrized by the upper half plane
$\{ (s,t) \mid s,t\in {\Bbb R}, s>0  \}$.
In the $(s,t)$-plane,
the equations of walls are very simple.
They define circles or lines
forming a pencil of circles passing through imaginary
points (Remark \ref{rem:base-point}). 
Thus they do not intersect each other.
For a principally polarized abelian surface
with Picard number 1,
these kind of results are obtained by 
Maciocia and Meachan \cite{MM}.
Thus our results are generalization of theirs.

We also explain that there are one or two 
unbounded chambers which parametrize 
Gieseker semi-stable sheaves.
In \cite{Stability}, we showed that Gieseker's stability is 
preserved under the Fourier-Mukai transform, if the degree of the stable
sheaf is sufficiently large.
We shall explain the result as an application of this section
(Proposition \ref{prop:asymptotic}).

In section \ref{sect:rho=1},
we assume that
$\NS(X)={\Bbb Z}H$.
We are mainly interested in a Mukai vector $v$ 
which is a Fourier-Mukai transform of 
$1-\ell \varrho_X$.
So we may assume that $v=1-\ell \varrho_X$.

In \cite{YY}, we described the algebraic part of the
Mukai lattice $H^*(X,{\Bbb Z})_{\alg}$
as a lattice in the vector space of
quadratic forms of two variables.
Then we can describe the action
of Fourier-Mukai transforms as a natural $\GL(2,{\Bbb R})$-action. 
By using these results, we shall study the structure of walls.
In particular, we shall classify codimension 0 walls by using
our description of Mukai lattice.
We set $n:=(H^2)/2$.
If $\sqrt{\ell/ n} \in {\Bbb Q}$, then there is a unique wall of
codimension 0, since there is one numerical solution.
Assume that $\sqrt{\ell/ n} \not \in {\Bbb Q}$. In this case,
these are infinitely many numerical solutions, thus we have
infinitely many codimension 0 walls.
Let $G_{n,\ell} \subset \GL(2,{\Bbb R})$
be the subgroup
generated by the cohomological action of
(covariant or contravariant) Fourier-Mukai transforms
preserving $\pm v$.
$G_{n,\ell}$ acts on the set of walls.
We show that there are finitely many 
$G_{n,\ell}$-orbits of walls and
the set of codimension 0 walls forms an orbit of 
$G_{n,\ell}$. 
Each orbit has two
accumulation points  
$(\pm \sqrt{\ell/n},0)$ in the $(s,t)$-plane. 
An observant reader will see that
the main part of section of \ref{sect:rho=1} essentially 
appeared in 
\cite{YY:preprint} without using Bridgeland stability condition.

In section \ref{sect:example}, 
we shall study the structure of walls for
$v=1-\ell \varrho_X$, $\ell \leq 4$ 
on a principally polarized abelian surface $X$.
As an application, we shall classify $M_H(v)$
for a primitive Mukai vector with $\langle v^2 \rangle/2 \leq 4$.

In \cite{Mukai:1979} and \cite[Thm. 3]{Mukai:1980},
Mukai announced that
$M_H(v) \cong X \times \Hilb{\langle v^2 \rangle/2}{X}$
for a primitive Mukai vector with $\langle v^2 \rangle/2=1,2,3$.
Moreover he determined the Fourier-Mukai transform
which induces the isomorphism.
By using the structure of walls,
we give an explanation of Mukai's results for 
$\langle v^2 \rangle/2=2,3$.
It is quite surprising that Mukai discovered his results 30 years ago
without using 
Bridgeland's stability conditions.

In appendix, we continue to assume that 
$X$ is an abelian surface with $\NS(X)={\Bbb Z}H$ and
$(H^2)/2=n$.
We shall identify the period space
with the upper half plane. 
Then we show that the action of auto-equivalences 
is the action of the modular group $\Gamma_0(n)$.

Finally we would like to mention related works
which appeared during our preparation of this manuscript. 
We note that
the examples of Bridgeland stability conditions
in this paper are generalized to
an arbitrary projective surfaces by
Arcara and Bertram \cite{AB}. 
For these stability conditions,
Maciocia \cite{Ma} studied the structure of walls.
In particular, he  
proved similar results to
section \ref{sect:restricted-parameter}
in a much more general context. 
For the stability conditions on principally
polarized abelian surfaces,
Meachan \cite{Meachan} studied the structure of walls
in detail. In particular he 
independently found
examples of walls with accumulation points.

\section{Preliminaries on Bridgeland's stability condition}
\label{sect:walls}

As in the introduction,
let $X$ be an abelian surface over a field $\frk{k}$,
and fix an ample divisor $H$ on $X$.


\subsection{Notations for Mukai lattice}

We set $A^*_{\alg}(X)=\oplus_{i=0}^2 A^{i}_{\alg}(X)$ 
to be the quotient of the cycle group of $X$ 
by the algebraic equivalence.
Then we have $A^{0}_{\alg}(X)\cong \bb{Z}$,
$A^{1}_{\alg}(X)\cong \NS(X)$ and 
$A^{2}_{\alg}(X)\cong \bb{Z}$.
We denote the fundamental class of $A^{2}_{\alg}(X)$ by $\varrho_X$,
and express an element $x\in A^{*}_{\alg}(X)$ 
by $x=x_0+x_1+x_2 \varrho_X$ 
with $x_0 \in \bb{Z}$, $x_1 \in \NS(X)$ and $x_2 \in \bb{Z}$.
The lattice structure $\langle \cdot,\cdot\rangle$
of $A^*_{\alg}(X)$ is given 
by  
\begin{equation}\label{eq:mukai_pairing}
 \langle x,y \rangle:= (x_1,y_1)-(x_0 y_2+x_2 y_0),
\end{equation}
where $x=x_0+x_1+x_2 \varrho_X$ and $y=y_0+y_1+y_2 \varrho_X$. 
We shall call $(A^*_{\alg}(X), \langle \cdot,\cdot \rangle)$ 
the Mukai lattice for $X$.
In the case of $\frk{k}=\bb{C}$, this lattice is 
sometimes denoted by $H^{*}(X,\bb{Z})_{\text{alg}}$ in literature. 
In this paper, we shall use the symbol $H^{*}(X,\bb{Z})_{\text{alg}}$ 
even when $\frk{k}$ is arbitrary.

The Mukai vector $v(E)\in H^*(X,\bb{Z})_{\alg}$ 
for $E \in \Coh(X)$ is defined by 
\begin{equation*}
v(E):=\ch(E) \sqrt{\td_X}
=\ch(E)
=\rk E+c_1(E)+ \chi(E) \varrho_X.
\end{equation*}
We also use the vectorial notation 
$$
v(E)=(\rk E,c_1(E),\chi(E)).
$$
For an object $E$ of ${\bl{D}}(X)$, 
$v(E)$ is defined by $\sum_{k}(-1)^k v(E^k)$,
where $(E^k)=(\cdots \to E^{-1} \to E^0 \to E^1 \to \cdots)$
is the bounded complex representing the object $E$.

We take an ample ${\Bbb Q}$-divisor $H$.
For 
$v \in H^*(X,\bb{Z})_{\alg}$ and $\beta \in \NS(X)_\bb{Q}$,
we set
\begin{equation}\label{eq:v:rda}
r_\beta(v):=-\langle v,\varrho_X \rangle,\quad
a_\beta(v):=-\langle v,e^\beta \rangle,\quad
d_{\beta,H}(v):=\frac{\langle v,H+(H,\beta)\varrho_X \rangle}{(H^2)}.
\end{equation}
If the choice of $H$ is clear, then we 
write $d_\beta(v):=d_{\beta,H}(v)$ for simplicity.
By using \eqref{eq:v:rda}, we have 
\begin{equation}\label{eq:v}
v=r_\beta(v)e^\beta+a_\beta(v) \varrho_X +
(d_{\beta,H}(v) H+D_\beta(v))+(d_{\beta,H}(v) H+D_\beta(v),\beta)\varrho_X,\;
D_\beta(v) \in H^{\perp} \cap \NS(X)_\bb{Q}.
\end{equation}

A Mukai vector $v=(r,\xi,a) \ne 0$ is positive, if
(i) $r>0$ or (ii) $r=0$ and $\xi$ is effective, or
(iii) $r=\xi=0$ and $a>0$.
We denote a positive Mukai vector $v$ by
$v>0$.
A Mukai vector $v$ is called isotropic if
 $\langle v^2 \rangle =0$.

For $\beta \in \NS(X)_\bb{Q}$,
we define the $\beta$-twisted semi-stability
replacing the usual Hilbert polynomial
$\chi(E(nH))$ by $\chi(E(-\beta+nH))$.
Then $v$ is positive if and only if
$\chi(E(-\beta+nH)) >0$ for $E \in {\bf D}(X)$ with
$v(E)=v$ and 
$n \gg 0$.
  
For a positive Mukai vector $v$, 
${\cal{M}}_H^\beta(v)^{ss}$ denotes
the moduli stack of $\beta$-twisted semi-stable sheaves $E$ on $X$
with $v(E)=v$.
$\overline{M}_H^\beta(v)$ denotes
the moduli scheme of $S$-equivalence classes of
$\beta$-twisted semi-stable sheaves $E$ on $X$
with $v(E)=v$
and $M_H^\beta(v)$ denotes the open subscheme consisting of
$\beta$-twisted stable sheaves.
If $\beta=0$, then
we write $\overline{M}_H(v):=\overline{M}_H^\beta(v)$.

For a proper morphism $f:Z_1 \to Z_2$,
we denote the derived pull-back ${\bf L}f^*$ and
the derived direct image ${\bf R}f_*$ by $f^*$ and $f_*$
respectively.   

For ${\bf E} \in {\bf D}(X \times Y)$,
$\Phi_{X \to Y}^{{\bf E}}:{\bf D}(X) \to {\bf D}(Y)$
denotes the integral functor whose kernel is ${\bf E}$:
\begin{equation}
\Phi_{X \to Y}^{{\bf E}}(E)=p_{Y*}(p_X^*(E) \otimes {\bf E}),\;
E \in {\bf D}(X),
\end{equation}
where $p_X$ and $p_Y$ are projections from $X \times Y$ to
$X$ and $Y$ respectively. 
If $\Phi_{X \to Y}^{\bf E}$ is an equivalence, it is
called a \emph{Fourier-Mukai transform}.  
If a Fourier-Mukai transform $\Phi_{X \to Y}^{{\bf E}}$ exists and 
$X$ is an abelian surface, then $Y$ is also an abelian surface and
$\Phi_{X \to Y}^{\bf E}$ induces an isometry
of Mukai lattices $H^*(X,{\Bbb Z})_{\alg} 
\to H^*(Y,{\Bbb Z})_{\alg}$.
We also denote this isometry by $\Phi_{X \to Y}^{\bf E}$.
 
${\cal D}_X(*):={\bf R}{\cal H}om_{{\cal O}_X}(*,{\cal O}_X)$
denotes the taking dual functor.
It is a contravariant functor from ${\bf D}(X)$ to
${\bf D}(X)$.
A contravariant Fourier-Mukai transform is a
composite of a Fourier-Mukai functor and ${\cal D}_X$.
If $X$ is an abelian surface, then it is of the form
$\Phi_{X \to Y}^{{\bf E}[2]} \circ {\cal D}_X={\cal D}_Y \circ
\Phi_{X \to Y}^{{\bf E}^{\vee}}$
with 
${\bf E}^{\vee} := {\cal D}_{X \times Y}({\bf E})$.

\subsection{Stability conditions and wall/chamber structure}

Let us recall the stability conditions given in 
\cite{Br:3} and \cite[\S1]{MYY:2011:1}.
Let
$$
\Amp(X)_{\Bbb R}:=\{x \in \NS(X)_{\Bbb R} \mid (x^2)>0,(x,D)>0 \}.
$$
be the ample cone of $X$,
where $D$ is an effective divisor. 
We take $(\beta,\omega) \in \NS(X)_{\Bbb R} \times \Amp(X)_{\Bbb R}$
and $H \in {\Bbb R}_{>0}\omega$. 
For $E \in K(X)$ with $v=v(E)$ expressed as \eqref{eq:v}, we have
\begin{align*}
Z_{(\beta,\omega)}(E)=&
\langle e^{\beta+\sqrt{-1}\omega},v(E) \rangle\\
=& -a_\beta(v(E))+\frac{(\omega^2)}{2}r_\beta(v(E))
  +d_{\beta,H}(v(E))(H,\omega)\sqrt{-1}.
\end{align*}
Assume that $(\beta,\omega) \in 
\NS(X)_{\Bbb Q} \times \Amp(X)_{\Bbb Q}$.
Let $\frk{A}_{(\beta,\omega)}$ be the tilt of $\Coh(X)$
with respect to
the torsion pair $(\frk{T}_{(\beta,\omega)},\frk{F}_{(\beta,\omega)})$
defined by
\begin{enumerate}
\item
$\frk{T}_{(\beta,\omega)}$ is generated by
$\beta$-twisted stable sheaves with $Z_{(\beta,\omega)}(E) \in \bb{H} \cup
\bb{R}_{<0}$.
\item
$\frk{F}_{(\beta,\omega)}$ is generated by
$\beta$-twisted stable sheaves with $-Z_{(\beta,\omega)}(E) \in \bb{H} \cup
\bb{R}_{<0}$,
\end{enumerate}
where $\bb{H} :=\{z \in \bb{C} \mid \mathrm{Im} z>0 \}$ 
 is the upper half plane.
$\frk{A}_{(\beta,\omega)}$ is the abelian category
in \cite{Br:3}
and it depends only on $\beta$ and the ray ${\Bbb Q}_{>0}\omega$.

\begin{NB}
Since we are considering the abelian surface $X$,
the category $\frk{A}_{(\beta,\omega)}$ 
depends only on $(\beta,{\Bbb Q}_{>0}\omega)$,
as shown in \cite[sect. 1.4]{MYY:2011:1}.
More explicitly, the category is given by 
\begin{align*}
\frk{A}_{(\beta,\omega)}=
\{E \in {\bl{D}}(X) \mid H^i(E)=0 \,  (i \ne -1, 0),\ 
\mu_{\max}(H^{-1}(E)) \leq (\beta,H),\ \mu_{\min}(H^0(E))>(\beta,H) \}
\end{align*}
for any $(\eta,\omega) \in \frk{H}_\bb{R}$.
Here $H^*$ is the cohomology with respect to $\Coh(X)$. 
\end{NB}

Then the pair 
$\sigma_{(\beta,\omega)}=(\frk{A}_{(\beta,\omega)},Z_{(\beta,\omega)})$ 
satisfies the requirement of stability conditions on $\bl{D}(X)$
\cite{Br:3}.
In particular, the (semi-)stability of objects in 
$\frk{A}_{(\beta,\omega)}$ with respect to $Z_{(\beta,\omega)}$
is well-defined.

\begin{defn}\label{defn:phase(v)}
For a non-zero Mukai vector $v \in H^*(X,\bb{Z})_{\alg}$, 
we define $Z_{(\beta,\omega)}(v) \in \bb{C}$ and 
$\phi_{(\beta,\omega)}(v) \in (-1,1]$ by 
\begin{equation*}
Z_{(\beta,\omega)}(v)=
\langle e^{\beta+\sqrt{-1} \omega},v \rangle 
= |Z_{(\beta,\omega)}(v)|e^{\pi \sqrt{-1}\phi_{(\beta,\omega)}(v)}.
\end{equation*}
Then
$$
\phi_{(\beta,\omega)}(v(E))
=\phi_{(\beta,\omega)}(E) 
$$
for $0 \ne 
E \in \frk{A}_{(\beta,\omega)} \cup \frk{A}_{(\beta,\omega)}[-1]$. 
\end{defn}

\begin{defn}
$E \in \bl{D}(X)$ is called \emph{semi-stable} of phase $\phi$,
if there is an integer $n$ such that
$E[-n]$ is a semi-stable object of $\frk{A}_{(\beta,\omega)}$
with $\phi_{(\beta,\omega)}(E[-n])=\phi-n$. 
If we want to emphasize the dependence on the stability condition,
we say that $E$ is \emph{$\sigma_{(\beta,\omega)}$-semi-stable}.
\end{defn}

\begin{defn}\label{defn:moduli}
For a Mukai vector $v$,
$\cal{M}_{(\beta,\omega)}(v)$ denotes the moduli stack of 
$\sigma_{(\beta,\omega)}$-semi-stable objects $E$ of 
$\frk{A}_{(\beta,\omega)}$ with $v(E)=v$.
$M_{(\beta,\omega)}(v)$ denotes the moduli scheme of
the $S$-equivalence classes of  
$\sigma_{(\beta,\omega)}$-semi-stable objects $E$ of 
$\frk{A}_{(\beta,\omega)}$ with $v(E)=v$, if it exists.
\end{defn}

\begin{rem}
If $Z_{(\beta,\omega)}(-v) \in {\Bbb H} \cup {\Bbb R}_{<0}$, then we have 
\begin{equation}
M_{(\beta,\omega)}(v)=\{E[-1] \mid E \in M_{(\beta,\omega)}(-v) \},
\end{equation}
since $\phi_{(\beta,\omega)}(v) \in (-1,0]$.
\end{rem}

\begin{rem}\label{rem:phase}
If we take the phase of $v$ as
$\phi_{(\beta,\omega)}(v) \in (1,2]$, then
\begin{equation}
M_{(\beta,\omega)}(v)=\{E[1] \mid E \in M_{(\beta,\omega)}(-v) \}.
\end{equation}
More generally, if we take the phase 
of $v$ as $\phi_{(\beta,\omega)}(v) \in (n,n+1]$,
then we have 
\begin{equation}
M_{(\beta,\omega)}(v)=\{E[n] \mid E \in M_{(\beta,\omega)}((-1)^n v) \}.
\end{equation}
 
\end{rem}

\begin{NB}
We set $s_0:=d_\beta(v)/\rk v$.
Assume that $\rk v>0$.
Then $d_{\beta+sH}(v)$ is a decreasing function of $s$
with 
$Z_{(\beta+s_0 H,tH)}(v) \in {\Bbb R}_{>0}$.
Requiring the continuity of
$\phi_{(\beta+s_0 H,tH)}(v)$ as a function of $s$,
we should have $\phi_{(\beta+s_0 H,tH)}(v):=0$.
Then we have 
$\phi_{(\beta+s H,tH)}(v) \in (-1,0]$
for $s \geq s_0$. 

Assume that $\rk v<0$.
Then 
$Z_{(\beta+s_0 H,tH)}(v) \in {\Bbb R}_{<0}$.
Hence $E=F[1]$, $F \in M_{(\beta+sH,\omega)}(-v)$ for $s<s_0$.
Thus $\phi_{(\beta+sH,\omega)}(F) \in (0,1]$, which implies that
$\phi_{(\beta+sH,\omega)}(E) \in(1,2]$.

\end{NB}

\begin{NB}
\begin{rem}
Let $E$ be an object of ${\bf D}(X)$ with $v(E)=v$.
If $E[n] \in M_{(\beta,\omega)}(v)$ and 
$E[n'] \in M_{(\beta',\omega')}(v)$, 
$v(E[n])=v(E[n'])=v$. Thus $n$ and $n'$ are even integer.
We note that $\phi_{(\beta,\omega)}(E[n]) \in (-1,1]$.
We first assume that 
$\phi_{(\beta,\omega)}(E[n]) \in (0,1]$. Thus we have
$E[n] \in {\frak A}_{(\beta,\omega)}$, which implies that
$H^i(E[n])=0$ for $i \ne -1,0$.
Since $\phi_{(\beta',\omega')}(E[n']) \in (-1,1]$,
$\phi_{(\beta',\omega')}(E[n]) \in (-1+n-n',1+n-n']$.
Hence $n-n'=0$ or $n-n'=-2$.

If $\phi_{(\beta,\omega)}(E[n]) \in (-1,0]$, then $n-n'=0,2$. 
c\end{rem}
\end{NB}

\begin{NB}
\begin{lem}
Assume that $E=\oplus_i E_i$ such that
$E_i$ are $\sigma_{(\beta,\omega)}$-stable objects
of ${\frak A}_{(\beta,\omega)}$
with $\phi_{(\beta,\omega)}(E_i)=\phi_{(\beta,\omega)}(E)$.
For a Fourier-Mukai transform $\Phi:{\bf D}(X) \to {\bf D}(Y)$,
we set $F_i=\Phi(E_i)$.
Let $(\beta',\omega')$ be the image of $(\beta,\omega)$
by $\Phi$.
Then $\phi_{(\beta',\omega')}(F_i)=\phi_{(\beta',\omega')}(F)$.
\end{lem}

\begin{proof}
We express $\Phi$ as
$\Phi=\Phi_{X \to Y}^{{\bf G}^{\vee}}$.
We may assume that 
${\bf G}_{|X \times \{ y \}} \in {\frak A}_{(\beta,\omega)}$.
If $\phi_{(\beta',\omega')}({\bf G}_{|X \times \{ y \}})>
\phi_{(\beta,\omega)}(E_i)$, then
$\Hom({\bf G}_{|X \times \{ y \}},E_i)=0$ for all $y$.
Hence $H^k(F_i)=0$ except for $k \ne 1,2$.
On the other hand, our assumption
implies that $Z_{(\beta',\omega')}(F_i)=Z_{(\beta',\omega')}(F)$.
By \cite[Prop. 4.2.2]{MYY:2011:2}, we get the claim.

If $\phi_{(\beta',\omega')}({\bf G}_{|X \times \{ y \}})<
\phi_{(\beta,\omega)}(E_i)$, then
$\Hom({\bf G}_{|X \times \{ y \}},E_i[2])=0$ for all $y$.
Hence $H^k(F_i)=0$ except for $k \ne 0, 1$.
Then by the same argument, we get the claim.

If $\phi_{(\beta,\omega')}({\bf G}_{|X \times \{ y \}})=
\phi_{(\beta,\omega)}(E_i)$, then
we see that $F_i[1] \in \Coh(Y)$ or $F_i={\frak k}_y$.
By the same argument, we get the claim.
\end{proof}
\end{NB}

\begin{defn}\label{defn:wall}
Let $v_1 \not \in {\Bbb Q}v$ 
be a Mukai vector with
$\langle v_1^2 \rangle \geq 0$, $\langle (v-v_1)^2 \rangle \geq 0$
and $\langle v_1,v-v_1 \rangle >0$.
We define a \emph{wall for $v$} by
$$
W_{v_1}:=\{(\beta,\omega) \in \NS(X)_{\Bbb R} \times \Amp(X)_{\Bbb R}
\mid {\Bbb R}Z_{(\beta,\omega)}(v_1)=
{\Bbb R}Z_{(\beta,\omega)}(v) \}.
$$
A connected component of
$\NS(X)_{\Bbb R} \times \Amp(X)_{\Bbb R} \setminus \cup_{v_1} W_{v_1}$
is called a \emph{chamber for $v$}.
\begin{NB}
The equation of $W_{v_1}$ is
$$
\langle e^{\beta+\sqrt{-1}\omega},v \rangle
\langle e^{\beta-\sqrt{-1}\omega},v_1 \rangle
=\langle e^{\beta-\sqrt{-1}\omega},v \rangle
\langle e^{\beta+\sqrt{-1}\omega},v_1 \rangle.
$$
\end{NB}
\end{defn}

Replacing $v$ by $-v$ if necessary, we may assume that
$Z_{(\beta,\omega)}(v) \in {\Bbb H} \cup {\Bbb R}_{<0}$.
Assume that $(\beta,\omega) \in W_{v_1}$.
By \cite[Prop. 4.2.2]{MYY:2011:2},
there are $\sigma_{(\beta,\omega)}$-semi-stable objects
$E_1$ and $E_2$ with $v(E_1)=v_1$ and $v(E_2)=v-v_1$.
Then ${\Bbb R}_{>0}Z_{(\beta,\omega)}(E_1)=
{\Bbb R}_{>0}Z_{(\beta,\omega)}(E_2)$.
Thus there is a properly $\sigma_{(\beta,\omega)}$-semi-stable
object with the Mukai vector $v$.

\begin{prop}[{\cite[Prop. 9.3]{Br:3}}]\label{prop:open}
Let $v$ be a Mukai vector with $\langle v^2 \rangle>0$.
Let ${\cal C}$ be a chamber for $v$.
Then ${\cal M}_{(\beta,\omega)}(v)$ 
is independent of $(\beta,\omega) \in {\cal C}$. 
\end{prop}

\begin{prop}\label{prop:wall:isometry}
Let $\xi:H^*(X,{\Bbb Z})_{\alg} \to H^*(Y,{\Bbb Z})_{\alg}$
be an isometry of Mukai lattices.
Let $W_u$ be a wall for $v$ defined by $u$.
Then $\xi(u)$ defines a wall for $\xi(v)$.
\end{prop}

\begin{proof}
We set $w:=v-u$.
Then $u$ defines a wall for $v$ if and only if
$\langle u^2 \rangle, \langle w^2 \rangle \geq 0$ and
$\langle u,w \rangle>0$.
This condition is preserved under $\xi$.
So the claim holds.
\end{proof}

\begin{NB}
If $\xi$ an orientation preserving isometry, then
the claim will follows from \cite{Br:3}. 
\end{NB}

\begin{NB}
Let $\overline{\NS(X)_{\Bbb R} \times {\Bbb R}_{>0}H}
=\NS(X)_{\Bbb R} \times {\Bbb R}_{>0}H \cup \{\infty\}$
be the one point compactification of
$\NS(X)_{\Bbb R} \times {\Bbb R}_{>0}H$. 
Then we may regard that 
$v({\frak k}_x)$ corresponds to
$(\beta,\omega)=\infty$.
Since 
$(r_1 d_0(v)-r d_0(v_1))-\frac{2}{(\omega^2)}
(a_\beta(v_1) d_0(v)-a_\beta(v) d_0(v_1)) \to
(r_1 d_0(v)-r d_0(v_1)),\; (\omega^2) \to \infty$,
$\sigma_{\infty}$-semi-stable objects are 
generated by $\mu$-semi-stable sheaves and 
$0$-dimensional sheaves.

 \end{NB}

Let $v$ be a primitive Mukai vector with $\rk v>0$
and assume that $M_H^\beta(v)$ consists of $\beta$-twisted stable
sheaves (in the sense of Gieseker).

\begin{lem}
Assume that $\langle v^2 \rangle>0$ and 
$r:=-\langle v,\varrho_X \rangle \ne 0$.
Let $v_1$ be a Mukai vector such that
$r c_1(v_1)-r_1 c_1(v)=0$,
where we set $r_1 := -\langle v_1,\varrho_X \rangle$.
Then ${\Bbb R}Z_{(\beta,\omega)}(v)={\Bbb R}Z_{(\beta,\omega)}(v_1)$
if and only if $d_\beta(v)=0$.
In particular, if $v_1$ defines a wall, then
the defining equation of the wall $W_{v_1}$ for $v$ 
is $d_\beta(v)=0$.
\end{lem}

\begin{proof}
We have $r v_1-r_1 v=b \varrho_X$, $0 \ne b \in {\Bbb Z}$.
Hence $rZ_{(\beta,\omega)}(v_1)-r_1 Z_{(\beta,\omega)}(v)=
-b \in {\Bbb R}$.
Then we have
 $Z_{(\beta,\omega)}(v_1)=\frac{r_1}{r}Z_{(\beta,\omega)}(v)-\frac{b}{r}$.
Hence the condition is $d_\beta(v)=0$.
\end{proof}

\begin{NB}
\begin{lem}
Assume that $\langle v^2 \rangle>0$ and $r \ne 0$.
Let $v_1$ be a Mukai vector such that
$(r c_1(v_1)-r_1 c_1(v),H)=0$.
If $v_1$ defines a wall, then
the equation is
$d_\beta(v)(r_1 a_\beta(v)-r a_\beta(v_1))=0$.
\end{lem}

\begin{proof}
Since $(r c_1(v_1)-r_1 c_1(v),H)=0$, we have
$(c_1(v_1)-r_1 \beta,H)=\frac{r_1}{r}(c_1(v)-r \beta,H)$.
Thus $d_\beta(v_1)=\frac{r_1}{r}d_\beta(v)$.
Hence we get 
$d_\beta(v)(r_1 a_\beta(v)-r a_\beta(v_1))=0$.
\end{proof}
\end{NB}

\begin{NB}
We set $\beta:=\beta_0+\eta$. Then
$r_1 a_\beta(v)-r a_\beta(v_1)=(r_1 a_{\beta_0}(v)-r a_{\beta_0}(v_1))
-(r_1 c_1(v)-r c_1(v_1),\eta)$.
Since $d_\beta(v)=d_{\beta_0}(v)-r(\eta,H)/(H^2)$,
the set consists of two lines.
If $\eta \in {\Bbb R}H$, then 
$r_1 a_\beta(v)-r a_\beta(v_1)=r_1 a_{\beta_0}(v)-r a_{\beta_0}(v_1)$
is constant.
\end{NB}

\begin{NB}
\begin{lem}
$d_\beta(v)-r s=0$ is not a wall if and only if
every $\mu$-semi-stable sheaf $E$ with $v(E)=v$ is a
$\mu$-stable vector bundle.
\end{lem}

\begin{proof}
We note that 
every $G \in {\frak A}_{(\beta+sH,\omega)}$
with $d_{\beta+sH}(G)=0$ is generated by
${\frak k}_x$ ($x \in X$) and $F[1]$,
where $F$ is a $\mu$-stable vector bundle
with $d_{\beta+sH}(F)=0$.
If $d_{\beta+sH}(v)=d_\beta(v)-rs=0$ 
is a wall, then
we have a non-trivial decomposition
$-v=n_1 v({\frak k}_x)+\sum_{i=1}^{n_2} v(F_i[1])$ ($n_1, n_2 \geq 0$),
where $F_i$ are $\mu$-stable locally free sheaves with
$d_{\beta+sH}(F_i)=0$.
Then we see that
there is a non-locally free sheaf $F$ with 
$v(F)=v$
or there is a properly $\mu$-semi-stable sheaf $F$ with $v(F)=v$.
 
Conversely if 
there is a non-locally free sheaf $F$ with 
$v(F)=v$
or there is a properly $\mu$-semi-stable sheaf $F$ with $v(F)=v$, then
there is a non-trivial decomposition
$v(E)=n_1 v({\frak k}_x)+\sum_{i=1}^{n_2} v(F_i[1])$ ($n_1, n_2 \geq 0$).
Then $d_\beta-rs$ defines a wall.
\end{proof}
\end{NB}

\section{Fourier-Mukai transforms.}\label{sect:FM}

\subsection{Stability conditions and Fourier-Mukai transforms.}

We shall recall the (twisted) Fourier-Mukai transform
of $\sigma_{(\beta,\omega)}$ and its relation to 
Bridgeland's stability conditions 
explained in \cite{MYY:2011:2}.
Let $\Phi:{\bf D}(X) \to {\bf D}^{\alpha_1}(X_1)$ be a
twisted Fourier-Mukai transform such that
$\Phi(r_1 e^\gamma)=-\varrho_{X_1}$ and
$\Phi(\varrho_X)=-r_1 e^{\gamma'}$,
where $\alpha_1$ is a representative of a suitable
Brauer class.
Then we can describe the cohomological Fourier-Mukai transform
as 
\begin{equation*}
\Phi(r e^\gamma+a \varrho_X+\xi+(\xi,\gamma)\varrho_X)
=-\frac{r}{r_1} \varrho_{X_1}-r_1 a e^{\gamma'}+
\frac{r_1}{|r_1|}
( \widehat{\xi}+(\widehat{\xi},\gamma')\varrho_{X_1}),
\end{equation*}
where $\xi \in \NS(X)_{\Bbb Q}$ and $
\widehat{\xi}:=
\frac{r_1}{|r_1|} 
c_1(\Phi(\xi+(\xi,\gamma)\varrho_X)) \in \NS(X_1)_{\Bbb Q}$.
\begin{rem}
By taking a locally free $\alpha_1$-twisted stable
sheaf $G$ with $\chi(G,G)=0$,
we have a notion of Mukai vector, thus, we have
a map (\cite[Rem. 1.2.10]{MYY:2011:1}):
$$
v_G:{\bf D}^{\alpha_1}(X_1) \to H^*(X_1,{\Bbb Q})_{\alg}.
$$
\end{rem}

We set 
\begin{equation}\label{eq:tilde(beta)}
\begin{split}
\widetilde{\omega}:= & -\frac{1}{|r_1|}
\frac{\frac{((\beta-\gamma)^2)-(\omega^2)}{2}\widehat{\omega}-
(\beta-\gamma,\omega)(\widehat{\beta}-\widehat{\gamma})}
{\left(\frac{((\beta-\gamma)^2)-(\omega^2)}{2} \right)^2
+(\beta-\gamma,\omega)^2},\\
\widetilde{\beta}:= & \gamma'-\frac{1}{|r_1|}
\frac{\frac{((\beta-\gamma)^2)-(\omega^2)}{2}(\widehat{\beta}-\widehat{\gamma})
-(\beta-\gamma,\omega) \widehat{\omega}}
{\left(\frac{((\beta-\gamma)^2)-(\omega^2)}{2} \right)^2
+(\beta-\gamma,\omega)^2}.
\end{split}
\end{equation}

\begin{NB}
We set
\begin{equation}
\begin{split}
\xi:=& 
\frac{1}{|r_1|}
\frac{2}{((\lambda-s)^2+t^2) (H^2)}(\lambda-s) H,\\
\eta:=& 
\frac{1}{|r_1|}
\frac{2}{((\lambda-s)^2+t^2) (H^2)}t H.\\
\end{split}
\end{equation}
\end{NB}

By \cite[sect. 5.1]{MYY:2011:2},
we get the following commutative diagram:
\begin{equation}
\xymatrix{
   {\bf D}(X) \ar[r] \ar[d]_{Z_{(\beta,\omega)}}
 & {\bf D}^{\alpha_1}(X_1) \ar[d]^{Z_{(\wt{\beta},\wt{\omega})}} \\
   {\Bbb C}  \ar[r]_{\zeta^{-1}} 
 & {\Bbb C}
}
\end{equation}
where 
$$
\zeta=-r_1 \left(
\frac{((\gamma-\beta)^2)-(\omega^2)}{2}
+\sqrt{-1}(\beta-\gamma,\omega) \right).
$$

\begin{NB}
If we fix a base point $\beta'$
and write $\beta'=\gamma'-\lambda' \widehat{H}$,
then $\beta'+s' \widehat{H}=\gamma'+(s'-\lambda')\widehat{H}$
and the relation of
$Z_{(\beta+sH,tH)}$ and $Z_{(\beta'+s' \widehat{H},t' \widehat{H})}$
is 
$$
s'-\lambda'=\frac{1}{|r_1|}
\frac{2(\lambda-s)}{((\lambda-s)^2+t^2) (H^2)},\;
t'=\frac{1}{|r_1|}
\frac{2t}{((\lambda-s)^2+t^2) (H^2)}.
$$

\end{NB}

Let ${\bf E}$ be a complex such that
$\Phi=\Phi_{X \to X_1}^{{\bf E}^{\vee}[1]}$.
Then
\begin{equation}\label{eq:Phi(phase)}
\phi_{(\widetilde{\beta},\widetilde{\omega})}(\Phi(E))=
\phi_{(\beta,\omega)}(E)-
\phi_{(\beta,\omega)}({\bf E}_{|X \times \{ x_1 \}}).
\end{equation}

\begin{NB}
\begin{proof}
We note that
$$
\phi_{(\widetilde{\beta},\widetilde{\omega})}(\Phi(E)) \equiv 
\phi_{(\beta,\omega)}(E)+m \mod 2.
$$
For a $\sigma_{(\beta,\omega)}$-semi-stable
object $E \in {\bf D}(X)$ with 
$\phi_{(\beta,\omega)}({\bf E}_{|X \times \{ x_1 \}})+1>
\phi_{(\beta,\omega)}(E)>\phi_{(\beta,\omega)}({\bf E}_{|X \times \{ x_1 \}})$,
we see that $H^i(\Phi(E))=0$ for $i \ne -1,0$.
Hence $\Phi(E) \in {\frak A}_{(\widetilde{\beta},\widetilde{\omega})}[1]
\cup {\frak A}_{(\widetilde{\beta},\widetilde{\omega})} \cup 
{\frak A}_{(\widetilde{\beta},\widetilde{\omega})}[-1]$.
In particular, 
$\phi_{(\widetilde{\beta},\widetilde{\omega})}(\Phi(E)) \in (-1,2)$.
Since $\Phi({\bf E}_{|X \times \{ x_1 \}})={\frak k}_{x_1}[-1]$,
$\phi_{(\widetilde{\beta},\widetilde{\omega})}
(\Phi({\bf E}_{|X \times \{ x_1 \}}))=0$.
Then 
\begin{equation}
\phi_{(\widetilde{\beta},\widetilde{\omega})}(\Phi(E))=
\phi_{(\widetilde{\beta},\widetilde{\omega})}(\Phi(E))
-\phi_{(\widetilde{\beta},\widetilde{\omega})}
(\Phi({\bf E}_{|X \times \{ x_1 \}}))
\equiv \phi_{(\beta,\omega)}(E)-
\phi_{(\beta,\omega)}({\bf E}_{|X \times \{ x_1 \}}) \mod 2{\Bbb Z}.
\end{equation}
Hence $\phi_{(\widetilde{\beta},\widetilde{\omega})}(\Phi(E)) \in (2k,2k+1)$.
Therefore $k=0$ and 
\begin{equation}
\phi_{(\widetilde{\beta},\widetilde{\omega})}(\Phi(E))=
\phi_{(\beta,\omega)}(E)-
\phi_{(\beta,\omega)}({\bf E}_{|X \times \{ x_1 \}}).
\end{equation}
\end{proof}
\end{NB}

\begin{NB}
Assume that 
$\langle v(E),v({\bf E}_{|X \times \{ x_1 \}})
\rangle \ne 0$. 
There is $(\beta,\omega)$
such that ${\Bbb R}Z_{(\beta,\omega)}(v(E))=
{\Bbb R}Z_{(\beta,\omega)}({\bf E}_{|X \times \{ x_1 \}})$.
Assume that $E \in {\frak A}_{(\beta,\omega)}$.
If ${\bf E}_{|X \times \{ x_1 \}}[n] \in {\frak A}_{(\beta,\omega)}$, then
$\rk \Phi(E[1-n])>0$.
In the region 
$\phi_{(\beta,\omega)}(E)>
\phi_{(\beta,\omega)}({\bf E}_{|X \times \{ x_1 \}})$,
$1>\phi_{(\widetilde{\beta},\widetilde{\omega})}(\Phi(E[1-n]))>0$.
In particular, $d_{\widetilde{\beta}}(\Phi(E[1-n]))>0$.
\begin{NB2}
$d_{\widetilde{\beta}}(\Phi(E[1-n]))>0$ and $H^i(\Phi(E[1-n]))=0$
for $i \ne -1,0$ is enough to show
$1>\phi_{(\widetilde{\beta},\widetilde{\omega})}(\Phi(E[1-n]))>0$.
Indeed $d_{\widetilde{\beta}}(\Phi(E[1-n]))>0$ 
implies that 
$\phi_{(\widetilde{\beta},\widetilde{\omega})}(\Phi(E[1-n])) \in
(2k,2k+1)$.
Then $H^i(\Phi(E[1-n]))=0$ for $i \ne -2k-1,-2k$.
Therefore $k=0$, which implies the claim.
\end{NB2}

Assume that 
$E[1] \in {\frak A}_{(\beta,\omega)}$.
If ${\bf E}_{|X \times \{ x_1 \}}[n] \in {\frak A}_{(\beta,\omega)}$, then
$\rk \Phi(E[2-n])>0$.
In the region 
$\phi_{(\beta,\omega)}(E[1])>
\phi_{(\beta,\omega)}({\bf E}_{|X \times \{ x_1 \}})$,
$1>\phi_{(\widetilde{\beta},\widetilde{\omega})}(\Phi(E[2-n]))>0$.

If $\langle v(E),v({\bf E}_{|X \times \{ x_1 \}})
\rangle= 0$. Then there is $(\beta,\omega)$ with
 ${\Bbb R}Z_{(\beta,\omega)}(v(E))=
{\Bbb R}Z_{(\beta,\omega)}({\bf E}_{|X \times \{ x_1 \}})$
if and only if $v(E) \in {\Bbb Q} v({\bf E}_{|X \times \{ x_1 \}})$.

For ${\bf E}_{|X \times \{ x_1 \}}$,
$\Phi({\bf E}_{|X \times \{ x_1 \}})={\frak k}_{x_1}[-2]$.
Hence $\phi_{(\widetilde{\beta},\widetilde{\omega})}
(\Phi({\bf E}_{|X \times \{ x_1 \}}))=-1$.

If $\phi_{(\beta,\omega)}(E)>
\phi_{(\beta,\omega)}({\bf E}_{|X \times \{ x_1 \}})$,
then $H^i(\Phi(E))=0$ for $i \geq 2$ implies that
$\phi_{(\widetilde{\beta},\widetilde{\omega})}(\Phi(E)) > -2$.
If 
$\phi_{(\beta,\omega)}({\bf E}_{|X \times \{ x_1 \}})+1
> \phi_{(\beta,\omega)}(E)$,
then we have $H^i(\Phi(E))=0$ for $i < 0$,
which implies that
$\phi_{(\widetilde{\beta},\widetilde{\omega})}(\Phi(E)) \leq 1$.

\end{NB}

\begin{thm}(cf. \cite[Thm. 2.2.1]{MYY:2011:2})\label{thm:MYY2}
Assume that ${\Bbb R}Z_{(\beta,\omega)}(v)=
{\Bbb R}Z_{(\beta,\omega)}(e^{\gamma})$.
If $r_1 e^\gamma$ does not define a wall, then 
$\Phi$ induces an isomorphism
$$
{\cal M}_{(\beta,\omega)}(v) \to {\cal M}_{\widetilde{\omega}}(u)^{ss},
$$
where $u:=\Phi(v)$.
\end{thm}

As an application of Theorem \ref{thm:MYY2},
we give a different proof of Proposition \ref{prop:open}.

\begin{proof}
\begin{NB}
We take $(\beta_1,\omega_1) \in {\cal C}$.
We set $\beta_1:=b_1 H+\eta$, $\eta \in H^{\perp}$.
We can take a curve $(\beta_1+sH,t(s)H)$, $-\epsilon \leq s \leq \epsilon$
such that $t(0)H$ is sufficiently close to $\omega_1$ and
${\Bbb R}Z_{(\beta_1+sH,t(s)H)}(v)=
{\Bbb Z}_{(\beta_1+sH,t(s)H)}(e^{\beta+\lambda H})$ for some
$\lambda \in {\Bbb Q}$.
For a neighborhood $(-\epsilon,\epsilon) \times U$ 
of $(\beta_1,\omega_1)$ such that
$(-\epsilon,\epsilon) \times U \subset {\cal C}$,
\cite[Cor. 4.2.5]{MYY:2011:2} implies that
$M_{(\beta,\omega)}(v)$ is independent of $(\beta,\omega) \in 
U \subset {\cal C}$.
\end{NB}
For $H \in \Amp(X)_{\Bbb Q}$,
we take $\delta \in \NS(X)_{\Bbb Q}$ such that
$(\delta,H)=1$.
We set $H(\delta):=\{L \in \Amp(X)_{\Bbb R} \mid (\delta,L)=1 \}$. Then  
$\{tL \in \Amp(X)_{\Bbb R} \mid L \in H(\delta),\; t \in {\Bbb R}_{>0} \}$
is an open neighborhood of $H$.
Assume that $(\beta_0,\omega_0) \in {\cal C}$ and
$\omega_0 \in {\Bbb R}_{>0} H$.
We shall show that $\sigma_{(\beta,\omega)}$-semi-stability
is independent of $(\beta,\omega)$ in a neighborhood of
$(\beta_0,\omega_0)$.

We can take $\gamma \in {\Bbb R}H+\beta_0 \subset \NS(X)_{\Bbb R}$   
such that
${\Bbb R} Z_{(\beta_0,\omega_0)}(v)=
{\Bbb R} Z_{(\beta_0,\omega_0)}(e^\gamma)$,
$(\gamma-\beta_0,H)>0$
and $(r \gamma-c_1(v),H) \ne 0$
(\cite[sect. 4.1]{MYY:2011:2}).
Replacing $\omega_0$ if necessary, we may assume that
$\gamma \in \NS(X)_{\Bbb Q}$.
We take a neighborhood $U$ of $H(\delta)$ and $I \subset {\Bbb R}_{>0}$
such that 
$\omega_0 \in
\{tL \mid  L \in U,\; t \in I \} \subset {\cal C}$.
For a fixed $L \in H(\delta)_{\Bbb Q}$, 
the semi-stability is independent of $(\beta,L) \in {\cal C}$.

Replacing $U$ if necessary, we may assume that the following equation for
$t>0$ has a solution for each $L \in U$ and a neighborhood $V$ of
$\beta_0$: 
$$
t^2 \frac{(L^2)}{2}=
\frac{(\langle e^\gamma,e^\beta \rangle(c_1(v)-r \beta)-
\langle v,e^\beta \rangle(\gamma-\beta),L)}{
(r \gamma- c_1(v),L)}.
$$
We set $\omega:=tL$. $\omega$ is a function on
$V \times U$ and 
we have ${\Bbb R}Z_{(\beta,\omega)}(v)
={\Bbb R}Z_{(\beta,\omega)}(e^\gamma)$.
\begin{NB}
If $(\gamma-\beta,L)=0$, then
$$
\frac{(\langle e^\gamma,e^\beta \rangle(c_1(v)-r \beta)-
\langle v,e^\beta \rangle(\gamma-\beta),L)}{
(r \gamma- c_1(v),L)}=
\frac{((\beta-\gamma)^2)}{2} \leq 0. 
$$
Moreover if the equality holds, then
$\beta=\gamma$.
\end{NB}
For the proof of our claim, it is sufficient to
show the independence of $\sigma_{(\beta,\omega)}$-semi-stability, 
where $(\beta,L) \in V \times U$.

We set $X_1:=M_{H}(r_1 e^\gamma)$.
For a universal family ${\bf E}$ on $X \times X_1$
as a twisted object,
we consider the Fourier-Mukai transform
$\Phi_{X \to X_1}^{{\bf E}^{\vee}[1]}$.
Then we have an isomorphism
${\cal M}_{(\beta,\omega)}(v) \to {\cal M}_{\widetilde{\omega}}(u)^{ss}$,
where $u:=\Phi_{X \to X_1}^{{\bf E}^{\vee}[1]}(v)$.
\begin{NB}
$$
\widetilde{L}:=-\frac{1}{r_1}
\frac{\frac{((\beta-\gamma)^2)-t^2(L^2)}{2}t \widehat{L}-
(\beta-\gamma,tL)(\widehat{\beta}-\widehat{\gamma})}
{\left(\frac{((\beta-\gamma)^2)-t^2(L^2)}{2} \right)^2
+t^2(\beta-\gamma,L)^2}.
$$
We set
$$
\widetilde{\beta}:=\gamma'-\frac{1}{r_1}
\frac{\frac{((\beta-\gamma)^2)-t^2(L^2)}{2}(\widehat{\beta}-\widehat{\gamma})
-(\beta-\gamma,tL)t \widehat{L}}
{\left(\frac{((\beta-\gamma)^2)-t^2(L^2)}{2} \right)^2
+t^2(\beta-\gamma,L)^2}.
$$
\end{NB}
Since $(c_1(u e^{-\widetilde{\beta}}),\widetilde{L})=0$
for all $(\beta,L) \in V \times U$,
if $F \in {\cal M}_{\widetilde{L}}(u)^{ss}$ contains
a subsheaf $F_1$ with 
$(c_1(F_1(-\widetilde{\beta})),\widetilde{\omega})=0$
and $v(F_1) \not \in {\Bbb Q}u$,
then $\Phi_{X_1 \to X}^{{\bf E}[1]}(F_1)$ defines a wall
for $v$.
\begin{NB}
${\Bbb R}Z_{(\beta,tL)}(\Phi_{X_1 \to X}^{{\bf E}[1]}(F_1))=
{\Bbb R}Z_{(\beta,tL)}(e^\gamma)={\Bbb R}Z_{(\beta,tL)}(v)$.
We set $u_1:=v(F_1)$ and $u_2:=u-u_1$.
Then $\langle u_1,u_2 \rangle \geq 
(\langle u_1^2 \rangle+\langle u_2^2 \rangle)/2 \geq 0$.
If the equality holds, then 
$\langle u_1 \rangle =\langle u_2^2 \rangle=0$ and
$\frac{c_1(u_1)}{\rk u_1}=\frac{c_1(u_2)}{\rk u_2}$.
Then $u \in {\Bbb Q}u_1$ and $u$ is an isotropic Mukai vector.
\end{NB}
Therefore   
for any subsheaf $F_1$ of 
$F \in {\cal M}_{\widetilde{\omega}}(u)^{ss}$,
$(c_1(F_1(-\widetilde{\beta})),\widetilde{\omega}) \ne 0$
for any $(\beta,L) \in V \times U$ or 
we have $v(F_1) \in {\Bbb Q}u$.
For $F \in {\cal M}_{\widetilde{\omega_1}}(u)^{ss} \setminus
{\cal M}_{\widetilde{\omega_2}}(u)^{ss}$,
there is a subsheaf $F_1$ such that
$(c_1(F_1(-\widetilde{\beta})),\widetilde{\omega_2}) \geq 0$
and
$(c_1(F_1(-\widetilde{\beta})),\widetilde{\omega_1}) \leq 0$.
Since $(c_1(F_1(-\widetilde{\beta})),\widetilde{\omega})$
is a continuous function on $V \times U$, it is a contradiction.
Therefore ${\cal M}_{\widetilde{\omega}}(u)^{ss}$
is independent of $(\beta,L)$.
Then we see that ${\cal M}_{(\beta,\omega)}(v)$ is independent of
$(\beta,\omega)$. 
\end{proof}

\begin{defn}
Let $\sigma_{(\beta,\omega)}$ be a stability condition.
For the contravariant 
Fourier-Mukai transform
$\Phi \circ {\cal D}_X$, we set $(\beta',\omega'):=(-\beta,\omega)$
and attach the stability condition 
$\sigma_{(\widetilde{\beta'},\widetilde{\omega'})}$ associated to
$Z_{(\widetilde{\beta'},\widetilde{\omega'})}$.
We say $\sigma_{(\widetilde{\beta'},\widetilde{\omega'})}$
the stability condition induced by $\Phi \circ {\cal D}_X$.
\end{defn}

\begin{lem}\label{lem:dual-phase}
$\phi_{(-\beta,\omega)}(E^{\vee}[1])=-\phi_{(\beta,\omega)}(E)+1$.
\end{lem}

\begin{proof}
For a non-zero object $E \in {\bf D}(X)$, we have
\begin{equation}
Z_{(-\beta,\omega)}(E^{\vee}[1])=
-\langle e^{-\beta+\omega \sqrt{-1}},v(E^{\vee}) \rangle
=-\overline{\langle e^{\beta+\omega \sqrt{-1}},v(E) \rangle}
=|Z_{(\beta,\omega)}(v)|e^{\pi \sqrt{-1} (1-\phi_{(\beta,\omega)}(E))}.
\end{equation}
Hence $\phi_{(-\beta,\omega)}(E^{\vee}[1])=-\phi_{(\beta,\omega)}(E)+1
\mod 2{\Bbb Z}$.
Since $\phi_{(-\beta,\omega)}((E[n])^{\vee}[1])=
\phi_{(-\beta,\omega)}(E^{\vee}[1])-n$, we shall show that
$\phi_{(-\beta,\omega)}(E^{\vee}[1]) \in [0,1)$
for $E \in {\frak A}_{(\beta,\omega)}$.

We note that
${\frak A}_{(\beta,\omega)}$ is generated by
(i) 0-dimensional object $T$, (ii)
$F[1]$ where $F$ is a locally free $\mu$-semi-stable sheaf
with $d_\beta(F) \leq 0$, (iii)
$\mu$-semi-stable sheaf $E$ with $d_\beta(E)>0$ and
(iv)
purely 1-dimensional sheaf $E$.
 
(i) For a 0-dimensional sheaf $T$, 
$T^{\vee}[1] \in {\frak A}_{(-\beta,\omega)}[-1]$.
Thus $\phi_{(-\beta,\omega)}(T^{\vee}[1])=0$.
(ii) For a locally free $\mu$-semi-stable 
sheaf $F$ with $d_\beta(F) \leq 0$,
$(F[1])^{\vee}[1]=F^{\vee}$ is a $\mu$-semi-stable sheaf with
$d_{-\beta}((F[1])^{\vee}[1]) \geq 0$.
Hence $\phi_{(-\beta,\omega)}((F[1])^{\vee}[1]) \in [0,1)$.
(iii) Let  $E$ be a $\mu$-semi-stable sheaf of
$\rk E>0$ and $d_\beta(E)>0$.
Let $E^{**}$ be the reflexive hull of $E$ and set $T:=E^{**}/E$.
Then we have an exact triangle
$$
T^{\vee}[1] \to (E^{**})^{\vee}[1] \to E^{\vee}[1] \to 
T^{\vee}[2].
$$
Since $(E^{**})^{\vee}$ is a locally free $\mu$-semi-stable sheaf
with $d_{-\beta}((E^{**})^{\vee})=-d_\beta(E)<0$ and
$T^{\vee}[2]$ is a 0-dimensional sheaf,
$E^{\vee}[1] \in {\frak A}_{(-\beta,\omega)}$.
(iv) If $E$ is a purely 1-dimensional sheaf, then
$E^{\vee}[1]$ is a purely 1-dimensional sheaf, which implies
$E^{\vee}[1] \in {\frak A}_{(-\beta,\omega)}$. 
Therefore the claim holds.
\end{proof}

\subsection{Relations of moduli spaces under Fourier-Mukai
transforms.}

We shall say that a pair 
$(\beta,\omega)$ is \emph{general with respect to $v$}
if $(\beta,\omega)$ is not on any wall $W_{v_1}$ for $v$.

\begin{thm}\label{thm:B:3-10.3}
Let $v$ be a Mukai vector with
$\langle v^2 \rangle>0$.
Assume that $(\beta,\omega)$ is general with respect to
$v$.
\begin{enumerate}
\item[(1)]
Any Fourier-Mukai transform
$\Phi_{X \to X'}^{{\bf E}}$
preserves Bridgeland's stability condition.
\item[(2)]
Any contravariant Fourier-Mukai transform
$\Phi_{X \to X'}^{{\bf E}} \circ {\cal D}_X$
preserves Bridgeland's stability condition.
\end{enumerate}
\end{thm}
\begin{NB}
The same claims in this section also hold for 
the Bridgeland stability of twisted sheaves. 
\end{NB}
The first claim is due to Bridgeland (\cite[Prop. 10.3]{Br:3}).
For the proof of (2),
it is sufficient to prove the following claim.

\begin{prop}
Let $v$ be a Mukai vector with
$\langle v^2 \rangle>0$.
Assume that $(\beta,\omega)$ 
is general with respect to $v$.
If $Z_{(\beta,\omega)}(v) \in {\Bbb H} \cup {\Bbb R}_{<0}$, then
we have an isomorphism
\begin{equation*}
\begin{matrix}
M_{(\beta,\omega)}(v) & \to & M_{(-\beta,\omega)}(-v^{\vee})\\
E & \mapsto & E^{\vee}[1].
\end{matrix}
\end{equation*}
\begin{NB}
If $Z_{(\beta,\omega)}(v) \in {\Bbb R}_{<0}$,
then 
we have an isomorphism
\begin{equation}
\begin{matrix}
M_{(\beta,\omega)}(v) & \to & M_{(-\beta,\omega)}(v^{\vee})\\
E & \mapsto & E^{\vee}[2].
\end{matrix}
\end{equation}
Since $Z_{(-\beta,\omega)}(-v^{\vee}) \in {\Bbb R}_{>0}$,
$\phi_{(-\beta,\omega)}(-v^{\vee})=0$.
By Definition \ref{defn:moduli}, $M_{(-\beta,\omega)}(-v^{\vee})=
\{F[-1]|F:=E^{\vee}[2] \in M_{(-\beta,\omega)}(v^{\vee}) \}$.
\end{NB}
\end{prop}

\begin{NB}
\begin{equation}
\deg_{\beta}(E)=-\deg_{-\beta}(E^{\vee})
=\deg_{-\beta}(E^{\vee}[1]).
\end{equation}
$$
Z_{(\beta,\omega)}(-v^{\vee})
=a_{-\beta}(v)-\frac{(\omega^2)}{2}r(v)+d_{-\beta}(v)(H,\omega)\sqrt{-1}.
$$
Hence 
$Z_{(\beta,\omega)}(E^{\vee}[1]) \in {\Bbb H} \cup {\Bbb R}_{\geq 0}$
for $E \in {\frak A}_{(-\beta,\omega)}$.
\end{NB}

\begin{NB}
If $(\omega^2) \gg 0$ and $\rk v<0$, then
$M_{(\beta,\omega)}(v)=\{E^{\vee}[1]| E \in M_\omega^{-\beta}(-v^{\vee})\}$
and
$M_{(-\beta,\omega)}(-v^{\vee})=M_\omega^{-\beta}(-v^{\vee})$.
Hence the claim is trivial.
\end{NB}

\begin{proof}
\begin{NB}
We don't need the case $d_\beta(v)=0$, since we can move
$\beta$. However the following is important, since
the general case can be reduced to this special case.
 
We fitst assume that $d_\beta(v)=0$.
In this case, $\sigma_{(\beta,\omega)}$-semi-stable objects are generated by
$E[1]$ and ${\frak k}_x$, where $E$ is a locally free $\mu$-sermi-stable 
sheaves.
Since $E^{\vee}$ and ${\frak k}_x^{\vee}[2]$ are $\mu$-semi-stable
sheaf and a point sheaf respectively,
$(E[1])^{\vee}[2]=E^{\vee}[1]$ and ${\frak k}_x^{\vee}[2]$
are $\sigma_{(\beta,\omega)}$-semi-stable objects of 
${\frak A}_{(-\beta,\omega)}$.
Thus the claim holds.
\begin{NB2}
If $(\beta,\omega)$ is general, then
every $\sigma_{(\beta,\omega)}$-semi-stable objects $E$ with $v(E)=v$
are $\mu$-stable. 
\end{NB2}
\end{NB}
Let ${\cal C}$ be a chamber containing $(\beta,\omega)$.
We may assume that $d_{\beta,\omega}(v)>0$.
Replacing $(\beta,\omega) \in {\cal C}$ if necessary,
we can take $\gamma \in \NS(X)_{\Bbb Q}$ such that
${\Bbb R}Z_{(\beta,\omega)}(e^\gamma)={\Bbb R}Z_{(\beta,\omega)}(v)$.
We take a primitive vector $r_1 e^\gamma$ such that
$d_{\beta,\omega}(r_1 e^\gamma)>0$.
We set $X_1:=M_{(\beta,\omega)}(r_1 e^{\gamma})$.
Let
${\bf E}$ be the universal object on $X \times X_1$
as a complex of twisted sheaves.
We set $w:=\Phi_{X \to X_1}^{{\bf E}^{\vee}[1]}(v)$.
By Theorem \ref{thm:MYY2}, 
$M_{\widetilde{\omega}}^{\alpha_1}(w)$ consists of $\mu$-stable 
locally free sheaves and
we have an isomorphism
$$
\Phi_{X \to X_1}^{{\bf E}^{\vee}[1]}:
M_{(\beta,\omega)}(v) \to M_{\widetilde{\omega}}^{\alpha_1}(w),
$$
where $w$ satisfies
$(c_1(w e^{-\widetilde{\beta}}),\widetilde{\omega})=0$.
By taking the dual, we have an isomorphism
$$
M_{\widetilde{\omega}}^{\alpha_1}(w) \to 
M_{\widetilde{\omega}}^{-\alpha_1}(w^{\vee}).
$$
We note that ${\bf F}:={\bf E}^{\vee}[1]$ is a family of stable objects
with $v({\bf F}_{|X \times \{x_1 \}})=-r_1 e^{-\gamma}$.
For $\Phi_{X \to X_1}^{{\bf F}^{\vee}[1]}$
we define $(\widetilde{-\beta},\widetilde{\omega})$
by \eqref{eq:tilde(beta)}.
Thus we substitute $(-\gamma,-\gamma',-\beta)$
in \eqref{eq:tilde(beta)} 
instead of $(\gamma,\gamma',\beta)$
for the definition of 
$(\widetilde{-\beta},\widetilde{\omega})$.
Then we have $
(\widetilde{-\beta},\widetilde{\omega})=
(-\widetilde{\beta},\widetilde{\omega})$.
Since ${\Bbb R}Z_{(-\beta,\omega)}(-r_1 e^{-\gamma})=
{\Bbb R}Z_{(-\beta,\omega)}(-v^{\vee})$,
we have an isomorphism
\begin{equation}
\Phi_{X \to X_1}^{{\bf F}^{\vee}[1]}:
M_{(-\beta,\omega)}(-v^{\vee}) \to 
M_{\widetilde{\omega}}^{-\alpha_1}(w^{\vee}).
\end{equation}
By the Grothendieck-Serre duality,
we have
$$
\Phi_{X \to X_1}^{{\bf F}^{\vee}[1]}(E^{\vee}[1])=
\Phi_{X \to X_1}^{{\bf E}}(E^{\vee}[1])=
D_{X_1} \circ \Phi_{X \to X_1}^{{\bf E}^{\vee}[1]}(E).
$$
Hence the claim holds.
\end{proof}

\begin{NB}
\begin{proof}
We first recall the main result in \cite{MYY:2011:2}.
If $\lambda r \ne d_\beta$,
then Lemma \ref{lem:C_v}
implies that
the condition
${\Bbb R}Z_{(\beta+sH,\omega)}(v)=
{\Bbb R}Z_{(\beta+sH,\omega)}(e^{\beta+\lambda H})$
defines a circle
\begin{equation}\label{eq:circle-lambda}
C_{v,\lambda}:\;
t^2=
(\lambda-s)
\left( \frac{a_\beta-d_\beta \lambda \frac{(H^2)}{2}}{\lambda r-d_\beta}
\frac{2}{(H^2)}+s \right).
\end{equation}
Assume that $C_{v,\lambda}$ is not contained in any wall
for $v$.
Then $C_{v,\lambda}$ is contained in a chamber ${\cal C}$.
We set $\gamma:=\beta+\lambda H$.
\begin{equation}
\frac{a_\beta-d_\beta \lambda \frac{(H^2)}{2}}{\lambda r-d_\beta}
\frac{2}{(H^2)}+\lambda
=\frac{(\lambda r-d_\beta)^2 (H^2)-
(\langle v^2 \rangle-(D_\beta^2))}
{r(H^2)(\lambda r-d_\beta)}
=\frac{2a_\gamma}{-d_\gamma (H^2)}=
\frac{2\langle e^{\beta+\lambda H},v \rangle}
{d_\gamma (H^2)}
\end{equation}
Then the intersection of $C_{v,\lambda}$ with the
circle
$(s-\lambda)^2+t^2=\frac{2}{|r_1|(H^2)}$
lies on the line
\begin{equation}\label{eq:line}
s=\lambda+\frac{d_\gamma}{a_\gamma |r_1|}.
\end{equation}
We set $X_1:=M_{(\beta+sH,tH)}(r_1 e^{\gamma})$.
Let
${\bf E}$ be the universal object on $X \times X_1$
as a complex of twisted sheaves.
By the Fourier-Mukai transform
$\Phi_{X \to Y}^{{\bf E}^{\vee}[1]}:
{\bf D}(X) \to {\bf D}^{\alpha_1}(X_1)$,
the circle is transformed to
the line \eqref{eq:line}.
\begin{NB2}
$$
s'=\Phi_{X \to Y}^{{\bf E}^{\vee}[1]}(s)=
\lambda+\frac{2}{(H^2)|r_1|}
\frac{s-\lambda}{(s-\lambda)^2+t^2}.
$$
\end{NB2}
Hence the chamber ${\cal C}$ containing 
the circle $C_{v,\lambda}$ is mapped to a unbounded domain.
We set $w:=\Phi_{X \to X_1}^{{\bf E}^{\vee}[1]}(v)$.
By \cite[Thm. 2.2.1]{MYY:2011:2},
we have a moduli of $\mu$-stable vector bundles
$M_{\widehat{H}}^{\alpha_1}(w)$, 
we have an isomorphism
$\Phi_{X \to X_1}^{{\bf E}[1]}:
M_{(\beta+sH,tH)}(v) \to M_{\widehat{H}}^{\alpha'}(w)$
for $(s,t) \in C_{v,\lambda}$.
For $\Phi_{X_1 \to X}^{{\bf E}^{\vee}[2]}$,
we also have an isomorphism
$M_{\widehat{H}}^{-\alpha_1}(w^{\vee}) \to M_{(-\beta+sH,tH)}(-v^{\vee})$,
where $(s,t) \in C_{-v^{\vee},-\lambda}$.

By the Grothendieck-Serre duality,
we have
$$
\Phi_{X \to X_1}^{{\bf E}}(E^{\vee}[1])=
D_{X_1} \circ \Phi_{X \to X_1}^{{\bf E}^{\vee}[1]}(E).
$$
\begin{NB2}
We have an isomorphism
$M_{(\beta+sH,tH)}(r_1 e^{\beta+\lambda H})
\to M_{(-\beta-sH,tH)}(-r_1 e^{-\beta-\lambda H})$
by sending $E$ to $E^{\vee}[1]$.
Thus $M_{(-\beta-sH,tH)}(-r_1 e^{-\beta-\lambda H}) \cong Y$
and
${\bf E}^{\vee}[1]$ is the universal family
on $X \times Y$.
Then $\Phi_{X \to Y}^{({\bf E}^{\vee}[1])^{\vee}[1]}=
\Phi_{X \to Y}^{{\bf E}}$.
\end{NB2}

\begin{NB2}
Assume that $r>0$.
Then 
\begin{equation}
\begin{cases}
d_{\beta+sH}(v)=d_\beta(v)-rs>0,& s<\frac{d_\beta}{r},\\
d_{\beta+sH}(-v)=d_\beta(-v)-(-r)s>0,& s>\frac{d_\beta}{r}.
\end{cases}
\end{equation}
$s=\frac{d_\beta}{r}$ is a wall for $v$ if and only if
there is a properly $\mu$-semi-stable sheaf or
there is a non-locally free $\mu$-semi-stable sheaf.



\end{NB2}

If $\lambda r =d_\beta$,
then the condition
${\Bbb R}Z_{(\beta+sH,\omega)}(v)=
{\Bbb R}Z_{(\beta+sH,\omega)}(e^{\beta+\lambda H})$
defines a line $L_{v,\lambda}$
$$
s=\frac{d_\beta}{r}=\lambda.
$$
\begin{NB2}
If $\frac{d_\beta}{r}=\lambda$, then
$a_\beta-\frac{(H^2)}{2}r \lambda^2 \ne 0$.
Indeed if $a_\beta-\frac{(H^2)}{2}r \lambda^2=0$, then
$v=r e^{\beta+\lambda H}+D_\beta+(D_\beta,\beta+\lambda H)\varrho_X$.
Hence $\langle v^2 \rangle =(D_\beta^2) \leq 0$,
which is a contradiction. 
\end{NB2}
If $L_{v,\lambda}$ is not contained in any wall for $v$,
then the claim follows from
the fact that $M_{(\beta+\lambda H,tH)}(v)$ and
$M_{(-\beta-\lambda H,tH)}(v)$
consists of $\mu$-stable
vector bundles up to shift.   
\end{proof}
\end{NB}

\begin{rem}
By a similar argument, we can prove Theorem
\ref{thm:B:3-10.3} (1):
Let ${\cal C}$ be a chamber containing $(\beta,\omega)$.
We may assume that $d_{\beta,\omega}(v)>0$.
Replacing $(\beta,\omega) \in {\cal C}$ if necessary,
we can take $\gamma \in \NS(X)_{\Bbb Q}$ such that
${\Bbb R}Z_{(\beta,\omega)}(e^\gamma)={\Bbb R}Z_{(\beta,\omega)}(v)$
and $(\beta,\omega) \in {\cal C}$.
We take a primitive vector $r_1 e^\gamma$ such that
$d_{\beta,\omega}(r_1 e^\gamma)>0$.
For any Fourier-Mukai transform $\Phi:{\bf D}(X) \to {\bf D}(Y)$,
let $v_0 \in {\Bbb Q}\Phi(e^{\gamma})$
be a primitive and positive Mukai vector. 
If $r_0:=\rk v_0 \ne 0$, then
$v_0=r_0 e^{\widetilde{\gamma}}$,
$\widetilde{\gamma} \in \NS(Y)$.
We first assume that $r_0 \ne 0$.
We set
\begin{equation}
\Phi(e^{\beta+\sqrt{-1} \omega})=
-\langle e^{\beta+\sqrt{-1} \omega}, \Phi^{-1}(\varrho_Y) \rangle 
e^{\beta'+\sqrt{-1} \omega'}.
\end{equation} 
Then $X_1 \cong M_{\omega'}(v_0)$ and 
the universal object ${\bf G}$ on $X_1 \times Y$
satisfies $\Phi_{X_1 \to Y}^{{\bf G}[n]}=
\Phi \circ \Phi_{X_1 \to X}^{{\bf E}}$, $n \in {\Bbb Z}$.
Since 
${\Bbb R}Z_{(\beta,\omega)}(e^\gamma)
={\Bbb R}Z_{(\beta,\omega)}(v)$, 
we have 
${\Bbb R}Z_{(\beta',\omega')}(v_0)
={\Bbb R}Z_{(\beta',\omega')}(\Phi(v))$.
Then 
$\Phi_{X_1 \to Y}^{{\bf G}[n]}$
induces an isomorphism
$$
M_{H_1}^{\alpha_1}(w) \to 
M_{(\beta',\omega')}(\Phi(v)).
$$
Hence we have an isomorphism
$$
\Phi:M_{(\beta,\omega)}(v) \to 
M_{(\beta',\omega')}(\Phi(v)).
$$
If $r_0=0$, then $\Phi$ is induced by isomorphisms 
on underlying abelian surfaces and
the action of $\Piczero{X}$.
Hence we also have an isomorphism
$M_{(\beta,\omega)}(v) \to 
M_{(\Phi(\beta),\Phi(\omega))}(\Phi(v))$.

If $(\beta,\omega)$ belongs to a wall, then 
$\Phi$ also preserves semi-stability. Indeed
assume that 
$E$ is $S$-equivalent to
$\oplus_i E_i$, where $E_i$ are $\sigma_{(\beta,\omega)}$-stable
objects with $\phi_{(\beta,\omega)}(E_i)=\phi_{(\beta,\omega)}(E)$.
Then $\Phi(E_i)$ are $\sigma_{(\beta',\omega')}$-stable
objects. By \eqref{eq:Phi(phase)},
$\phi_{(\beta',\omega')}(E_i)$ are the same.
\begin{NB}
Then $\Phi(E_i)$ are $\sigma_{(\widetilde{\beta},\widetilde{\omega})}$-stable
objects. By \eqref{eq:Phi(phase)},
$\phi_{(\widetilde{\beta},\widetilde{\omega})}(E_i)$ are the same.
\end{NB}
Therefore the claim holds.

\begin{NB}
Old argument:
For any Fourier-Mukai fransform $\Phi:{\bf D}(X) \to {\bf D}(Y)$,
let $v_0 \in {\Bbb Q}\Phi(e^{\beta+\lambda H})$
be a primitive and positive Mukai vector. 
If $r_0:=\rk v_0 \ne 0$, then
$v_0=r_0 e^{\widetilde{\gamma}}$,
$\widetilde{\gamma} \in \NS(Y)$.
We first assume that $r_0>0$.
We set
\begin{equation}
\Phi(e^{\beta+sH+\sqrt{-1} tH})=
-\langle e^{\beta+sH+\sqrt{-1} tH}, \Phi^{-1}(\varrho_Y) \rangle 
e^{\widetilde{\beta}(s)+\sqrt{-1} \widetilde{\omega}(s)}.
\end{equation} 
Then $X_1 \cong M_{\widetilde{\omega}(s)}(v_0)$ and 
the universal object ${\bf G}$ on $X_1 \times Y$
satisfies $\Phi_{X_1 \to Y}^{{\bf G}[n]}=
\Phi \circ \Phi_{X_1 \to X}^{{\bf E}}$, $n \in {\Bbb Z}$.
For $(s,t) \in C_{v,\lambda}$,
${\Bbb R}Z_{\widetilde{\beta}(s),\widetilde{\omega}(s)}(v_0)
={\Bbb R}Z_{\widetilde{\beta}(s),\widetilde{\omega}(s)}(\Phi(v))$.
Then 
$\Phi_{X_1 \to Y}^{{\bf G}[n]}$
induces an isomorphism
$$
M_{H_1}^{\alpha_1}(w) \to 
M_{(\widetilde{\beta}(s),\widetilde{\omega}(s))}(\Phi(v)).
$$
Hence we have an isomorphism
$$
\Phi:M_{(\beta+s H,t H)}(v) \to 
M_{(\widetilde{\beta}(s),\widetilde{\omega}(s))}(\Phi(v)).
$$
If $r_0=0$, then $\Phi$ is induced by isomorphisms and
the action of $\Piczero{X}$.
Hence we also have an isomorphism
$M_{(\beta+sH,tH)}(v) \to 
M_{(\widetilde{\beta}(s),\widetilde{\omega}(s))}(\Phi(v))$.
For each chamber ${\cal C}$, we can take 
$C_{v,\lambda} \subset {\cal C}$.
Hence $\Phi$ preserves the stability condition. 
\end{NB}
\end{rem}

\section{Numerical solutions and the walls.}

\subsection{Semi-homogeneous presentations and numerical solutions.}
\label{subsec:intro:sh}

Let us recall the notion of \emph{semi-homogeneous presentations}
introduced in \cite{YY}.

\begin{defn}\label{defn:semi-homog}
A \emph{semi-homogeneous presentation} of $E\in\Coh(X)$ is an exact sequence
\begin{align*}
\begin{matrix}
0\to  E   \to  E_1 \to  E_2 \to  0 \quad \mbox{or}\quad
0\to  E_1 \to  E_2 \to  E \to  0,
\end{matrix}
\end{align*}
where $E_i$ ($i=1,2$) are semi-homogeneous sheaves 
satisfying the following condition:
if we write $v(E_i)=\ell_i v_i$ 
with $\ell_i$ positive integers 
and $v_i$ primitive Mukai vectors, 
then 
\begin{align*}
(\ell_1-1)(\ell_2-1)=0,\quad
\langle v_1^2 \rangle=\langle v_2^2 \rangle=0,\quad
\langle v_1,v_2 \rangle=-1.
\end{align*}
We call the first sequence the \emph{kernel presentation} 
and the second one the \emph{cokernel presentation}.
\end{defn}

Since semi-homogeneous sheaves on abelian varieties are well-known objects,
semi-homogeneous presentations give useful informations on $E$.   
Moreover 
the property of having a semi-homogeneous presentation is an open condition,
and describes the birational structure of the moduli spaces $M_H(v)$.
The numerical data appeared in Definition \ref{defn:semi-homog}
are useful to find a semi-homogeneous presentation.
So we introduce the following definition.

\begin{defn}
For a Mukai vector $v$, the equation
\begin{align*}
\begin{split}
&v=\pm(\ell_1 v_1-\ell_2 v_2),
\\
&\ell_1, \ell_2 \in {\Bbb{Z}}_{>0},\quad
 v_1, v_2\colon \mbox{positive primitive Mukai vectors,}
\end{split}
\end{align*}
is called \emph{the numerical equation of $v$}.

A solution $(v_1,v_2,\ell_1,\ell_2)$ of this equation satisfying 
\begin{align*}
(\ell_1-1)(\ell_2-1)=0,\quad
\langle v_1^2 \rangle =\langle v_2^2 \rangle=0,\quad
\langle v_1,v_2 \rangle=-1,
\end{align*}
is called a \emph{numerical solution} of $v$.
\end{defn}

\begin{thm}[{\cite[Thm. 3.9]{YY}}]\label{fct:presentation}
Suppose $\NS(X)=\Bbb{Z} H$ and let $v$ be a Mukai vector with 
$\langle v^2 \rangle>0$. 
\begin{enumerate}
\item
If $v$ has at least two numerical solutions, 
then a general member of $M_H(v)$ has 
both kernel presentation and cokernel presentation. 
Each presentation is unique.
\item
If $v$ has only one numerical solution, 
then a general member of $M_H(v)$ has 
either kernel presentation or cokernel presentation. 
Such a presentation is unique.
\end{enumerate}
\end{thm}

For each numerical solution 
$(v_1,v_2,\ell_1,\ell_2)$ of $v$,
we constructed moduli spaces of simple two-term complexes.
These moduli spaces plays an important role to prove
\cite[Thm. 3.9]{YY}. 
We fix an ample divisor $H$ on $X$.

\begin{thm}[{\cite[Thm. 4.9]{YY}}]
\label{fct:moduli}
Let $v$ be a positive Mukai vector with $\mpr{v^2}>0$ and 
$(v_1,v_2,\ell_1,\ell_2)$ be a numerical solution of $v$.
\begin{enumerate}
\item
We have the fine moduli space $\frk{M}^{-}(v_1,v_2,\ell_1,\ell_2)$
of simple complexes $V^{\bullet}$ such that
$H^i(V^{\bullet})=0$ ($i \neq -1,0$),
$H^{-1}(V^{\bullet}) \in {\cal M}_H(\ell_1 v_1)^{ss}$ and 
$H^0(V^{\bullet}) \in {\cal M}_H(\ell_2 v_2)^{ss}$.
\item
We have the fine moduli space $\frk{M}^+(v_1,v_2,\ell_1,\ell_2)$
of simple complexes $V^{\bullet}$ such that
$V^{\bullet} \cong [W^{-1} \to W^0]$,  
$W^{-1} \in {\cal M}_H(\ell_1 v_1)^{ss}$ and 
$W^0 \in {\cal M}_H(\ell_2 v_2)^{ss}$.
\end{enumerate}
\end{thm}

\begin{rem}
Since
${\cal M}_H(\ell_i v_i)^{ss}$ ($i=1,2$) are independent of the choice of $H$,
$\frk{M}^\pm(v_1,v_2,\ell_1,\ell_2)$ are independent of the choice of $H$.
\end{rem}

The relation of $\frk{M}^\pm(v_1,v_2,\ell_1,\ell_2)$
is described as follows:

\begin{prop}[{\cite[Prop. 4.11]{YY}}]
\label{prop:dual}
For $i=1,2$, we denote $Y_i := M_H(v_i)$, 
and let $\bl{E}_i$ be a universal family
such that $v(\Phi_{X \to Y_i}^{\bl{E}_i^{\vee}}(v_j))=(1,0,0)$
for $j \ne i$. 
For $V^{\bullet} \in \frk{M}^{+}(v_1,v_2,\ell_1,\ell_2)$,
we set
\begin{align*}
\Psi(V^{\bullet}):=
\begin{cases}
\Phi_{Y_1 \to X}^{{\bl{E}}_1[2]}\,\cal{D}_{Y_1}\,
\Phi_{X \to Y_1}^{{\bl{E}}_1^{\vee}}(V^{\bullet})
& \ell_1=\langle v^2 \rangle/2, 
\\
\Phi_{Y_2 \to X}^{{\bl{E}}_2[1]}\,\cal{D}_{Y_2}\,
\Phi_{X \to Y_2}^{{\bl{E}}_2^{\vee}[1]}(V^{\bullet})
& \ell_2=\langle v^2 \rangle/2. 
\end{cases}
\end{align*}
Then 
$\Psi$ induces an isomorphism
\begin{align*}
\frk{M}^{+}(v_1,v_2,\ell_1,\ell_2) \simto
\frk{M}^{-}(v_1,v_2,\ell_1,\ell_2). 
\end{align*}
\end{prop}

\subsection{Relation with the walls.}

 The operation $\Psi$ is first introduced in \cite{Yoshioka:2009:ann}
to construct birational map of moduli spaces,
and it is reformulated as an isomorphism of 
$\frk{M}^\pm(v_1,v_2,\ell_1,\ell_2)$
in \cite{YY}.
$\Psi$ plays a fundamental role in these papers.
In \cite[sect. 4]{MYY:2011:1}, 
we explained its relation with Bridgeland's stability
condition. In particular, we showed that 
$\frk{M}^\pm(v_1,v_2,\ell_1,\ell_2)$ are moduli of stable objects 
if $v_1$ defines a wall. We shall slightly generalize this fact.
Thus we show the following.

\begin{lem}\label{lem:num.sol:stability}
There is a stability condition 
$\sigma_{(\beta,\omega)}=({\frak A}_{(\beta,\omega)},Z_{(\beta,\omega)})$
such that 
${\frak M}^\pm(v_1,v_2,\ell_1,\ell_2)$
are the moduli schemes of stable objects
$M_{(\beta,\omega^\pm)}(\pm v)$,
where $\omega^\pm=t^\pm \omega$,
$t^-<1<t^+$, $t^+-t^- \ll 1$.
\end{lem}

\begin{proof}
Replacing $v$ by $-v$, we assume that
$v=\ell_2 v_2-\ell_1 v_1$.
We set
$v=r e^{\gamma}-\frac{\ell}{r}\varrho_X$
and
$v_i=r_i e^{\gamma+\xi_i}$ ($i=1,2$).
Then $r=\pm(r_1 \ell_1-r_2 \ell_2)$,
$r_1 \ell_1 \xi_1=r_2 \ell_2 \xi_2$ and
$\frac{\ell}{r}=
\pm(r_1 \ell_1 \frac{(\xi_1^2)}{2}-r_2 \ell_2 \frac{(\xi_2^2)}{2})$.
Since $0>\langle e^{\gamma+\xi_1},e^{\gamma+\xi_2} \rangle
=-\frac{((\xi_2-\xi_1)^2)}{2}$,
there is an ample divisor $H$ with
$(\xi_1-\xi_2,H) \ne 0$.
We may assume that $(\xi_1,H)<(\xi_2,H)$.
We take $\beta(x):=\gamma+x \xi_1+(1-x) \xi_2$ with
$0<x<1$.
Then 
\begin{equation*}
\begin{split}
v_1= & r_1 e^{\beta(x)+(1-x)(\xi_1-\xi_2)}\\
=&
r_1 \left(e^{\beta(x)}+(1-x)(\xi_1-\xi_2+(\xi_1-\xi_2,\beta)\varrho_X)
+\frac{(1-x)^2((\xi_1-\xi_2)^2)}{2}\varrho_X \right),\\
v_2=& r_2 e^{\beta(x)+x(\xi_2-\xi_1)}\\
=&
r_2 \left(e^{\beta(x)}+x(\xi_2-\xi_1+(\xi_2-\xi_1,\beta)\varrho_X)
+\frac{x^2((\xi_1-\xi_2)^2)}{2}\varrho_X \right).
\end{split}
\end{equation*}
If $r_1 \ell_1>r_2 \ell_2$, then
$(\xi_2,H)>(\xi_1,H)$ implies that
$(\xi_2,H)>(\xi_1,H)>0$.
Hence $d_{\beta(x)}(-v_1)>0$ and $d_{\beta(x)}(v_2)>0$.
We have
\begin{equation*}
\frac{d_{\beta(x)}(-v_1)a_{\beta(x)}(v_2)
-d_{\beta(x)}(v_2)a_{\beta(x)}(-v_1)}
{d_{\beta(x)}(-v_1)r_{\beta(x)}(v_2)-d_{\beta(x)}(v_2)r_{\beta(x)}(-v_1)}
=(1-x)x \frac{((\xi_1-\xi_2)^2)}{2}>0.
\end{equation*}
Hence there is a positive number $t$ such that 
$Z_{(\beta(x),tH)}(-v_1) \in 
{\Bbb R}_{>0}Z_{(\beta(x),tH)}(v_2)$.
Then $(\beta,\omega):=(\beta(x),tH)$ belongs to the wall defined by
$-v_1$.
By \cite[Thm. 4.9]{YY} and \cite[Prop. 4.1.5]{MYY:2011:1},
${\frak M}^\pm(v_1,v_2,\ell_1,\ell_2)=M_{(\beta,\omega^\pm)}(v)$.

If $r_1 \ell_1<r_2 \ell_2$, then
we have $(\xi_1,H)<(\xi_2,H)<0$.
Hence $d_\beta(v_1)>0$ and $d_\beta(-v_2)>0$.
In this case, we also have
\begin{equation}
\frac{d_{\beta(x)}(v_1)a_{\beta(x)}(-v_2)-
d_{\beta(x)}(-v_2)a_{\beta(x)}(v_1)}
{d_{\beta(x)}(v_1)r_{\beta(x)}(-v_2)-d_{\beta(x)}(-v_2)r_{\beta(x)}(v_1)}
=(1-x)x \frac{((\xi_1-\xi_2)^2)}{2}>0.
\end{equation}
Hence there is a positive number $t$ such that 
$Z_{(\beta(x),tH)}(v_1) \in 
{\Bbb R}_{>0}Z_{(\beta(x),tH)}(-v_2)$.
Then $(\beta,\omega):=(\beta(x),tH)$ belongs to the wall defined by
$v_1$.
By \cite[Thm. 4.9]{YY} and \cite[Prop. 4.1.5]{MYY:2011:1},
${\frak M}^\pm(v_1,v_2,\ell_1,\ell_2)=M_{(\beta,\omega^\pm)}(-v)$.
\end{proof}

\begin{NB}
 
$\mu_i:=(\xi_i,H)$.

(i)
Assume that $\mu_1,\mu_2>0$. 
Since $r_1 \ell_1 \mu_1=r_2 \ell_2 \mu_2$,
$r_1 \ell_1> r_2 \ell_2 $ if and only if $\mu_2>\mu_1$.

(ii)
Assume that $\mu_1,\mu_2<0$. 
Since $r_1 \ell_1 \mu_1=r_2 \ell_2 \mu_2$,
$r_1 \ell> r_2 \ell_2$ if and only if $\mu_2<\mu_1$.

(i) Assume that $r>0$.
Then $\mu_1<\mu_2$. 
$d_{\beta+sH}(v[1])>0$
for $s-\frac{d_\beta}{r} \in (\mu_1,\mu_2)$.
Assume that $r<0$. 
Then $\mu_2 <\mu_1$.
$d_{\beta+sH}(v)>0$ for
$s-\frac{d_\beta}{r} \in (\mu_2,\mu_1)$.

(ii)
Assume that $r>0$.
Then $\mu_2<\mu_1$.
$d_{\beta+sH}(v)>0$
for $s-\frac{d_\beta}{r} \in (\mu_2,\mu_1)$.
Assume that $r<0$.
Then $\mu_2>\mu_1$.
$d_{\beta+sH}(v[1])>0$
for $s-\frac{d_\beta}{r} \in (\mu_1,\mu_2)$.
\end{NB}

\section{Stability conditions on a restricted parameter space.}
\label{sect:restricted-parameter}

\subsection{The structure of walls.}

In this section, we shall partially generalize the structure of walls
in \cite{MM}.
We note that
\begin{equation*}
\{\sigma_{(\xi,\omega)} \mid 
\xi \in \NS(X)_{\Bbb R},\ \omega \in \Amp(X)_{\Bbb R} \}=
\bigcup_{(\beta,H) \in \NS(X)_{\Bbb R} \times \Amp(X)_{\Bbb R}} 
\{ \sigma_{(\beta+sH ,tH)} \mid  s \in {\Bbb R},\ t \in {\Bbb R}_{>0} \}.
\end{equation*}
In this section,
we fix $\beta \in \NS(X)_{\Bbb R}$ and an ample divisor
$H$ on $X$ and 
study the wall for the space 
of special stability conditions:
$$
\{ \sigma_{(\beta+sH,tH)}\mid s \in {\Bbb R},\ t \in {\Bbb R}_{>0} \}.
$$

\begin{defn}\label{defn:wall2}
Let $v_1$ 
be a Mukai vector with
$\langle v_1^2 \rangle \geq 0$, $\langle (v-v_1)^2 \rangle \geq 0$
and $\langle v_1,v-v_1 \rangle >0$.
If $(\rk v_1,d_\beta(v_1),a_\beta(v_1)) \not \in
{\Bbb Q}(\rk v,d_\beta(v),a_\beta(v))$, 
we define a {\it wall} for $v$ by
$$
W_{v_1}^{\beta,H}:=
\{(s,t) \in {\Bbb R}^2 \mid {\Bbb R}Z_{(\beta+sH,tH)}(v_1)=
{\Bbb R}Z_{(\beta+sH,tH)}(v) \}.
$$
A connected component of
${\Bbb R} \times {\Bbb R}_{>0} \setminus \cup_{v_1} 
W_{v_1}^{\beta,H}$
is called a {\it chamber} for $v$.
\end{defn}

\begin{NB}
\begin{equation}
\begin{matrix}
W_{v_1}^{\beta,H} & \to & 
W_{v_1} \cap 
\{(\beta+sH,tH) 
| s \in {\Bbb R}, t \in {\Bbb R}_{>0} \}\\
(s,t) & \mapsto & (\beta+sH,tH)
\end{matrix}
\end{equation}

\end{NB}
For the study of walls, we collect elementary facts
on a family of circles.

\begin{lem}\label{lem:circle}
We take $p \in {\Bbb R}$ and $q \in {\Bbb R}_{>0}$.
For $a \in {\Bbb R}$,
let
$C_a:(x+a)^2+y^2=(a+p)^2-q$ be the circle in $(x,y)$-plane.
We also set
$C_\infty:x=p$.
Thus $C_\infty$ is a line in $(x,y)$-plane. 
\begin{enumerate}
\item[(1)]
$C_a \cap C_{a'}=\emptyset$, if $a \ne a'$.
\item[(2)]
$C_a \cap C_\infty=\emptyset$ for $a \in {\Bbb R}$.
\end{enumerate}
\end{lem}

\begin{proof}
(1)
If $C_a \cap C_{a'} \not =\emptyset$, then the intersection
satisfies $x=p$. Then $(p+a)^2+y^2=(a+p)^2-q$, which implies that
$y^2=-q<0$.
Therefore the claim holds.
The proof of (2) is similar.
\end{proof}

\begin{rem}\label{rem:base-point}
By the proof, we see that
$C_a$ $(a \in {\Bbb R})$ forms a pencil
of conics passing through the imaginary points
$\{(p,\pm\sqrt{-q}) \}$.
\end{rem}

\begin{NB}
Applonius circles:

Let $A(a+p-\sqrt{q},0)$ and $B(a+p+\sqrt{q},0)$ be two
points on $(s,t)$-plane.
Then $C_a$ is the locus of $(s,t)$
such that $AP:BP=\sqrt{a+p-\sqrt{q}}:\sqrt{ a+p+\sqrt{q}}$.
Indeed
\begin{equation}
\begin{split}
0=
& (a+p+\sqrt{q})((s-p+\sqrt{q})^2+t^2)-
(a+p-\sqrt{q})((s-p-\sqrt{q})^2+t^2)\\
=&
2\sqrt{q}((s-p)^2+q+t^2)+2(a+p)(s-p)\sqrt{q}.
\end{split}
\end{equation}
Hence the claim holds.
\end{NB}

\begin{lem}\label{lem:circle-center}
Assume that
$C_a:(x+a)^2+y^2=(a+p)^2-q$ with $q>0$
is non-empty.
If $a+p>0$, then
$(p-\sqrt{q},0)$ is contained in $C_a$.
If $a+p<0$, then
$(p+\sqrt{q},0)$ is contained in $C_a$.
\end{lem}

\begin{proof}
Assume that $a+p>0$. Then
$a+p > \sqrt{q}$.
Then
$(a+p)^2-q-(p-\sqrt{q}+a)^2=2\sqrt{q}(a+p-\sqrt{q})>0$.
Hence the claim holds.
If $a+p<0$, then
$a+p < -\sqrt{q}$.
Hence 
$(a+p)^2-q-(p+\sqrt{q}+a)^2=-2\sqrt{q}(a+p+\sqrt{q})>0$.
Therefore the claim holds.
\end{proof}
We shall study the structure of walls.
We first assume that $r \ne 0$ and set
\begin{equation*}
\begin{split}
v:= & r+dH+D+b\varrho_X \\
=& r e^{\beta}+(d_\beta H+D_\beta)+(d_\beta H+D_\beta,\beta)\varrho_X
+a_\beta \varrho_X \\
=& r e^{\beta+sH}+(d_\beta-rs)(H+(H,\beta+sH) \varrho_X)
+D_\beta+(D_\beta,\beta)\varrho_X
+\widetilde{a}_\beta \varrho_X\\
\end{split}
\end{equation*}
Then
\begin{equation*}
a_\beta=\frac{d_\beta^2 (H^2)-(\langle v^2 \rangle-(D^2_\beta))}{2r},\;
\widetilde{a}_\beta=
\frac{(d_\beta-rs)^2 (H^2)-(\langle v^2 \rangle-(D_\beta^2))}{2r}. 
\end{equation*}
We also set
\begin{equation*}
\begin{split}
v_2:=&
r_2 e^{\beta}+(d_2 H+D_2)+(d_2 H+D_2,\beta)\varrho_X
+a_2 \varrho_X\\
=& r_2 e^{\beta+sH}+(d_2-r_2 s)(H+(H,\beta+sH) \varrho_X)
+D_2+(D_2,\beta)\varrho_X
+\widetilde{a}_2 \varrho_X.
\end{split}
\end{equation*}

\begin{prop}\label{prop:wall}
Assume that $r  \ne 0$ and
$\langle v^2 \rangle>0$.
\begin{enumerate}
\item[(1)]
Assume that $rd_2-r_2 d_\beta \ne 0$.
\begin{NB}
$rd_2-r_2 d_\beta$ does not depend on the choice of $\beta$.
\end{NB}
Then
${\Bbb R}Z_{(\beta+sH,tH)}(v)={\Bbb R}Z_{(\beta+sH,tH)}(v_2)$
holds for $(s,t) \in {\Bbb R}^2$ if and only if 
\begin{equation}\label{eq:circle}
t^2+\left(s-\frac{a_2 r-a_\beta r_2}{(H^2)(rd_2-r_2 d_\beta)} \right)^2
=\left(\frac{d_\beta}{r}-\frac{a_2 r-a_\beta r_2}{(H^2)(rd_2-r_2 d_\beta)} 
\right)^2
-\frac{\langle v^2 \rangle-(D_\beta^2)}{(H^2) r^2}.
\end{equation}
\item[(2)]
Assume that $rd_2-r_2 d_\beta=0$ and $a_2 r-a_\beta r_2 \ne 0$.
Then
${\Bbb R}Z_{(\beta+sH,tH)}(v)={\Bbb R}Z_{(\beta+sH,tH)}(v_2)$
holds for $(s,t) \in {\Bbb R}^2$ if and only if 
$$
r s-d_\beta=0.
$$
\begin{NB}
If $r=0$, then $d_\beta=0$.
Since $\langle v^2 \rangle=(D_\beta^2) \leq 0$,
$v=a_\beta \varrho_X$.
\end{NB}
\end{enumerate}
\begin{NB}
Assume that $\langle v^2 \rangle>0$.
Then any pair of different two walls are disjoint
by Lemma \ref{lem:circle} or direct computation.  
\end{NB}
\end{prop}

\begin{proof}
(1)
We first note that
$$
a_\beta d_2-a_2 d_\beta=\frac{a_\beta}{r}(rd_2-d_\beta r_2)
-\frac{d_\beta}{r}(a_2 r-a_\beta r_2).
$$
Then 
\begin{equation*}
\begin{split}
& (d_2-r_2 s)\widetilde{a}_\beta-(d_\beta-rs)\widetilde{a}_2\\
=& -s^2 \frac{(H^2)}{2}(rd_2-r_2 d_\beta)+s (a_2 r-a_\beta r_2)+
(a_\beta d_2-a_2 d_\beta)\\
=& \frac{(H^2)}{2}\left\{-s^2 (rd_2-r_2 d_\beta)+
s \frac{2}{(H^2)}(a_2 r-a_\beta r_2)+
\frac{2}{(H^2)}
\left(\frac{a_\beta}{r}(rd_2-d_\beta r_2)-
\frac{d_\beta}{r}(a_2 r-a_\beta r_2) \right)
 \right\}\\
=& \frac{(H^2)}{2}\left\{-(rd_2-r_2 d_\beta)\left(s-\frac{(a_2 r-a_\beta r_2)}
{(H^2)(rd_2-r_2 d_\beta)} \right)^2+
(rd_2-r_2 d_\beta)
\left(\frac{(a_2 r-a_\beta r_2)}{(H^2)(rd_2-r_2 d_\beta)} \right)^2 \right.\\
& \left. 
+\frac{2}{(H^2)}\frac{a_\beta}{r}(rd_2-d_\beta r_2)
-2(rd_2-r_2 d_\beta)\frac{d_\beta}{r}
\frac{(a_2 r-a_\beta r_2)}{(H^2)(rd_2-r_2 d_\beta)} 
 \right\}\\
=& \frac{(H^2)}{2}\left\{-(rd_2-r_2 d_\beta)
\left(s-\frac{(a_2 r-a_\beta r_2)}
{(H^2)(rd_2-r_2 d_\beta)} \right)^2+
(rd_2-r_2 d_\beta) \left(\frac{d_\beta}{r}-\frac{(a_2 r-a_\beta r_2)}
{(H^2)(rd_2-r_2 d_\beta)} \right)^2 \right.\\
& \left. +(rd_2-r_2 d_\beta)\left(\frac{2}{(H^2)}\frac{a_\beta}{r}-
\left(\frac{d_\beta}{r}\right)^2
\right) \right\}\\
=& \frac{(rd_2-r_2 d_\beta)(H^2)}{2}\left\{-\left(s-\frac{(a_2 r-a_\beta r_2)}
{(H^2)(rd_2-r_2 d_\beta)} \right)^2+
\left(\frac{d_\beta}{r}-\frac{(a_2 r-a_\beta r_2)}
{(H^2)(rd_2-r_2 d_\beta)} \right)^2 -
\frac{\langle v^2 \rangle-(D^2)}{r^2 (H^2)}\right\}.
\end{split}
\end{equation*}
Hence the claim holds.

(2)
If $rd_2-r_2 d_\beta=0$, then the claim follows from
$$
(d_2-r_2 s)\widetilde{a}_\beta-(d_\beta-rs)\widetilde{a}_2=
-s^2 \frac{(H^2)}{2}(rd_2-r_2 d_\beta)+s (a_2 r-a_\beta r_2)+
(a_\beta d_2-a_2 d_\beta).
$$
\end{proof}

By Lemma \ref{lem:circle-center}, we get the following corollary.
\begin{cor}\label{cor:square}
\begin{enumerate}
\item[(1)]
If $
\frac{d_\beta}{r}-\frac{a_2 r-a_\beta r_2}{(H^2)(rd_2-r_2 d_\beta)}>0$,
then
$(s,t)=\left(\frac{d_\beta}{r}-
\sqrt{\frac{\langle v^2 \rangle-(D_\beta^2)}{(H^2) r^2}},0 \right)$
is contained in the circle \eqref{eq:circle}.
\item[(2)]
If $\frac{d_\beta}{r}-\frac{a_2 r-a_\beta r_2}{(H^2)(rd_2-r_2 d_\beta)}<0$,
then $(s,t)=
\left(\frac{d_\beta}{r}+
\sqrt{\tfrac{\langle v^2 \rangle-(D_\beta^2)}{(H^2) r^2}},0 \right)$
is contained in the circle \eqref{eq:circle}.
\end{enumerate}
\end{cor}

\begin{rem}\label{rem:circle-phase}
\begin{enumerate}
\item[(1)]
If $rd_2-r_2 d_\beta<0$, then
$(s,t)$ is surrounded by the circle in Proposition 
\ref{prop:wall} if and only if 
$\phi_{(\beta+sH,tH)}(v_2) \mod 2{\Bbb Z}$ satisfies
$\phi_{(\beta+sH,tH)}(v)+1>\phi_{(\beta+sH,tH)}(v_2)
>\phi_{(\beta+sH,tH)}(v)$.
\item[(2)]
If $rd_2-r_2 d_\beta>0$, then
$(s,t)$ is surrounded by the circle in Proposition 
\ref{prop:wall} if and only if 
$\phi_{(\beta+sH,tH)}(v_2) \mod 2{\Bbb Z}$ satisfies
$\phi_{(\beta+sH,tH)}(v)>\phi_{(\beta+sH,tH)}(v_2)
>\phi_{(\beta+sH,tH)}(v)-1$.
\end{enumerate}
\end{rem}

We next treat the case where $r=0$.
\begin{prop}\label{prop:wall:r=0}
Assume that $r=0$ and $\langle v^2 \rangle>0$.
If $rd_2-r_2 d_\beta \ne 0$, then 
${\Bbb R}Z_{(\beta+sH,tH)}(v)={\Bbb R}Z_{(\beta+sH,tH)}(v_2)$
holds for $(s,t) \in {\Bbb R}^2$
if and only if
\begin{equation}\label{eq:circle:r=0}
t^2+\left( s-\frac{ a_\beta }{d_\beta (H^2)} \right)^2
=\left(\frac{a_\beta}{d_\beta (H^2)}-\frac{d_2}{r_2} \right)^2
-\frac{\langle v_2^2 \rangle- (D_2^2)}{r_2^2(H^2)}.
\end{equation}
\end{prop}

\begin{proof}
We note that $d_\beta \ne 0$.
Hence $r_2 \ne 0$.
Then the claim follows from the following computation:
\begin{equation*}
\begin{split}
& (d_2-r_2 s)\widetilde{a}_\beta-(d_\beta-rs)\widetilde{a}_2\\
=& -s^2 \frac{(H^2)}{2}(rd_2-r_2 d_\beta)+s (a_2 r-a_\beta r_2)+
(a_\beta d_2-a_2 d_\beta)\\
=& \frac{r_2 d_\beta (H^2)}{2}
\left(s^2-
s \frac{2 a_\beta }{d_\beta (H^2)}+
\frac{2(a_\beta d_2-a_2 d_\beta)}{(H^2) r_2 d_\beta}\right)\\
=& \frac{r_2 d_\beta (H^2)}{2}
\left\{ \left( s-\frac{ a_\beta }{d_\beta (H^2)} \right)^2
-\left(\frac{a_\beta}{d_\beta (H^2)}-\frac{d_2}{r_2} \right)^2
+\frac{\langle v_2^2 \rangle- (D_2^2)}{r_2^2(H^2)}
\right\}.
\end{split}
\end{equation*}
\end{proof}

\begin{cor}\label{cor:finite}
If 
$\sqrt{\frac{\langle v^2 \rangle-(D_\beta^2)}{(H^2)}} \in {\Bbb Q}$,
then there are finitely many walls for $v$.
\end{cor}

\begin{proof}
For a fixed $s$,
there are finitely many walls. 
If $r=0$, then
$(s,t)=(\frac{a_\beta}{d_\beta (H^2)},0)$ is the center of 
\eqref{eq:circle:r=0}. Hence every wall intersects with 
$s=\frac{a_\beta}{d_\beta (H^2)}$.
\begin{NB}
If $r=0$, then the assumption always holds. 
\end{NB}
If $r \ne 0$, 
then Corollary \ref{cor:square}
implies that every wall intersects with 
$s=\frac{d_\beta}{r}-
\sqrt{\frac{\langle v^2 \rangle-(D_\beta^2)}{(H^2) r^2}}$
or 
$s=\frac{d_\beta}{r}+
\sqrt{\frac{\langle v^2 \rangle-(D_\beta^2)}{(H^2) r^2}}$.
Hence there are finitely many walls for $v$.
\end{proof}

\begin{lem}\label{lem:C_v}
For $v_2=e^{\beta+\lambda H}$,
the condition 
${\Bbb R}Z_{(\beta+sH,tH)}(v)={\Bbb R}Z_{(\beta+sH,tH)}(v_2)$
for $(s,t) \in {\Bbb R}^2$
is equivalent to the equation of the circle
$$
C_{v,\lambda}:
t^2+(s-\lambda)
\left(s-
\frac{1}{(r\lambda-d_\beta)}
\left(\lambda d_\beta-\frac{2}{(H^2)}a_\beta \right)
\right)=0
$$
for $r \lambda-d_\beta \ne 0$
and
$$
rs-d_\beta=0
$$
for $r \lambda-d_\beta= 0$.
In particular, the circle $C_{v,\lambda}$
passes the points $(\lambda,0)$
and 
$\left(\frac{1}{(r\lambda-d_\beta)}
(\lambda d_\beta-\frac{2}{(H^2)}a_\beta),0 \right)$.
\end{lem}

\begin{proof}
We note that
\begin{equation*}
\begin{split}
v_2= & e^{\beta+\lambda H}=e^\beta+\lambda(H+(H,\beta)\varrho_X)+
\frac{(H^2)}{2}\lambda^2 \varrho_X \\
=& e^{\beta+sH+(\lambda-s) H}=
e^{\beta+sH}+(\lambda-s)(H+(H,\beta)\varrho_X)+
\frac{(H^2)}{2}(\lambda-s)^2 \varrho_X.
\end{split}
\end{equation*}
Then we get
\begin{equation}\label{eq:lambda}
\begin{split}
(d_2-r_2 s)\widetilde{a}_\beta-(d_\beta-rs)\widetilde{a}_2
=& (\lambda-s)\widetilde{a}_\beta-(d_\beta-rs)\frac{(H^2)}{2}(\lambda-s)^2 \\
=& (\lambda-s)(\widetilde{a}_\beta-(\lambda-s)(d_\beta-rs)\frac{(H^2)}{2})\\
=& (\lambda-s) \left((r\lambda-d_\beta)\frac{(H^2)}{2}s-
\left(\lambda d_\beta\frac{(H^2)}{2}-a_\beta \right) \right).
\end{split}
\end{equation}
Assume that $r \lambda-d_\beta \ne 0$.
Then the condition is given by the circle
$$
t^2+(s-\lambda)
\left(s-
\frac{1}{(r\lambda-d_\beta)}\left(\lambda d_\beta-\frac{2}{(H^2)}a_\beta
\right) \right)=0.
$$
In particular, the circle passes the points $(\lambda,0)$
and $\left(\frac{1}{(r\lambda-d_\beta)}
(\lambda d_\beta-\frac{2}{(H^2)}a_\beta),0 \right)$.

Assume that $r \lambda-d_\beta=0$.
Then by \eqref{eq:lambda},
we get 
$$
0=(\lambda-s)\left(\lambda d_\beta \frac{(H^2)}{2}-a_\beta \right).
$$
If $a_\beta=\lambda d_\beta \frac{(H^2)}{2}=r\lambda^2 \frac{(H^2)}{2}$,
then we see that
$v=r e^{\beta+\lambda H}+(D_\beta+(D_\beta,\beta+\lambda H))\varrho_X$.
Hence $\langle v^2 \rangle=(D_\beta^2) \leq 0$, which is a contradiction.
Therefore we get 
$$
s=\lambda=\frac{d_\beta}{r}.
$$
\end{proof}

\begin{rem}
Assume that $r \ne 0$.
Then 
$$
\lambda \ne \frac{1}{r\lambda-d_\beta}
\left(\lambda d_\beta-\frac{2}{(H^2)}a_\beta \right)
\Longleftrightarrow
\lambda \ne
\frac{d_\beta}{r}\pm\sqrt{\frac{\langle v^2 \rangle-(D_\beta^2)}{(H^2)r^2}}.
$$
In particular, if 
$\sqrt{\frac{\langle v^2 \rangle-(D_\beta^2)}{(H^2)r^2}} \not \in {\Bbb Q}$,
then $C_{v,\lambda}$ is a circle.
\end{rem}

\begin{NB}
Assume that $v=\ell v_1-v_2$.  Then $\mu(v_1)>\mu(v_2)$
 or
$\mu(v_1)<\mu(v_2)$.
Assume that $(s,t)$ belongs to the circle
defined by $v_1$,
i.e.,
$s-\frac{d_\beta}{r}$
belongs to the interval whose boundaries
are $\lambda-\frac{d_\beta}{r}$ and
$\frac{\ell}{r n(r\lambda-d_\beta)}$.

If $\mu(v_1)>\mu(v_2)$, then
 $Z_{(sH,tH)}(v) \in {\Bbb H}$ for
$d_{\beta+sH}(v_1)>0$ and 
if $\mu(v_1)<\mu(v_2)$, then
 $Z_{(sH,tH)}(-v) \in {\Bbb H}$ for
$d_{\beta+sH}(v_1)<0$.  
\end{NB}

\begin{cor}\label{cor:C_v}
Assume that $v=r e^\beta+a_\beta \varrho_X+d_\beta (H+(H,\beta)\varrho_X)$.
For a numerical solution
$$
(r_1 e^{\beta+\lambda_1 H},r_2 e^{\beta+\lambda_2 H},\ell_1,\ell_2),
$$
we have $C_{v,\lambda_1}=C_{v,\lambda_2}$ and the equation is given by
\begin{equation}
t^2+(s-\lambda_1)(s-\lambda_2)=0.
\end{equation}
\end{cor}

\begin{proof}
Since $v=\pm(\ell_1 r_1 e^{\beta+\lambda_1}-\ell_2 r_2 e^{\beta+\lambda_2})$,
we have $C_{v,\lambda_1}=C_{v,\lambda_2}$.
Since $(\lambda_i,0) \in C_{v,\lambda_i}$ for $i=1,2$
and $\lambda_1 \ne \lambda_2$,
we get the claim.
\end{proof}

\subsection{Relation of stability conditions.}

All walls except $rs=d_\beta$ are disjoint to the line
$rs=d_\beta$. 
By Corollary \ref{cor:square},
there are at most two unbounded chambers.

\begin{prop}\label{prop:unbdd}
Let $v$ be a positive and primitive Mukai vector such that
$\langle v^2 \rangle>0$.
Assume that $(s,t) \in {\Bbb R}^2$ belongs to an unbounded chamber.
Then
$$
M_{(\beta+sH,tH)}(v) \cong 
\begin{cases}
M_H^\beta(v) & d_{\beta+sH}(v) >0,\\
M_H^{-\beta}(v^{\vee}) & d_{\beta+sH}(v)  \leq 0.
\end{cases}
$$
\end{prop}

\begin{proof}
If $d_{\beta+sH}(v)>0$, then 
$\rk v \geq 0$. 
By \cite[Cor. 2.2.9]{MYY:2011:1}, we get 
$M_{(\beta+sH,tH)}(v)=M_H^\beta(v)$.
If $d_\beta(v) \leq 0$,
then $d_{\beta+sH}(-v) \geq 0$ and $\rk (-v) \leq 0$.
If $\rk (-v)=0$, then $d_\beta(-v)=0$ and
$\langle v^2 \rangle  \leq 0$,
which is a contradiction. 
Thus we have $\rk (-v)< 0$.
Then \cite[Cor. 2.2.9]{MYY:2011:1} implies that we have an isomorphism
$M_H^{-\beta}(v^{\vee}) \to M_{(\beta+sH,tH)}(v)$
via $F \mapsto F^{\vee}$.
\end{proof}

\begin{lem}\label{lem:2-stability}
Let ${\cal C}_0$ and ${\cal C}_1$ be two chambers such that
${\cal C}_0$ is surrounded by ${\cal C}_1$.
We take $(s,t_0) \in {\cal C}_0$ and $(s,t_1) \in {\cal C}_1$.
Let $t_x:=(1-x) t_0+ x t_1$
($0 \leq x \leq 1$) be a segment connecting 
$t_0$ and $t_1$.
If $E$ is $\sigma_{(\beta+s H,t_i H)}$-semi-stable
for $i=0,1$, then
$E$ is $\sigma_{(\beta+sH,t_x H)}$-semi-stable
for all $x$.
\end{lem}

\begin{proof}
Assume that $E$ is not $\sigma_{(\beta+sH,t_x H)}$-semi-stable
for some $x \in (0,1)$.
Then there is a subobject $E_1$ of $E$ in 
${\frak A}_{(\beta+sH,t_x H)}$
such that $\phi_{(\beta+sH,t_x H)}(E_1)>
\phi_{(\beta+sH,t_x H)}(E)$.
Since $E$ is $\sigma_{(\beta+s H,t_i H)}$-semi-stable
for $i=0,1$,
$\phi_{(\beta+sH,t_i H)}(E_1) \leq 
\phi_{(\beta+sH,t_i H)}(E)$.
Then there are two numbers $x_1, x_2$ such that
$0<x_1,x_2<1$ and $\phi_{(\beta+sH,t_{x_i} H)}(E_1)= 
\phi_{(\beta+sH,t_{x_i} H)}(E)$. 
Since $t_{x_i}$ is uniquely determined by $v(E_2)$,
this does not occur.
Therefore $E$ is
$\sigma_{(\beta+sH,t_x H)}$-semi-stable
for all $x$. 
\end{proof}

\begin{defn}
Let $W$ be a wall for $v$ in $(s,t)$-plane. 
Let $(\beta,\omega)$ be a point of $W$ and
$(\beta',\omega')$ be a point in an adjacent chamber. 
Then we define \emph{the codimension of the wall} $W$ by 
\begin{equation}
\codim W:=\min_{v=\sum_i v_i } \left \{ \sum_{i<j}\langle v_i,v_j \rangle
-\left(\sum_i (\dim {\cal M}_H^{\beta'}(v_i)^{ss}-
\langle v_i^2 \rangle) \right)
+1 \right\},
\end{equation}
where $v=\sum_i v_i$ are decompositions of $v$
such that $\phi_{(\beta,\omega)}(v)=\phi_{(\beta,\omega)}(v_i)$
and $\phi_{(\beta',\omega')}(v_i)>\phi_{(\beta',\omega')}(v_j)$,
$i<j$.
\end{defn}

By using \cite[Lem. 4.2.4, Rem. 4.2.3]{MYY:2011:1},
we get the following result.
\begin{lem}\label{lem:class-codim0wall}
If $W$ is a codimension 0 wall, then
$W$ is defined by $v_1$ such that
\begin{equation}\label{eq:codim=0}
v=n v_1+v_2,\; \langle v_1,v_2 \rangle=1,\;
\langle v_1^2 \rangle=\langle v_2^2 \rangle=0.
\end{equation} 
\end{lem}

If $v_1,v_2$ in \eqref{eq:codim=0} satisfy
$v_1>0,v_2<0$ or $v_1<0$ and $v_2>0$, then $(v_1,-v_2,n,1)$
or $(-v_1,v_2,n,1)$ gives a numerical solution of $v$.
Conversely for a numerical solution $(v_1,v_2,\ell_1,\ell_2)$,
Lemma \ref{lem:num.sol:stability} implies that
for a suitable $\beta$ and $\omega=tH$,
$v_1$ defines a codimension 0 wall.

\begin{prop}
Assume that $\NS(X)={\Bbb Z}H$.
We fix $\beta$.
Then there is a bijective correspondence
between a codimension 0 wall and a numerical solution. 
\end{prop}

\begin{proof}
Assume that $\NS(X)={\Bbb Z}H$.
For a numerical solution $(v_1,v_2,\ell_1,\ell_2)$,
$\beta, c_1(v),c_1(v_1),c_1(v_2) \in {\Bbb Q}H$ implies that
$C_{v,\lambda_1}$ in Corollary \ref{cor:C_v}
gives a wall in the $(s,t)$-plane.
Combining Lemma \ref{lem:class-codim0wall},
we get the claim.
\end{proof}

\begin{prop}[{cf. \cite[Prop. 4.2.5]{MYY:2011:1}}]
If $(\beta_1,\omega_1)$ and $(\beta_2,\omega_2)$ are not separated by
any codimension 0 wall, then
$M_{(\beta_1,\omega_1)}(v) \cap M_{(\beta_2,\omega_2)}(v) \ne \emptyset$.
In particular,
$M_{(\beta_1,\omega_1)}(v)$ and $M_{(\beta_2,\omega_2)}(v)$
are birationally equivalent.
\end{prop}

\subsection{Semi-homogeneous presentation and the stability.}
\label{subsect:semi-hom-stab}

For a semi-homogeneous presentation,
Lemma \ref{lem:num.sol:stability} implies that
we can relate a $\sigma_{(\beta,\omega)}$-semi-stability.
We shall study Gieseker semi-stability
of coherent sheaves with two semi-homogeneous presentations.

\begin{lem}\label{lem:s.h-ss}
Let $v$ be a primitive Mukai vector with $r:=\rk v>0$.
Assume that
${\cal M}_H^\beta(v)^{ss}$ consists of $\beta$-stable sheaves. 
If a simple sheaf $E$ with $v(E)=v$ 
has two semi-homogeneous presentations
and $d_\beta(v)-r s=0$ is not a codimension 0 wall,
then $E$ is a $\mu$-stable vector bundle.
\end{lem}

\begin{proof}
Since $rs-d_\beta(v)=0$ is not a codimension 0 wall,
two semi-homogeneous presentations
define two circles $C_1$ and $C_2$
which are separated by the line
$rs-d_\beta(v)=0$.

There is a Fourier-Mukai transform $\Phi:=\Phi_{Y \to X}^{{\bf E}}$
and a complex $F$ such that $E=\Phi(F)$ and
$F$ is semi-stable
with respect to two chambers ${\cal C}_0$, ${\cal C}_1$ such that
${\cal C}_0$ is surrounded by ${\cal C}_1$
and ${\cal C}_i$ are adjacent to $C_i$.
Let ${\cal C}$ be an unbounded chamber between $\Phi({\cal C}_1)$ and
$\Phi({\cal C}_2)$. Then 
$F$ is semi-stable with respect to
$\Phi^{-1}({\cal C})$ by Lemma \ref{lem:2-stability}.
Hence $E$ is semi-stable with respect to all unbounded 
chambers.

We take an element $\omega \in \Amp(X)_{\Bbb Q}$. 
Then $E$ is $\sigma_{(\beta+sH,\omega)}$-semi-stable
if $1 \gg d_\beta(v)-rs>0$.
Hence $E$ is $\beta$-twisted semi-stable.
Since $E$ is also $\sigma_{(\beta+sH,\omega)}$-semi-stable
if $1 \gg -(d_\beta(v)-rs)>0$,
$E^{\vee}$ is $(-\beta)$-twisted semi-stable.
Hence $E$ is locally free.
Let $E_1$ be a locally free subsheaf of $E$.
Then there is a generically surjective homomorphism
$E^{\vee} \to E_1^{\vee}$.
Since 
$$
\frac{\chi(E_1(-\beta))}{\rk E_1}=
\frac{\chi(E_1^{\vee}(\beta))}{\rk E_1^{\vee}},
$$
$\frac{d_\beta(E_1)}{\rk E_1}=
\frac{d_\beta(E)}{\rk E}$ implies 
$\frac{\chi(E_1(-\beta))}{\rk E_1}=
\frac{\chi(E(-\beta))}{\rk E}$.
Thus $E$ is properly $\beta$-twisted semi-stable,
which is a contradiction.
Therefore $E$ is $\mu$-stable.
\end{proof}

\begin{lem}\label{lem:codim0-line}
Let $v$ be a primitive Mukai vector with $r:=\rk v>0$.
Assume that $d_\beta(v)-r s=0$ defines a wall.
\begin{enumerate}
\item[(1)]
$d_\beta(v)-r s=0$ is a codimension 0 wall if and only if
(a) 
$v=r e^{\xi}-a \varrho_X$, $\xi \in \NS(X), a \in {\Bbb Z}$,
$(r-1)(a-1)=0$
or (b) 
$v=v_1+\ell v_2$,
$v_i=r_i e^{\frac{\xi_i}{r_i}}$,
$r_1,r_2>0$,
$((r_2 \xi_1-r_1 \xi_2)^2)=-r_1 r_2$,
$(r_2 \xi_1-r_1 \xi_2,H)=0$.
\item[(2)]
Assume that $v$ satisfies (a).
We take $(s,t)$ such that $1 \gg d_\beta(v)-rs>0$
and let $E$ be a $\sigma_{(\beta+sH,tH)}$-semi-stable object with
$v(E)=v$. Then
$r=1$ and $E=I_Z(\xi)$ or
$a=1$ and $E=\ker(\oplus_{i=1}^r {\cal O}_X(\xi_i) \to {\frak k}_x)$.  
\item[(3)]
If $\NS(X)={\Bbb Z}H$, then (b) does not occur.
\end{enumerate}
\end{lem}

\begin{proof}
(1), (3) 
By Lemma \ref{lem:class-codim0wall},
we have $v=v_1+\ell v_2$,
$\langle v_i^2 \rangle=0$ and $\langle v_1,v_2 \rangle=1$,
where $\ell=\langle v^2 \rangle/2$.
Assume that $\rk v_1 \rk v_2 \ne 0$.
Then we can set
$v_i:=r_i e^{\frac{\xi_i}{r_i}}$
$((r_2 \xi_1-r_1 \xi_2)^2)=-r_1 r_2$.
Since the wall is $d_\beta-rs=0$,
we have $(r_2 \xi_1-r_1 \xi_2,H)=0$.
By the Hodge index theorem,
$r_1 r_2>0$.
Thus $r_1, r_2>0$.
If $\rk v_1=0$ or $\rk v_2=0$,
then $r_2 d_\beta(v_1)-r_1 d_\beta(v_2)=0$ and 
$\langle v_1,v_2\rangle=1$ implies that 
$\{v_1, v_2 \}=\{(0,0,1),(-1,0,0) \}$.
Therefore (1) holds.
If $\NS(X)={\Bbb Z}H$ and $(r_2 \xi_1-r_1 \xi_2,H)=0$,
then we have $r_2 \xi_1-r_1 \xi_2=0$.
Hence $r_1 r_2=0$. Thus (b) does not occur.

(2) 
By the choice of $(s,t)$,
$\sigma_{(\beta+sH,tH)}$-semi-stability
implies $\beta$-twisted semi-stability.
Thus for $E \in M_{(\beta+sH,tH)}(v)$,
we have 
a semi-homogeneous presentation 
$$
0 \to E \to E_0 \to E_1 \to 0,
$$
where $(v(E_0),v(E_1))=(e^\xi,\ell \varrho_X)$
or $(v(E_0),v(E_1))=(\ell e^\xi, \varrho_X)$
with $\langle v^2 \rangle=2\ell$.
Hence the claim holds.
\end{proof}

\begin{prop}
Let $X$ be an abelian surface with $\NS(X)={\Bbb Z}H$.
Let $E$ be a simple sheaf on $X$.
If $E$ has two semi-homogeneous presentations,
then $E$ is Gieseker semi-stable.
\end{prop}

\begin{proof}
If $d_\beta(v)-rs=0$ is a codimension 0 wall,
then Lemma \ref{lem:codim0-line} implies the claim.
If $d_\beta(v)-rs=0$ is not a codimension 0 wall, 
then the claim follows from Lemma \ref{lem:s.h-ss}.
\end{proof}

\subsection{Relation with the Fourier-Mukai transforms.}

We shall study the Fourier-Mukai transform
on our space of stability conditions.
Let $r_1 e^\gamma$ be a primitive and isotropic Mukai vector.
We set $X_1:=M_{(\beta+sH,tH)}(r_1 e^{\gamma})$.
Let
${\bf E}$ be the universal object on $X \times X_1$
as a complex of twisted sheaves.
Assume that $\gamma=\beta+\lambda H$, $\lambda \in {\Bbb Q}$.
We consider 
the Fourier-Mukai transform
$\Phi:=\Phi_{X \to X_1}^{{\bf E}^{\vee}[1]}:
{\bf D}(X) \to {\bf D}^{\alpha_1}(X_1)$.
We set $(\widetilde{\beta+sH},\widetilde{tH})=
(\gamma'+s' \widehat{H},t' \widehat{H})$. 
Then
\begin{equation}\label{eq:s'}
\begin{split}
s'=& 
\frac{1}{|r_1|}
\frac{2(\lambda-s)}{((\lambda-s)^2+t^2) (H^2)},\\
t'=& 
\frac{1}{|r_1|}
\frac{2 t}{((\lambda-s)^2+t^2) (H^2)}.\\
\end{split}
\end{equation}
Since $((s-\lambda)^2+t^2)({s'}^2+{t'}^2)=
\left(\frac{2}{|r_1|(H^2)} \right)^2$,
the image of $(s-\lambda)^2+t^2=\frac{2}{|r_1|(H^2)}$
is ${s'}^2+{t'}^2=\frac{2}{|r_1|(H^2)}$.

\begin{NB}
If we fix a base point $\beta'$
and write $\beta'=\gamma'-\lambda' \widehat{H}$,
then $\beta'+s' \widehat{H}=\gamma'+(s'-\lambda')\widehat{H}$
and the relation of
$Z_{(\beta+sH,tH)}$ and $Z_{(\beta'+s' \widehat{H},t' \widehat{H})}$
is 
$$
s'-\lambda'=\frac{1}{|r_1|}
\frac{2(\lambda-s)}{((\lambda-s)^2+t^2) (H^2)},\;
t'=\frac{1}{|r_1|}
\frac{2t}{((\lambda-s)^2+t^2) (H^2)}.
$$

\end{NB}

If $\lambda r \ne d_\beta$,
then Lemma \ref{lem:C_v}
implies that
the condition
${\Bbb R}Z_{(\beta+sH,tH)}(v)=
{\Bbb R}Z_{(\beta+sH,tH)}(e^{\beta+\lambda H})$
defines a circle
\begin{equation}\label{eq:circle-lambda}
C_{v,\lambda}:\;
t^2=
(\lambda-s)
\left( \frac{a_\beta-d_\beta \lambda \frac{(H^2)}{2}}{\lambda r-d_\beta}
\frac{2}{(H^2)}+s \right).
\end{equation}
%
We have
\begin{equation}\label{eq:a_gamma}
\frac{a_\beta-d_\beta \lambda \frac{(H^2)}{2}}{\lambda r-d_\beta}
\frac{2}{(H^2)}+\lambda
=\frac{(\lambda r-d_\beta)^2 (H^2)-
(\langle v^2 \rangle-(D_\beta^2))}
{r(H^2)(\lambda r-d_\beta)}
=\frac{2a_\gamma}{-d_\gamma (H^2)}=
\frac{2\langle e^{\beta+\lambda H},v \rangle}
{d_\gamma (H^2)}.
\end{equation}
Thus 
$C_{v,\lambda}$is
$$
\left(s-\left(\lambda+\frac{a_\gamma}{d_\gamma (H^2)} \right)\right)^2+t^2
=\left(\frac{a_\gamma}{d_\gamma (H^2)} \right)^2.
$$
\begin{lem}\label{lem:circle-lambda-interior}
The image of
\begin{equation}\label{eq:circle-lambda-interior}
t^2 \leq 
(\lambda-s)
\left( \frac{a_\beta-d_\beta \lambda \frac{(H^2)}{2}}{\lambda r-d_\beta}
\frac{2}{(H^2)}+s \right)
\end{equation}
by $\Phi$ is 
$$
\left \{(s',t') \left| -\tfrac{|r_1| a_\gamma}{d_\gamma}s' \geq 1
\right. \right\}.
$$
\end{lem}

\begin{proof}
By \eqref{eq:a_gamma} and \eqref{eq:s'},
$(s,t)$ satisfies \eqref{eq:circle-lambda-interior}
if and only if 
\begin{equation}
\begin{split}
0 \geq & (s-\lambda)^2+t^2-(s-\lambda)\frac{2a_\gamma}{d_\gamma (H^2)} \\
=& \left((s-\lambda)^2+t^2 \right)\left(
1+s' |r_1| \frac{a_\gamma}{d_\gamma} \right).
\end{split}
\end{equation}
Hence the claim holds.
\end{proof}

By Lemma \ref{lem:circle-lambda-interior},
we have the following.
\begin{prop}\label{prop:Phi(C^pm)}
Let ${\cal C}^\pm$ be the adjacent chamber of
$C_{v,\lambda}$ such that ${\cal C}^-$ is surrounded by
${\cal C}^+$.
\begin{enumerate}
\item[(1)]
If $\frac{a_\gamma}{d_\gamma}<0$, 
then $\Phi({\cal C}^+)$ (resp. $\Phi({\cal C}^-)$)
is the unbounded 
chamber satisfying
$s' < -\frac{d_\gamma}{|r_1| a_\gamma}$
(resp. $s' > -\frac{d_\gamma}{|r_1| a_\gamma}$).
\item[(2)]
If $\frac{a_\gamma}{d_\gamma}>0$, 
then $\Phi({\cal C}^-)$ (resp.
$\Phi({\cal C}^+)$)
is the unbounded 
chamber satisfying
$s' < -\frac{d_\gamma}{|r_1| a_\gamma}$
(resp. $s' > -\frac{d_\gamma}{|r_1| a_\gamma}$).
\end{enumerate}
\end{prop}
\begin{NB}
$\frac{d_\gamma}{|r_1| a_\gamma}=
\frac{d_{\gamma'} (\Phi(v))}{-\rk \Phi(v)}$.
\end{NB}
\begin{NB}
We note that $rs=d_\beta$ does not meet $C_{v,\lambda}$.
Assume that
$d_\gamma=d_\beta-r \lambda>0$.
Then $d_{\beta+sH}(e^\gamma)=\lambda-s \geq 0$ for 
$(s,t) \in C_{v,\lambda}$ and $a_\gamma<0$.
We also have $r_1>0$, since we require 
$d_{\beta+sH}(r_1 e^\gamma)>0$.
Since $\rk \Phi(v)=-r_1 a_\gamma>0$,
For $E \in M_{(\beta+sH,tH)}(v)$, $(s,t) \in {\cal C}^+$,
$\Phi(E)$ is a twisted stable sheaf.
\end{NB}

\begin{NB}
Assume that
$d_\gamma=d_\beta-r \lambda>0$.
Then $d_{\beta+sH}(e^\gamma)=\lambda-s \leq 0$ for 
$(s,t) \in C_{v,\lambda}$ and $a_\gamma>0$.
We also have $r_1<0$, since we require 
$d_{\beta+sH}(r_1 e^\gamma)>0$.
Since $\rk \Phi(v)=-r_1 a_\gamma>0$,
For $E \in M_{(\beta+sH,tH)}(v)$, $(s,t) \in {\cal C}^-$,
$\Phi(E)$ is a twisted stable sheaf.
\end{NB}

\begin{NB}
Wrong arguments (Feb. 18, 2012):
Then the intersection of $C_{v,\lambda}$ with the
circle
$(s-\lambda)^2+t^2=\frac{2}{|r_1|(H^2)}$
lies on the line
\begin{equation}\label{eq:line}
s=\lambda+\frac{d_\gamma}{a_\gamma |r_1|}.
\end{equation}
By $\Phi$,
the circle is transformed to
the line 
\begin{equation}
s'=\lambda+\frac{d_\gamma}{a_\gamma |r_1|}.
\end{equation}
\end{NB}

\begin{NB}
$$
s'=\Phi_{X \to Y}^{{\bf E}^{\vee}[1]}(s)=
\lambda+\frac{2}{(H^2)|r_1|}
\frac{s-\lambda}{(s-\lambda)^2+t^2}.
$$
\end{NB}
We set $w:=\Phi_{X \to X_1}^{{\bf E}^{\vee}[1]}(v)$.
Proposition \ref{prop:unbdd} implies 
$M_{(\gamma'+s'\widehat{H},t' \widehat{H}) }^{\alpha_1}(w)$ 
is isomorphic to the moduli space of
semi-stable sheaves. 
Then Theorem \ref{thm:B:3-10.3} implies 
a generalization of \cite[Thm. 3.3.3]{MYY:2011:2}
for abelian surfaces.

For the preservation of Gieseker's semi-stability, we also
have the following, which is a generalization of 
\cite{Stability}.

\begin{prop}\label{prop:asymptotic}
Let $v$ be a positive Mukai vector and
assume that there are walls for $v$.
Let $W^{\max}$ be the wall in the region $rs<d_\beta$ such that
$W^{\max}$ surround all walls in $rs<d_\beta$, that is,
the boundaries of the unbounded chamber is $rs=d_\beta$ and
$W^{\max}$.
We set 
$$
W^{\max} \cap \{(s,0) \mid s \in {\Bbb R} \}
=\{(\lambda_1,0), (\lambda_2,0) \},\quad
\lambda_1<\lambda_2.
$$
Let $\Phi:{\bf D}(X) \to {\bf D}(X_1)$ be the Fourier-Mukai
transform as above.
If $\lambda \leq \lambda_1$ or 
$\lambda_2 \leq 
\lambda<\frac{d_\beta}{r}$, then 
$\Phi$ or $\Phi \circ {\cal D}_X$
preserves the Gieseker's semi-stability.
\end{prop}

In particular, if $\lambda$ is sufficiently small, we can apply
this proposition, which is nothing but the main result of
\cite{Stability}.

\begin{rem}
Assume that $r:=\rk v>0$.
We set
$$
D:=\min \{((\rk w) c_1(v)-(\rk v) c_1(w),H)>0 \mid
w \in H^*(X,{\Bbb Z})_{\alg} \}
$$
and assume that
$$
D=\min \{(C,H)>0  \mid C \in \NS(X) \}.
$$
We take 
$w_0 \in H^*(X,{\Bbb Z})_{\alg}$ such that
\begin{equation*}
((\rk w_0) d_\beta(v)-(\rk v) d_\beta(w_0))(H^2)=
((\rk w_0) c_1(v)-(\rk v) c_1(w_0),H)=D.
\end{equation*}
Replacing $w_0$ by $w_0+k v$ $(k \in {\Bbb Z})$,
we may assume that $\rk v \geq \rk w_0>0$.
Then there is no wall for 
$d_\beta(w_0)/\rk w_0 \leq s< d_\beta(v)/\rk v$.
\end{rem}

\section{
The chamber structure for an abelian surface $X$ with
$\NS(X)={\Bbb Z}H$.}
\label{sect:rho=1}

From now on, we assume that $\NS(X)={\Bbb Z}H$.
Let $v$ be a primitive Mukai vector with a numerical solution.
We shall study the walls and chambers for $v$.
By our assumption,
there is an isometry of Mukai lattice sending $v$ to $1-\ell \varrho_X$. 
So we may assume that $v=1-\ell \varrho_X$.
Since a generic classification of stable objects 
(it induces the birational classification)
is most fundamental,
we are mainly interested in codimension 0 walls.

\subsection{Cohomological Fourier-Mukai transforms}
Let $H_X$ be the ample generator of $\NS(X)$.
We shall describe the action of Fourier-Mukai transforms 
on the cohomology lattices in \cite{YY}.

Two smooth projective varieties $Y_1$ and $Y_2$ are said to be 
Fourier-Mukai partners if there is an equivalence 
$\bl{D}(Y_1)\simeq\bl{D}(Y_2)$.
We denote by $\FM(X)$ the set of Fourier-Mukai partners of $X$. 
The set of equivalences between $\bl{D}(X)$ and $\bl{D}(Y)$ 
is denoted by $\Eq(\bl{D}(X),\bl{D}(Y))$. 
For $Y,Z\in\FM(X)$, we set 
\begin{align*}
&\Eq_0(\bl{D}(Y),\bl{D}(Z))
:=
 \{\Phi_{Y \to Z}^{\bl{E}[2k]} 
   \in \Eq(\bl{D}(Y),\bl{D}(Z))
   \mid \bl{E} \in \Coh(Y \times Z),\, k \in {\bb{Z}} 
 \},
\\
&\cal{E}(Z)
:=
 \bigcup_{Y\in\FM(Z)}\Eq_0(\bl{D}(Y),\bl{D}(Z)),
\qquad
\cal{E}
:=
 \bigcup_{Z\in\FM(X)}\cal{E}(Z)
=\bigcup_{Y,Z\in\FM(X)}\Eq_0(\bl{D}(Y),\bl{D}(Z)).
\end{align*}
Note that $\cal{E}$ is a groupoid with respect to the composition 
of the equivalences.
For $Y \in \FM(X)$, we have $(H_Y^2)=(H_X^2)$.
We set $n:= (H^2_X)/2$.

In \cite[sect. 6.4]{YY}, we constructed an isomorphism of lattices
\begin{align*}
\iota_X: 
(H^*(X,\bb{Z})_{\alg},\langle\cdot,\cdot\rangle) 
\simto
(\Sym_2(\bb{Z}, n),B),
\quad
(r,dH_X,a) 
\mapsto 
\begin{pmatrix}
 r & d\sqrt{n} \\ d\sqrt{n} & a
\end{pmatrix},
\end{align*}
where $\Sym_2(\bb{Z}, n)$ is given by 
\begin{align*}
 \Sym_2(\bb{Z},n):=
 \left\{\begin{pmatrix} x &y \sqrt{n} \\ y\sqrt{n}&z\end{pmatrix}\, 
 \Bigg|\, x,y,z\in \bb{Z}\right\},
\end{align*}
and the bilinear form $B$ on $\Sym_2(\bb{Z},n)$ is given by 
\begin{align*}
B(X_1,X_2) := 2ny_1 y_2-(x_1 z_2+z_1 x_2)
\end{align*}
for
$X_i =\begin{pmatrix}x_i & y_i \sqrt{n} \\ y_i \sqrt{n} &z_i \end{pmatrix}
 \in\Sym_2(\bb{Z},n)$
($i=1,2$).

Each $\Phi_{X \to Y}$ gives an isometry 
\begin{align}
 \iota_{Y} \circ \Phi^H_{X \to Y} \circ \iota_{X}^{-1}
\in \LieO(\Sym_2(\bb{Z}, n)),
\end{align}
where $\LieO(\Sym_2(\bb{Z}, n))$ is the isometry group
of the lattice $(\Sym_2(\bb{Z}, n),B)$. 
Thus we have a map
$$
\eta:
{\cal E} \to \LieO(\Sym_2(\bb{Z}, n))
$$
which preserves the structures of 
multiplications. 

\begin{defn}\label{defn:G}
We set
\begin{align*}
&\widehat{G} := 
\left\{
 \begin{pmatrix} a \sqrt{r} & b \sqrt{s}\\
c \sqrt{s} & d \sqrt{r}
\end{pmatrix}
 \Bigg|\, 
\begin{aligned}
 a,b,c,d,r,s \in \bb{Z},\, r,s>0\\ 
rs=n, \, adr-bcs = \pm1
\end{aligned}
\right\},
\\
&G := \widehat{G}\cap \SL(2,\bb{R}).
\end{align*}
\end{defn}

We have a right action $\cdot$ of $\widehat{G}$ on the lattice
$(\Sym_2(\bb{Z}, n),B)$:
 
\begin{align}
\label{eq:action_cdot}
\begin{pmatrix}
r &d  \sqrt{n}\\d  \sqrt{n}&a \end{pmatrix}
\cdot g := 
 {}^t g
 \begin{pmatrix}r &d  \sqrt{n}\\d  \sqrt{n}&a \end{pmatrix} 
 g,\;
g \in \widehat{G}.
\end{align}
Thus we have an anti-homomorphism:
$$
\alpha:\widehat{G}/\{\pm 1\} \to \LieO(\Sym_2(n,{\Bbb Z})).
$$

\begin{thm}[{\cite[Thm. 6.16, Prop. 6.19]{YY}}]
Let
$\Phi \in \Eq_0(\bl{D}(Y),\bl{D}(X))$ be an equivalence.
\begin{enumerate}
\item[(1)]
$v_1:=v(\Phi({\cal O}_Y))$ and $v_2:=\Phi(\varrho_Y)$
are positive isotropic Mukai vectors with 
$\langle v_1,v_2 \rangle=-1$ and we can write 
\begin{equation}
\begin{split}
& v_1=(p_1^2 r_1,p_1 q_1 H_Y, q_1^2 r_2),\quad
v_2=(p_2^2 r_2,p_2 q_2 H_Y, q_2^2 r_1),\\
& p_1,q_1,p_2,q_2, r_1, r_2 \in {\Bbb Z},\;\;
p_1,r_1,r_2 >0,\\
& r_1 r_2=n,\;\; p_1 q_2 r_1-p_2 q_1 r_2=1.
\end{split}
\end{equation}
\item[(2)]
We set
\begin{equation}
\theta(\Phi):=\pm
\begin{pmatrix}
p_1 \sqrt{r_1} & q_1 \sqrt{r_2}\\
p_2 \sqrt{r_2} & q_2 \sqrt{r_1}
\end{pmatrix}
\in G/\{\pm 1\}.
\end{equation}
Then $\theta(\Phi)$ is uniquely determined by 
$\Phi$ and  
we have a map
\begin{equation}
\theta:{\cal E} \to G/\{\pm 1\}.
\end{equation}
\item[(3)]
The action of $\theta(\Phi)$ on $\Sym_2(n,{\Bbb Z})$
is the action of $\Phi$ on $\Sym_2(n,{\Bbb Z})$: 
\begin{equation}
\iota_X \circ \Phi(v)
=\iota_Y(v)\cdot 
\theta(\Phi).
\end{equation}
\begin{NB}
\begin{equation}
\iota_X \Phi\iota_Y^{-1}
\begin{pmatrix}r &d  \sqrt{n}\\d  \sqrt{n}&a \end{pmatrix} 
=
{}^t (\theta(\Phi))
\begin{pmatrix}r &d  \sqrt{n}\\d  \sqrt{n}&a \end{pmatrix} 
\theta(\Phi)
\end{equation}
\end{NB}
Thus we have the following commutative diagram:
\begin{align}
\label{diag:groups}
\xymatrix{    
{\cal E}  \ar[d]_-{\theta} \ar[dr]_{\eta} 
&
\\
\widehat{G}/\{\pm 1\} \ar[r]_-{\alpha}
&
\LieO(\Sym_2(n,{\Bbb Z}))   
}
\end{align}

\end{enumerate}
\end{thm}
From now on, we identify the Mukai lattice
$H^*(X,{\Bbb Z})_{\alg}$ with $\Sym_2(n,{\Bbb Z})$
via $\iota_X$. 
Then for $g \in \widehat{G}$ and $v \in H^*(X,{\Bbb Z})_{\alg}$,
$v \cdot g$ means
$\iota_X(v \cdot g)=\iota_X(v) \cdot g$.    

\begin{NB}
Then the action of the cohomological FMT 
$\Phi^H : H^{\ev}(Y,\bb{Z})_{\alg}\to H^{\ev}(X,\bb{Z})_{\alg}$ 
can be written as follows (\cite[sect. 6.4]{YY}). 
Let $\widehat{H}$ be the ample generator of $\NS(Y)$. 
For a Mukai vector $v=(r,d\widehat{H},a)\in H^{\ev}(Y,\bb{Z})_{\alg}$, 
the image $\Phi^H(v)$ is given by
\begin{align*}
\Phi^H(v)=v\cdot g,\quad 
g := \begin{pmatrix}
p_1\sqrt{r_1}&q_1\sqrt{r_2}\\p_2\sqrt{r_2}&q_2\sqrt{r_1}
\end{pmatrix}\in G.
\end{align*}
Here the action $\cdot$ is given by
\begin{align}
\label{eq:action_cdot}
 (r,d H,a)\cdot g := (r',d' H,a'),\quad
 {}^t g
 \begin{pmatrix}r &d  \sqrt{n}\\d  \sqrt{n}&a \end{pmatrix} 
 g
=\begin{pmatrix}r'&d' \sqrt{n}\\d' \sqrt{n}&a'\end{pmatrix}.
\end{align}
\end{NB}

We also need to treat the composition of a Fourier-Mukai transform 
and the dualizing functor ${\cal D}_X$.
For a
Fourier-Mukai transform
 $\Phi_{X \to Y}^{\bl{E}^\bullet} \in \Eq(\bl{D}(X),\bl{D}(Y))$,
we set
\begin{align*}
\theta(\Phi_{X \to Y}^{\bl{E}}\cal{D}_X)
:=
\begin{pmatrix}
1 & 0\\
0 & -1 
\end{pmatrix}
\theta(\Phi_{X \to Y}^{\bl{E}})
\in \widehat{G}/\{\pm 1\}.
\end{align*}
Then the action of 
$\theta(\Phi_{X \to Y}^{\bl{E}}\cal{D}_X)$ on
$\Sym_2({\Bbb Z},n)$ is the same as the action of
$\Phi_{X \to Y}^{\bl{E}}\cal{D}_X$.

\begin{lem}[{\cite[Lemma 6.18]{YY}}]
\label{fct:matrix}
%
If $\theta(\Phi_{X \to Y}^{\bl{E}})=
\begin{pmatrix} a & b\\ c & d \end{pmatrix}$,
then 
\begin{align*}
\theta(\Phi_{Y \to X}^{\bl{E}})=
\pm\begin{pmatrix}d & b\\c & a \end{pmatrix},\quad
\theta(\Phi_{Y \to X}^{\bl{E}^{\vee}[2]})=
\pm\begin{pmatrix}d & -b\\-c & a \end{pmatrix},\quad
\theta(\Phi_{X \to Y}^{\bl{E}^{\vee}[2]})=
\pm\begin{pmatrix}a & -b\\-c & d \end{pmatrix}.
\end{align*}
\end{lem}


\subsection{The arithmetic group $G$ and numerical solutions 
for the ideal sheaf}\label{subsect:solution}

Let $\ell \in \bb{Z}_{>0}$.
We assume that $\sqrt{\ell n}\notin\bb{Z}$. 
Our next task is to describe the numerical solution 
of the ideal sheaf of 0-dimensional subscheme. 
First we introduce an arithmetic group $S_{n,\ell}$.

\begin{defn}
For $(x,y)\in \bb{R}^2$, set
\begin{align*}
P(x,y):=\begin{pmatrix}y&\ell x\\ x&y\end{pmatrix}.
\end{align*}
We also set
\begin{align*}
S_{n,\ell}:=
\left\{\begin{pmatrix}y&\ell x\\ x&y\end{pmatrix}
\,\bigg|\,
\begin{aligned}
x=a \sqrt{r}, y=b \sqrt{s},\;
a,b,r, s \in\bb{Z}\\
r,s>0,\; rs=n,\; 
y^2-\ell x^2=\pm 1 
\end{aligned}
\right\}.
\end{align*}
\end{defn}

\begin{lem}
\begin{enumerate}
\item[(1)]
$S_{n,\ell}$ is a commutative subgroup of $\GL(2,\bb{R})$.
\item[(2)]
We have a homomorphism
\begin{align*}
\begin{array}{c c c c}
\phi \colon & S_{n,\ell} & \longrightarrow & \bb{R}^\times \\
            & P(x,y)     & \mapsto         & y+x\sqrt{\ell}.
\end{array}
\end{align*}
\item[(3)]
For $\ell>1$, $\phi$ is injective.  
For $\ell=1$, we have
\begin{align*}
\Ker\phi=\left\langle\begin{pmatrix} 0&1 \\ 1&0 \end{pmatrix}\right\rangle.
\end{align*}
\item[(4)]
We set a subgroup $G_{n,\ell}$ of $\widehat{G}$ 
(Definition~\ref{defn:G}) to be 
\begin{align*}
G_{n,\ell}:= 
\left\{
 g \in \widehat{G} 
 \left|\; 
 {}^{t} g \begin{pmatrix} 1 & 0 \\ 0 & -\ell \end{pmatrix}  g
 =\pm\begin{pmatrix} 1 & 0 \\ 0 & -\ell \end{pmatrix}
 \right.
\right\}.
\end{align*}
Then
\begin{align*}
G_{n,\ell}
=S_{n,\ell} \rtimes 
\left \langle
 \begin{pmatrix} 1 & 0 \\ 0 & -1 \end{pmatrix}
\right\rangle.
\end{align*}
\end{enumerate}
\end{lem}

\begin{proof}
The proofs of (1) and (2) are straightforward.

For (3), assume that 
$x,y\in \bb{R}$ with $x^2,y^2,x y/\sqrt{n}\in \bb{Q}$ 
satisfy $y+x\sqrt{\ell}=1$. 
Then $(y^2+\ell x^2)+2(x y/\sqrt{n})\sqrt{\ell n}=(y+x\sqrt{\ell})^2=1$. 
Our assumptions yields $y^2+\ell x^2=1$ and $x y=0$. If $x=0$, then $y=\pm 1$. 
If $y=0$, then $\ell=1$ and $x=1$. 
Hence the conclusion holds.

(4) follows from direct computations.
\end{proof}

Then the Dirichlet unit theorem yields the following corollary.

\begin{cor}
If $\ell>1$, then $S_{n,\ell}\cong\bb{Z}\oplus\bb{Z}/2\bb{Z}$.
\end{cor}
\begin{proof}
Let $p_1,\ldots,p_m$ be the prime divisors of $\ell n$ and $\frk{o}$ be
the ring of algebraic integers in $\bb{Q}(\sqrt{p_1},\ldots,\sqrt{p_m})$. 
By Dirichlet unit theorem $\frk{o}^\times$ 
is a finitely generated abelian group whose torsion subgroup is $\{\pm 1\}$. 
Hence $\phi(S_{n,\ell})$ is a finitely generated abelian group 
whose torsion subgroup is $\{\pm 1\}$. 
For $A\in S_{n,\ell}$, we have $\phi(A^2)\in\bb{Z}[\sqrt{\ell n}]$. 
Since $\bb{Z}[\sqrt{n\ell}]^\times\cong\bb{Z}\oplus\bb{Z}/2\bb{Z}$, 
we get $S_{n,\ell}\cong\bb{Z}\oplus\bb{Z}/2\bb{Z}$.
\end{proof}

\begin{rem}
If $n=1$ and $\ell>1$, 
then $S_{1,\ell}$ is the group of units of $\bb{Z}[\sqrt{\ell}]$. 
Moreover if $\ell$ is square free and $\ell\equiv 2,3 \pmod{4}$, 
then since $\bb{Z}[\sqrt{\ell}]$ is 
the ring of the integers of $\bb{Q}[\sqrt{\ell}]$, 
a generator of $S_{1,\ell}$ becomes a fundamental unit.
\end{rem}

\begin{lem}\label{lem:numerical-P(p,q)}
For two positive isotropic Mukai vectors $w_0,w_1$ 
on the fixed abelian surface $X$, 
the condition
\begin{align*}
(1,0,-\ell)=\pm (\ell w_0-w_1),\quad \mpr{w_0,w_1}=-1
\end{align*}
is equivalent to
\begin{align*}
w_0=(p^2,-\tfrac{p q}{\sqrt{n}}H,q^2),\ 
w_1=(q^2,-\tfrac{\ell p q}{\sqrt{n}}H,\ell^2 p^2),\quad
P(p,q)\in S_{n,\ell}.
\end{align*}
\end{lem}

\begin{proof}
If there are isotropic Mukai vectors with the first condition, 
then we can write them as $w_0=(r,d H,(r\ell\pm 1))$ and 
$w_1=(r \ell\mp 1,d\ell H, r \ell^2)$, 
where $d^2(H^2)=2 r (r \ell\mp 1)$. 
We set $p:= \sqrt{r}$ and $q:= \sqrt{r \ell \mp 1}$.
Then $w_0=(p^2,-\tfrac{p q}{\sqrt{n}}H,q^2)$,
$w_1=(q^2,-\tfrac{\ell p q}{\sqrt{n}}H,\ell^2 p^2)$
and $\ell p^2-q^2=\pm 1$.
Thus $P(p,q)\in S_{n,\ell}$. 
The converse is obvious.
\end{proof}

\begin{cor}
Recall the action $\cdot$ of $\GL(2,\bb{R})$ given in \eqref{eq:action_cdot}.
By the correspondence 
\begin{align*}
G_{n,\ell} \ni g \mapsto (w_0,w_1),\quad
w_0:=(0,0,1)\cdot g, \ w_1:=(1,0,0)\cdot g, 
\end{align*}
we have a bijective correspondence:
\begin{align*}
\begin{array}{c c c}
G_{n,\ell}
\left/
 \left\langle \pm \begin{pmatrix} 1 & 0\\ 0 & -1 \end{pmatrix} \right \rangle 
\right.
 \cong 
S_{n,\ell}/\{\pm 1\} 
& 
\longleftrightarrow 
&
\left\{ (w_0,w_1) 
 \left|
 \begin{aligned}
 \mpr{w_0,w_1}=-1,\, \mpr{ w_0^2 }=\mpr{w_1^2}=0,\\
 w_0,w_1>0,\, (1,0,-\ell)=\pm(\ell w_0-w_1)
 \end{aligned}
 \right. 
\right\}
\end{array}
\end{align*}
\end{cor}

\begin{defn}
\label{defn:ambm}
Assume that $\ell>1$.
Let 
$$
A_\ell:=
\begin{pmatrix}
q & \ell p\\
p & q
\end{pmatrix},\;p,q>0
$$
be the generator of $S_{n,\ell}/\{\pm 1\}$.
\begin{NB}
It corresponds to
$z:= q+\sqrt{\ell}p$ ($p,q>0$). 
\end{NB}
We set $\epsilon:= q^2-\ell p^2\in\{\pm 1\}$.
For $m \in \bb{Z}$, we set
\begin{align*}
\begin{pmatrix}   q & \ell p  \\ p   & q   \end{pmatrix}^m= 
\begin{pmatrix} b_m & \ell a_m\\ a_m & b_m \end{pmatrix}.
\end{align*}
\end{defn}

By the definition we have 
\begin{align*}
(a_0,b_0)=(0,1),\quad
(a_{-m},b_{-m})=\ep^m(-a_m,b_m),\, m\in\bb{Z}_{>0}
\end{align*}
and
\begin{align}
\label{eq:S_{n,l}}
S_{n,\ell}= 
\left\{ \left. \pm
 \begin{pmatrix} b_m & \ell a_m\\ a_m & b_m \end{pmatrix}
 \right| m \in \bb{Z} \right\}.
\end{align}
Next we consider the right action of $\GL(2,\bb{R})$ on $\bb{R}^2$
\begin{align}
\label{eq:quad_action}
(x,y) \mapsto (x,y) X,
\quad X \in \GL(2,\bb{R}).
\end{align}
Then the quadratic map 
\begin{align*}
\begin{array}{c c c}
 \bb{R}^2 & \to     & \Sym_2(\bb{R})\\[3pt]
 (x,y)  & \mapsto &
\begin{pmatrix} x \\ y \end{pmatrix}
\begin{pmatrix} x & y  \end{pmatrix}
\end{array}
\end{align*}
is $\GL(2,\bb{R})$-equivariant.   
Using this action, we have the next descriptions of
the topological invariants of fine moduli spaces
$M_H(v)$ of dimension 2.
\begin{align}
\label{eq:top-inv}
\begin{split}
& 
\left\{ v 
 \left|
  \begin{aligned}
   v \in H^*(X,\bb{Z})_{\alg},\, \mpr{v^2}=0,\, v>0,
   \\
   \mpr{w,v}=-1,\,\exists w \in H^*(X,\bb{Z})_{\alg}
  \end{aligned}
 \right. 
\right\} 
\\ 
\overset{\varphi_1}{\longleftrightarrow} 
&
\left\{ 
 (a\sqrt{r}, b\sqrt{s}) \in \bb{R}^2
 \left|
  \begin{aligned}
  a,b\in\bb{Z},\, r,s\in\bb{Z}_{>0},
  \\ 
  r s=n,\, \gcd(a r,b s)=1 
  \end{aligned}
 \right.
\right\}
/\{\pm 1\}
=
\left\{  (0,1) X \mid  X \in G \right\}/\{\pm 1\}
\\
\overset{\varphi_2}{\longleftrightarrow} 
&
\left\{
 \left. \dfrac{b \sqrt{n}}{a r} \in {\bb{P}}^1(\bb{R}) 
=\bb{R} \cup \{ \infty \} 
 \right| 
 r s=n,\; \gcd(a r,b s)=1
\right\},
\end{split}
\end{align}
where we used the correspondences 
\begin{align*}
v=(a^2 r, a b H,b^2 s) 
\overset{\varphi_1}{\longleftrightarrow}
 \pm(a\sqrt{r},b\sqrt{s})
\overset{\varphi_2}{\longleftrightarrow}
 \dfrac{\mu(v)}{2\sqrt{n}}=\dfrac{b \sqrt{n}}{a r}.
\end{align*}
Here we used the slope for the Mukai vector 
defined by $\mu(v):=(H,c_1(v))/\rk v$.
These correspondences are $\widehat{G}$-equivariant 
under the action \eqref{eq:quad_action}. 
Lemma \ref{lem:numerical-P(p,q)}, \eqref{eq:S_{n,l}} and \eqref{eq:top-inv}
imply the following one to one correspondence:
\begin{align*}
\begin{array}{c c c}
\left\{ 
 \{ v_1,v_2 \} \left| 
 \begin{aligned}
  &\text{ There is a numerical solution}
  \\
  &\text{ $(v_1,v_2,\ell_1,\ell_2)$ of $(1,0,-\ell)$ }
 \end{aligned}
 \right. 
\right\} 
& \longleftrightarrow & 
\left\{ 
 \left. 
  \left\{\frac{b_m}{a_m}, \frac{\ell a_m}{b_m} \right\} 
\subset {\bb{P}}^1(\bb{R}) 
 \right| 
 m \in \bb{Z} 
\right\}
\\
\{ v_1,v_2 \} 
& \longleftrightarrow & 
\left\{\frac{\mu(v_1)}{2\sqrt{n}},\frac{\mu(v_2)}{2\sqrt{n}} \right\},
\end{array}
\end{align*}
where $(\ell_i,\ell_j)=(\ell,1)$ if and only if 
$(\frac{\mu(v_i)}{2\sqrt{n}},\frac{\mu(v_j)}{2\sqrt{n}})=
(\frac{b_m}{a_m},\frac{\ell a_m}{b_m})$.

\begin{defn}
For $m \in {\Bbb Z}$,
we set
\begin{equation}
\begin{split}
u_m:=& a_m^2 e^{\frac{b_m}{a_m \sqrt{n}}H}=
(a_m^2,\tfrac{a_m b_m}{\sqrt{n}}H,b_m^2),\\
u_m':=& b_m^2 e^{\frac{\ell a_m}{b_m \sqrt{n}}H}=
(b_m^2,\tfrac{\ell a_m b_m}{\sqrt{n}}H,\ell^2 a_m^2).
\end{split}
\end{equation}
\end{defn}

\begin{NB}
$u_0=(0,0,1)=\varrho_X$ and $u_0'=(1,0,0)$.
\end{NB}

\subsection{Codimension 0 walls and the action of 
$G_{n,\ell}$.}

\begin{defn}
$C_0$ is the wall associated to
the numerical solution
$(1,\varrho_X,1,\ell)$. 
For $m \ne 0$,
let $C_m$ be the wall associated to the numerical solution
$(u_m,u_m',\ell,1)$. 
\end{defn}

\begin{prop}
\begin{enumerate}
\item[(1)]
$C_0$ is the $t$-axis
and $C_m$ $(m \ne 0)$ is the circle defined by
$$
\left(s-\frac{1}{\sqrt{n}}\frac{b_m}{a_m} \right)
\left(s-\frac{1}{\sqrt{n}}\frac{\ell a_m}{b_m}\right)+t^2=0.
$$ 
\item[(2)]
$\{C_m \mid  m \in {\Bbb Z}\}$ is the set of codimension 0 walls. 
\end{enumerate}
\end{prop}

\begin{defn}
\begin{enumerate}
\item[(1)]
For $C_m$ ($m \in {\Bbb Z}$), we define
adjacent chambers $C_m^\pm$ as follows:
\begin{itemize}
\item
$C_m^-$ is surrounded by $C_m^+$ for $m<0$.
\item
$C_0^- \subset \{(s,t) \mid s<0 \}$ and
$C_0^+ \subset \{(s,t) \mid s>0 \}$.
\item
$C_m^+$ is surrounded by $C_m^-$ for $m>0$.
\end{itemize} 
\item[(2)]
Let $M_{C_m^\pm}(v)$ be the moduli of stable objects
$M_{(sH,tH)}(v)$ for $(s,t) \in C_m^\pm$.
\end{enumerate}
\end{defn}
Then
we have
$$
M_{C_m^\pm}(v)
=
\begin{cases}
\frk{M}^{\pm}(u_m,u_m',\ell,1), & \frac{b_m}{a_m}<\frac{\ell a_m}{b_m}\\
\frk{M}^{\pm}(u_m',u_m,1,\ell), & \frac{b_m}{a_m}>\frac{\ell a_m}{b_m}
\end{cases}
$$
for $m<0$
and 
$$
M_{C_m^\pm}(v)
=
\begin{cases}
\frk{M}^{\mp}(u_m,u_m',\ell,1)[-1], & \frac{b_m}{a_m}<\frac{\ell a_m}{b_m}\\
\frk{M}^{\mp}(u_m',u_m,1,\ell)[-1], & \frac{b_m}{a_m}>\frac{\ell a_m}{b_m}
\end{cases}
$$
for $m \geq 0$.
In particular, we have
\begin{equation}
M_{C_0^-}(v)={\frak M}^+(1,\varrho_X,1,\ell)[-1]=
M_H(1,0,-\ell).
\end{equation}

We note that $\Psi$ in Proposition \ref{prop:dual} satisfies
$\Psi^{-1}=\Psi$. By Proposition \ref{prop:dual}, we have the following 
isomorphisms.
\begin{equation*}
\Psi_m:M_{C_m^\pm}(v) \to M_{C_m^\mp}(v).
\end{equation*}
Let $\Phi_{X_1 \to X}^{{\bf E}_m}$ be a Fourier-Mukai
transform such that
$\Phi_{X_1 \to X}^{{\bf E}_m}(\varrho_{X_1})=u_m$
and $\Phi_{X_1 \to X}^{{\bf E}_m}(1)=u_m'$,
where $X_1=M_H(u_m)$.
Then we have 
\begin{equation}
\theta(\Phi_{X_1 \to X}^{{\bf E}_m})=
\begin{pmatrix}
b_m & \ell a_m\\
\epsilon^m a_m & \epsilon^m b_m
\end{pmatrix}
= 
\begin{pmatrix}
1 & 0\\
0 & -1
\end{pmatrix}^{\frac{-1+\epsilon^m}{2}}
\begin{pmatrix}
q & \ell p\\
p & q
\end{pmatrix}^m.
\end{equation}
Thus we get
\begin{equation}
\theta(\Psi_m)=
\begin{pmatrix}
q & \ell p\\
p & q
\end{pmatrix}^{-m}
\begin{pmatrix}
1 & 0\\
0 & -1
\end{pmatrix}
\begin{pmatrix}
q & \ell p\\
p & q
\end{pmatrix}^m.
\end{equation}
We also have
\begin{equation}
\begin{split}
\begin{pmatrix}
q & \ell p\\
p & q
\end{pmatrix}^{m+k}
\theta(\Psi_m)=&
\begin{pmatrix}
q & \ell p\\
p & q
\end{pmatrix}^{m+k}
\begin{pmatrix}
q & \ell p\\
p & q
\end{pmatrix}^{-m}
\begin{pmatrix}
1 & 0\\
0 & -1
\end{pmatrix}
\begin{pmatrix}
q & \ell p\\
p & q
\end{pmatrix}^m\\
=&\begin{pmatrix}
1 & 0\\
0 & -1
\end{pmatrix}
\begin{pmatrix}
q & \ell p\\
p & q
\end{pmatrix}^{m-k}.
\end{split}
\end{equation}
Hence 
$(\Psi_m(u_{m+k}),\Psi(u_{m+k}'))=(u_{m-k},u_{m-k}')$.
Thus we get the following proposition.

\begin{prop}\label{prop:Psi_m}
$\Psi_m(C_{m+k})=C_{m-k}$ and
$\Psi_m(C_{m+k}^\pm)=C_{m-k}^\mp$.
In particular, 
$\Psi_m$ induces an isomorphism
$$
M_{C_{m+k}^\pm}(v) \to M_{C_{m-k}^\mp}(v).
$$
\end{prop}

\begin{rem}\label{rem:2m}
$$
\theta(\Phi_{X_1 \to X}^{{\bf E}_m} \circ 
\Phi_{X \to X_1}^{{\bf E}_m})=
\pm A_\ell^{2m}.
$$
Hence $M_H(u_{2m}) \cong X$.
\end{rem}

\begin{prop}\label{prop:fund}
\begin{enumerate}
\item[(1)]
There are finitely many walls between $C_0$ and $C_{-1}$.
\item[(2)]
By the action of $G_{n,\ell}$, every wall is transformed
to a wall between $C_0$ and $C_{-1}$.
\end{enumerate}
\end{prop}

\begin{proof}
(1)
We set 
$\lambda_0:=-\frac{q}{p \sqrt{n}}$.
Then for $\beta=\lambda_0 H$, there is finitely many walls.
Since every wall between $C_0$ and $C_{-1}$ intersects with
the line $s=\lambda_0$, the claim holds.

(2)
Let $C$ be a wall between $C_m$ and $C_{m-1}$.
Since $\Psi_0(C)$ is a wall between 
$C_{-m}$ and $C_{-m+1}$, we may assume that $m < 0$. 
By Proposition \ref{prop:Psi_m}, we see that 
$\Psi_m(C)$ is a wall
between $C_{m+1}$ and $C_m$.
Therefore $\Psi_{-1}\circ \cdots \circ \Psi_{m-1}\circ \Psi_m(C)$
is a wall between $C_0$ and $C_{-1}$.
\end{proof}

\begin{rem}
$$
\theta(\Psi_{-1}\circ \cdots \circ \Psi_{m-1}\circ \Psi_m)
=
\begin{cases}
\theta(\Psi_{-1})A_\ell^{-m-1},& 2 \not |m,\\
A_\ell^{-m},& 2  |m.
\end{cases}
$$
\begin{NB}
$\theta(\Psi_{m-1} \circ \Psi_m)=A_\ell^2$.
\end{NB}
\end{rem}

\begin{NB}
\begin{defn}
For the numerical solution $(v_1,v_2,\ell_1,\ell_2)$ 
corresponding to $\{\frac{b_{-m}}{a_{-m}},\frac{\ell a_{-m}}{b_{-m}} \}
=\{-\frac{b_m}{a_m},-\frac{\ell a_m}{b_m} \}$,
$\frk{M}^{\pm}_m$ denotes the moduli space
$\frk{M}^{\pm}(v_1,v_2,\ell_1,\ell_2)$. 
$\frk{M}^{\pm}_m[-1]$ denotes the moduli space
of $V^{\bullet}[-1]$, $V^{\bullet} \in \frk{M}^{\pm}_m$. 
\end{defn}
\end{NB}

In Proposition \ref{prop:Psi_m},
we did not specify the correspondence of
complexes. Since $\phi_{(sH,tH)}(v) \mod 2{\Bbb Z}$
is well-defined,
the correspondence is determined up to shift
$[2k]$ ($k \in {\Bbb Z}$).
We next fix the ambiguity of this shift. 
We note that $\theta(\Psi_m)=\theta([1] \circ {\cal D}_X
\circ \Phi_{X \to X}^{{\bf E}_{2m}^{\vee}[1]})$,
where $[1]$ is the shift functor.
Since 
$Z_{(sH,tH)}(v) \in {\Bbb C} \setminus {\Bbb R}_{<0}$,
we take $\phi_{(sH,tH)}(v) \in(-1,1)$
to consider the moduli space
$M_{(sH,tH)}(v)$ as in Definition \ref{defn:phase(v)}.
Then the isomorphisms in Proposition \ref{prop:Psi_m}
are given by the following proposition.

\begin{prop}\label{prop:Psi_m(phase)}
Assume that $m<0$.
$[1] \circ {\cal D}_X
\circ \Phi_{X \to X}^{{\bf E}_{2m}^{\vee}[1]}$
induces isomorphisms
\begin{equation*}
\begin{matrix}
M_{C_{m+k}^\pm}(v) & \to & M_{C_{m-k}^\mp}(v)\\
E & \mapsto & 
(\Phi_{X \to X}^{{\bf E}_{2m}^{\vee}[1]}(E))^{\vee}[1]. 
\end{matrix}
\end{equation*}
\end{prop}

\begin{proof}
Since $b_{2m}^2-\ell a_{2m}^2=1$, we have
$\frac{b_{2m}}{a_{2m}}<\frac{\ell a_{2m}}{b_{2m}}<0$
for $m<0$.
For $(s,t) \in C_{2m}$,
$\phi_{(sH,tH)}(({\bf E}_{2m})_{|X \times \{x \}}) \in (-1,0]$.
Hence 
$\phi_{(sH,tH)}(({\bf E}_{2m})_{|X \times \{x \}}[1])=
\phi_{(sH,tH)}(E)$
for $(s,t) \in C_{2m}$ and $E \in M_{(sH,tH)}(v)$.
For $E \in M_{(sH,tH)}(v)$,
we also have 
$$
\begin{cases}
\phi_{(sH,tH)}(({\bf E}_{2m})_{|X \times \{x \}}[1])>
\phi_{(sH,tH)}(E)>\phi_{(sH,tH)}(({\bf E}_{2m})_{|X \times \{x \}}),\;
& 
\text{$(s,t)$ is outside of $C_{2m}$},\\
\phi_{(sH,tH)}(({\bf E}_{2m})_{|X \times \{x \}}[2])>
\phi_{(sH,tH)}(E)>\phi_{(sH,tH)}(({\bf E}_{2m})_{|X \times \{x \}}[1]),\;
& 
\text{$(s,t)$ is inside of $C_{2m}$}.
\end{cases}
$$
If $(s,t)$ is outside of $C_{2m}$, then
since $\phi_{(\widetilde{sH},\widetilde{tH})}
(\Phi_{X \to X}^{{\bf E}_{2m}^{\vee}[1]}(({\bf E}_{2m})_{|X \times \{x \}}))
=0$, we have
$\phi_{(\widetilde{sH},\widetilde{tH})}
(\Phi_{X \to X}^{{\bf E}_{2m}^{\vee}[1]}(E)) \in (0,1)$.
Hence
$$
\phi_{(-\widetilde{sH},\widetilde{tH})}
((\Phi_{X \to X}^{{\bf E}_{2m}^{\vee}[1]}(E))^{\vee}[1]) \in (0,1).
$$
If $(s,t)$ is inside of $C_{2m}$, then
$\phi_{(\widetilde{sH},\widetilde{tH})}
(\Phi_{X \to X}^{{\bf E}_{2m}^{\vee}[1]}(E)) \in (1,2)$.
Hence
$$
\phi_{(-\widetilde{sH},\widetilde{tH})}
((\Phi_{X \to X}^{{\bf E}_{2m}^{\vee}[1]}(E))^{\vee}[1]) \in (-1,0).
$$
Therefore the claim holds.
\end{proof}

\begin{rem}\label{rem:Psi_m(phase)}
Assume that $m>0$. Then
$[-1] \circ {\cal D}_X
\circ \Phi_{X \to X}^{{\bf E}_{2m}^{\vee}[1]}$
induces isomorphisms
\begin{equation*}
\begin{matrix}
M_{C_{m+k}^\pm}(v) & \to & M_{C_{m-k}^\mp}(v)\\
E & \mapsto & 
(\Phi_{X \to X}^{{\bf E}_{2m}^{\vee}[1]}(E))^{\vee}[-1]. 
\end{matrix}
\end{equation*}
\end{rem}

\begin{NB}
\begin{proof}
Since $r b_{2m}/a_{2m}-d_\beta=b_{2m}/a_{2m}>0$,
we have
$$
\begin{cases}
\phi_{(sH,tH)}(({\bf E}_{2m})_{|X \times \{x \}})>
\phi_{(sH,tH)}(E)>\phi_{(sH,tH)}(({\bf E}_{2m})_{|X \times \{x \}}[-1]),\;
& 
\text{$(s,t)$ is outside of $C_{2m}$},\\
\phi_{(sH,tH)}(({\bf E}_{2m})_{|X \times \{x \}}[-1])>
\phi_{(sH,tH)}(E)>\phi_{(sH,tH)}(({\bf E}_{2m})_{|X \times \{x \}}[-2]),\;
& 
\text{$(s,t)$ is inside of $C_{2m}$}.
\end{cases}
$$
Then we see that
$\phi_{(\widetilde{sH},\widetilde{tH})}
(\Phi_{X \to X}^{{\bf E}_{2m}^{\vee}[1]}(E)) \in (-2,0)$.
Hence
$$
\phi_{(-\widetilde{sH},\widetilde{tH})}
((\Phi_{X \to X}^{{\bf E}_{2m}^{\vee}[1]}(E))^{\vee}[-1]) \in (-1,1).
$$
\end{proof}
\end{NB}

\begin{NB}
It is NB for Proposition \ref{prop:Psi_m(phase)}.

We note that $\theta(\Psi_m)=\theta([1] \circ {\cal D}_X
\circ \Phi_{X \to X}^{{\bf E}_{2m}^{\vee}[1]})$.

\begin{lem}
Assume that $m<0$.
We set $s=\frac{1}{\sqrt{n}}\frac{b_{2m}}{a_{2m}}$.
For $m<k<-m$, 
$M_{C_{m+k}^\pm}(v) \subset {\frak A}_{(sH,tH)}$ and
$\Psi_m$ induces isomorphisms
\begin{equation}
M_{C_{m+k}^\pm}(v) \to M_{C_{m-k}^\mp}(v). 
\end{equation}
\end{lem}

\begin{proof}

We note that $\theta(\Psi_m)=\theta([1] \circ {\cal D}_X
\circ \Phi_{X \to X}^{{\bf E}_{2m}^{\vee}[1]})$.
We have $v(\Phi_{X \to X}^{{\bf E}_{2m}^{\vee}[1]}({\frak k}_x))=
u_{-2m}$.
Hence $\Phi_{X \to X}^{{\bf E}^{\vee}[1]}$ induces an equivalence
${\frak A}_{(sH,tH)} \to {\frak A}_{(-sH,tH)}$.
\end{proof}

We note that $\frac{b_{2m}}{a_{2m}}<\frac{\ell a_{2m}}{b_{2m}}<0$
for $m<0$.
For $(s,t) \in C_{2m}$,
$\phi_{(sH,tH)}(({\bf E}_{2m})_{|X \times \{x \}}) \in (-1,0]$.
Hence 
$\phi_{(sH,tH)}(({\bf E}_{2m})_{|X \times \{x \}}[1])=
\phi_{(sH,tH)}(E)$
for $(s,t) \in C_{2m}$ and $E \in M_{(sH,tH)}(v)$.
For $E \in M_{(sH,tH)}(v)$,
we also have 
$$
\begin{cases}
\phi_{(sH,tH)}(({\bf E}_{2m})_{|X \times \{x \}}[1])>
\phi_{(sH,tH)}(E)>\phi_{(sH,tH)}(({\bf E}_{2m})_{|X \times \{x \}}),\;
& 
\text{$(s,t)$ is outside of $C_{2m}$},\\
\phi_{(sH,tH)}(({\bf E}_{2m})_{|X \times \{x \}}[2])>
\phi_{(sH,tH)}(E)>\phi_{(sH,tH)}(({\bf E}_{2m})_{|X \times \{x \}}[1]),\;
& 
\text{$(s,t)$ is inside of $C_{2m}$}.
\end{cases}
$$

If $(s,t)$ is outside of $C_{2m}$, then
since $\phi_{(\widetilde{sH},\widetilde{tH})}
(\Phi_{X \to X}^{{\bf E}_{2m}^{\vee}[1]}(({\bf E}_{2m})_{|X \times \{x \}}))
=0$, we have
$\phi_{(\widetilde{sH},\widetilde{tH})}
(\Phi_{X \to X}^{{\bf E}_{2m}^{\vee}[1]}(E)) \in (0,1)$.
Therefore 
$$
\phi_{(-\widetilde{sH},\widetilde{tH})}
((\Phi_{X \to X}^{{\bf E}_{2m}^{\vee}[1]}(E))^{\vee}[1]) \in (0,1).
$$

If $(s,t)$ is inside of $C_{2m}$, then
$\phi_{(\widetilde{sH},\widetilde{tH})}
(\Phi_{X \to X}^{{\bf E}_{2m}^{\vee}[1]}(E)) \in (1,2)$.
Therefore 
$$
\phi_{(-\widetilde{sH},\widetilde{tH})}
((\Phi_{X \to X}^{{\bf E}_{2m}^{\vee}[1]}(E))^{\vee}[1]) \in (-1,0).
$$

We note that 
$Z_{(sH,tH)}(v) \in {\Bbb C} \setminus {\Bbb R}_{<0}$.
Hence we take $\phi_{(sH,tH)}(v) \in(-1,1)$.
$\phi_{(sH,tH)}(u_{2m})+1>\phi_{(sH,tH)}(v)>\phi_{(sH,tH)}(u_{2m})$
outside of $C_{2m}$.
Hence 
$\phi_{(\widetilde{sH},\widetilde{tH})}
(\Phi_{X \to X}^{{\bf E}_{2m}^{\vee}[1]}(E)) \in (0,1)$.
Therefore 
$$
\phi_{(-\widetilde{sH},\widetilde{tH})}
((\Phi_{X \to X}^{{\bf E}_{2m}^{\vee}[1]}(E))^{\vee}[1]) \in (0,1).
$$

$\phi_{(sH,tH)}(u_{2m})+2>\phi_{(sH,tH)}(v)>\phi_{(sH,tH)}(u_{2m})+1$
insides of $C_{2m}$.

\end{NB}

\begin{prop}
We set
$$
M_m:=M_{C_{m-1}^+}(v) \cap M_{C_m^-}(v).
$$
\begin{enumerate}
\item[(1)]
$M_m \ne \emptyset$ and $M_m$ is birationally equivalent to
$M_{C_{m-1}^+}(v)$ and $M_{C_m^-}(v)$.
 \item[(2)]
We have a sequence of isomorphisms
$$
\cdots \overset{\Psi_{-3}}{\to} M_{-2}
\overset{\Psi_{-2}}{\to} M_{-1}
\overset{\Psi_{-1}}{\to} M_0
\overset{\Psi_{0}}{\to} M_{1}
\overset{\Psi_{1}}{\to} M_{2} 
\overset{\Psi_{2}}{\to} \cdots.
$$
\end{enumerate}
\end{prop}
\begin{proof}
(1)
Since there is no codimension 0 wall between 
$C_{m-1}$ and $C_m$, $M_m \ne \emptyset$.
Since $M_{C_{m-1}^+}(v)$ and $M_{C_m^-}(v)$ are irreducible,
$M_m$ is birationally equivalent to
$M_{C_{m-1}^+}(v)$ and $M_{C_m^-}(v)$.

(2) 
By Proposition \ref{prop:Psi_m}
or Proposition \ref{prop:Psi_m(phase)} 
(and Remark \ref{rem:Psi_m(phase)}), we have isomorphisms
\begin{equation}
\begin{split}
\Psi_m:M_{C_{m-1}^+}(v) & \to M_{C_{m+1}^-}(v),\\
\Psi_m:M_{C_{m}^-}(v) & \to M_{C_m^+}(v).
\end{split}
\end{equation}
Hence we have an isomorphism
$$
\Psi_m:M_{C_{m-1}^+}(v) \cap M_{C_m^-}(v) \to
M_{C_m^+}(v) \cap M_{C_{m+1}^-}(v).
$$
Thus the claim holds.
\end{proof}

We note that
$M_0$ is an open subset of 
$M_{C_0^-}(v)=\{I_Z \otimes L \mid 
I_Z \in \Hilb{\ell}{X}, L \in \Piczero{X}\}$.
Hence  
we have two semi-homogeneous presentations of $I_Z \otimes L
\in M_0$:
\begin{equation}
0 \to I_Z \otimes L \to L \to {\cal O}_Z \to 0 
\end{equation}
and
\begin{equation}
0 \to E_{-1} \to E_0 \to I_Z \otimes L \to 0. 
\end{equation}

Starting from these two semi-homogeneous presentations,
we have a sequence of complexes $F_m^\bullet \in M_m$ 
such that
$F_0^{\bullet}=I_Z \otimes L$ and
$\Psi_m(F_m^\bullet)=F_{m+1}^\bullet$. 
Then we have exact triangles
\begin{equation}
\begin{split}
& V_{m-1}^+ \to F_m^{\bullet} \to V_{m-1}^- \to V_{m-1}^+[1]\\
& W_{m}^- \to F_m^{\bullet} \to W_{m}^+ \to W_{m}^-[1]\\
\end{split}
\end{equation}
such that 
\begin{itemize}
\item
$(W_0^-,W_0^+)=({\cal O}_Z[-1],L)$,
$(V_{-1}^+,V_{-1}^-)=(E_0,E_{-1}[1])$,
\item
$\Psi_m(W_m^\pm)=V_m^\pm$,
$\Psi_m(V_{m-1}^\pm)=W_{m+1}^\pm$,
\item
for $(s,t) \in C_{m-1}$,
$V_{m-1}^\pm$ are $\sigma_{(sH,tH)}$-semi-stable
objects with the same phase
and define the wall $C_{m-1}$,
\item
for $(s,t) \in C_{m}$,
$W_{m}^\pm$ are $\sigma_{(sH,tH)}$-semi-stable
objects with the same phase and 
define wall $C_m$.
\end{itemize}

\begin{NB}
For simplicity, assume that $m<0$.
Then
we also see that
$W_m^-[-1], V_m^-[-1] \in {\cal M}_H(u)$
$V_m$ and $W_m^\mp$ are semi-homogeneous sheaves

Then $W_m^-[-1]$ and $W_m^+$ are semi-homogeneous
sheaves with 
$\{v(W_m^-[-1]),v(W_m^+)\}=\{ \ell u_m,u_m'\}$.
Since $\Psi_m$ is determined by $\{ u_m,u_m'\}$,
\end{NB}

\begin{NB}
Let $E$ be be an object of $M_{(\beta_1,\omega_1)}(v) \cap
 M_{(\beta_2,\omega_2)}(v)$.
Let $\Phi$ be a contravariant Fourier-Mukai transform.
Assume that $d_{\widetilde{\beta_1}}(\Phi(E))>0$
and $d_{\widetilde{\beta_2}}(\Phi(E))>0$.
If $\Phi(E)$ is a 
$\sigma_{(\widetilde{\beta_1},\widetilde{\omega_1})}$-semi-stable
object of 
${\frak A}_{(\widetilde{\beta_1},\widetilde{\omega_1})}$,
then $\Phi(E)$ is also
a $\sigma_{(\widetilde{\beta_2},\widetilde{\omega_2})}$-semi-stable
object of 
${\frak A}_{(\widetilde{\beta_2},\widetilde{\omega_2})}$.

\begin{proof}
Since $d_{\widetilde{\beta_2}}(\Phi(E))>0$,
$\phi_{(\widetilde{\beta_2},\widetilde{\omega_2})}(\Phi(E))
\in (2k,2k+1)$.
By our assumption, we have $H^i(\Phi(E))=0$ for $i \ne -1,0$.
Hence $k=0$ and $\Phi(E) \in 
{\frak A}_{(\widetilde{\beta_2},\widetilde{\omega_2})}$.
\end{proof}

Assume that $E \in M_{C_{m-1}^+}(v) \cap M_{C_m^-}(v)$.
Since $E \in M_{C_{m-1}}(v)$, we have an exact sequence
$$
0 \to A \to E \to B \to 0
$$
such that $A$ and $B$ are $\sigma_{(\beta,\omega)}$-semi-stable
objects with 
$\phi_{(\beta,\omega)}(A)=\phi_{(\beta,\omega)}(B)$.
Then we have an exact sequence
$$
0 \to \Psi_m(B) \to \Psi_m(E) \to \Psi_m(A) \to 0
$$
in ${\frak A}_{(\widetilde{\beta},\widetilde{\omega})}$.

\end{NB}

\subsection{Fourier-Mukai transforms of the families
$F_m^{\bullet}$.}

We first assume $\ep=q^2-\ell p^2=-1$.
In this case, the algebraic integers $a_m,b_m$ 
in Definition~\ref{defn:ambm} satisfy the following relations:
\begin{align*}
 \dfrac{b_{2 k - 1}}{a_{2 k - 1}}<\dfrac{\ell a_{2 k}}{b_{2 k}}
<\dfrac{b_{2 k + 1}}{a_{2 k + 1}}<\sqrt{\ell}
<\dfrac{\ell a_{2 k + 1}}{b_{2 k + 1}}<\dfrac{b_{2 k}}{a_{2 k}}
<\dfrac{\ell a_{2 k - 1}}{b_{2 k - 1}}
\quad
(k\in\bb{Z}_{>0}),
\quad
\lim_{k\to\infty}\dfrac{b_k}{a_k}=\lim_{k\to\infty} \dfrac{\ell a_k}{b_k}=
\sqrt{\ell}.
\end{align*}
Thus $\pm \sqrt{\frac{\ell}{n}}$ are the accumulation points
of $\cup_m C_m$.

\begin{NB}
Set $z_m := b_m + \sqrt{\ell} a_m$ and 
$\overline{z}_m :=b_m - \sqrt{\ell} a_m$.
Then
$z=z_1=q+\sqrt{\ell}p$, 
$\overline{z}_1=q-\sqrt{\ell}p$, and 
\begin{align*}
&z_{m} = z_{m-1} z_1=z_1^m,\quad
 \overline{z}_m = \overline{z}_{m-1} \overline{z}_1 
 =\overline{z}_1^m
\quad
 (m \in \bb{Z}_{\ge0}).
\end{align*}
If $\ep=q^2-\ell p^2=-1$,
then $\overline{z}_1=q-\sqrt{\ell}p<0$.
Thus $\overline{z}_m<0$ if $m$ is odd, 
and $\overline{z}_m>0$ if $m$ is even.
Thus
\begin{align*}
\dfrac{b_{2k-1}}{a_{2k-1}}<\sqrt{\ell}<\dfrac{\ell a_{2k-1}}{b_{2k-1}},
\quad
\dfrac{\ell a_{2k}}{b_{2k}}<\sqrt{\ell}<\dfrac{b_{2k}}{a_{2k}}
\quad
(k \in \bb{Z}_{>0}).
\end{align*}
We also have
\begin{align*}
&\ell a_m a_{m-1}-b_m b_{m-1}
=(q^2-\ell p^2) (\ell a_{m-1}a_{m-2}-b_{m-1}b_{m-2})
\\
&\Longleftrightarrow
\dfrac{\ell a_m}{b_m}-\dfrac{b_{m-1}}{a_{m-1}}
=\dfrac{a_{m-2} b_{m-1}}{a_{m-1} b_m}(q^2-\ell p^2)
 (\dfrac{\ell a_{m-1}}{b_{m-1}}-\dfrac{b_{m-2}}{a_{m-2}})
\\
&\Longleftrightarrow
\dfrac{\ell a_{m}}{b_{m}}-\dfrac{b_{m+1}}{a_{m+1}}
=\dfrac{a_{m} b_{m-1}}{a_{m+1} b_{m}}(q^2-\ell p^2)
 (\dfrac{\ell a_{m-1}}{b_{m-1}}-\dfrac{b_{m}}{a_{m}})
\quad
(m\in \bb{Z}_{\ge2}).
\end{align*}
Note that 
$\tfrac{\ell a_1}{b_1}-\tfrac{b_0}{a_0}
=\tfrac{\ell \cdot p}{q}-\tfrac{0}{1}>0$
and
$\tfrac{\ell a_0}{b_0}-\tfrac{b_1}{a_1}
=\tfrac{\ell \cdot 0}{1}-\tfrac{q}{p}>0$.
Therefore if $\ep=q^2-\ell p^2=-1$, then we have
\begin{align*}
& \dfrac{b_{2 k - 1}}{a_{2 k - 1}}
<\dfrac{\ell a_{2 k}}{b_{2 k}}
<\dfrac{b_{2 k + 1}}{a_{2 k + 1}},
 \quad
 \dfrac{\ell a_{2 k + 1}}{b_{2 k + 1}}
<\dfrac{b_{2 k}}{a_{2 k}}
<\dfrac{\ell a_{2 k - 1}}{b_{2 k - 1}}
\quad
(k \in \bb{Z}_{>0}).
\end{align*}
\end{NB}

We regard $(a_m:b_m)$ and $(b_m:\ell a_m)$ 
as elements of ${\bb{P}}^1({\bb{R}})$. 
Then the inhomogeneous coordinates of these points give a sequence
\begin{multline}
\label{eq:interval:case1:050}
 -\infty=\dfrac{b_0}{a_0}
<-\dfrac{\ell p}{q}=\dfrac{\ell a_{-1}}{b_{-1}}
<\dfrac{b_{-2}}{a_{-2}}<\cdots<-\sqrt{\ell}<\cdots
<\dfrac{\ell a_{-2}}{b_{-2}}<\dfrac{b_{-1}}{a_{-1}}=-\dfrac{q}{p}
\\
<\dfrac{\ell a_0}{b_0}=0<\dfrac{b_1}{a_1}=\dfrac{q}{p}
<\dfrac{\ell a_2}{b_2}<\cdots<\sqrt{\ell}<\cdots<\dfrac{b_2}{a_2}
<\dfrac{\ell a_1}{b_1}=\dfrac{\ell p}{q}<\dfrac{b_0}{a_0}=\infty,
\end{multline}
where we write the inhomogeneous coordinate 
of $(0:1)$ as $\infty$ or $-\infty$.

For a Fourier-Mukai transform
$\Phi_{X \to X'}^{\bl{G}^{\vee}}\colon\bl{D}(X) \to \bl{D}(X')$, 
we write $c_1({\bl{G}}_{x'})/\rk {\bl{G}}_{x'}=(\lambda/\sqrt{n}) H$.
If $-\ell p / q<\lambda<-q/p$, then
$\Phi_{X \to X'}^{{\bl{G}^{\vee}}}(F_0^\bullet)$
is not a sheaf for all $F_0^\bullet$.

\begin{defn}
We set
\begin{align*}
\begin{array}{l l l l l l}
I_1&:=&
 [0,\tfrac{b_1}{a_1}) 
 \cup 
 [\tfrac{\ell a_1}{b_1},\infty),
&
I_0&:=&
 [-\infty,-\tfrac{\ell a_1}{b_1}) 
 \cup 
 [-\tfrac{b_1}{a_1},0),
\\[4pt]
I_{2k}&:=& 
 [\tfrac{b_{2k-1}}{a_{2k-1}},\tfrac{\ell a_{2k}}{b_{2k}})
 \cup 
 [\tfrac{b_{2k}}{a_{2k}},\tfrac{\ell a_{2k-1}}{b_{2k-1}}),
&
I_{-2k}&:=&
 [-\tfrac{b_{2k}}{a_{2k}},-\tfrac{\ell a_{2k+1}}{b_{2k+1}})
 \cup 
 [-\tfrac{b_{2k+1}}{a_{2k+1}},-\tfrac{\ell a_{2k}}{b_{2k}}),
\\[4pt]
I_{2k+1}&:=& 
 [\tfrac{\ell a_{2k}}{b_{2k}},\tfrac{b_{2k+1}}{a_{2k+1}})
 \cup 
 [\tfrac{\ell a_{2k+1}}{b_{2k+1}},\tfrac{b_{2k}}{a_{2k}}),
&
I_{-2k+1}&:=&
 [-\tfrac{\ell a_{2k-1}}{b_{2k-1}},-\tfrac{b_{2k}}{a_{2k}})
 \cup 
 [-\tfrac{\ell a_{2k}}{b_{2k}},-\tfrac{b_{2k-1}}{a_{2k-1}}).
\end{array}
\end{align*}
For $I=\coprod_i [s_i, t_i)$, we denote $I^*:=\coprod_i (s_i,t_i]$. 
\end{defn}

By \eqref{eq:interval:case1:050}, we have decompositions 
$\bb{P}^1({\bb{R}}) \setminus\{\pm\sqrt{\ell}\} =
\coprod_{m \in  {\bb{Z}}}I_m=
\coprod_{m \in {\bb{Z}}}I_m^*$.

\begin{thm}\label{thm:csv:1}
(1)
If $\lambda \in I_m$ $(m \leq 0)$, then
$\Phi_{X \to X'}^{{\bl{G}^{\vee}}}(F_m^\bullet)$ 
is a stable sheaf up to shift.

(2)
If $\lambda \in I_m^*$ $(m \leq 0)$, then
$\cal{D}_{X'}\Phi_{X \to X'}^{{\bl{G}^{\vee}}}(F_m^\bullet)=
\Phi_{X \to X'}^{{\bl{G}}[2]}(F_m^{\bullet \vee})$ 
is a stable sheaf up to shift.
\end{thm}

\begin{proof}
(1)
For a small number $t>0$,
$(\lambda,t)$ belongs to the interior of
the annulus bounded by $C_{-m-1}$ and $C_{-m}$.
By Lemma \ref{lem:2-stability},
$F_m^{\bullet}$ is $\sigma_{(\lambda H,tH)}$-semi-stable.   
By Proposition \ref{prop:Phi(C^pm)},
$\Phi_{X \to X'}^{{\bf G}^{\vee}[n]}(F_m^{\bullet})$
is a stable sheaf, where
$n=1$ for $\lambda> -\sqrt{\ell}$ and
$n=2$ for $\lambda<-\sqrt{\ell}$.
The proof of (2) is similar.
\end{proof}
By this theorem,
we have semi-homogeneous presentations
for a general member of $M_H(w)$, $w=\Phi_{X \to X'}^{{\bf G}^{\vee}}(v)$.

\begin{NB}
$(F_m^{\bullet})^{\vee}=F_{-m+1}^{\bullet}$.
\end{NB}

\begin{rem}
\label{fct:sheaf-criterion}
The claim also follows from \cite[Lem. 4.4]{YY}.
\end{rem}

We next 
assume that $\ep=q^2-\ell p^2=1$.
Then the algebraic integers $a_m,b_m$ 
in Definition~\ref{defn:ambm} satisfy 
\begin{align*}
0<\dfrac{b_m}{a_m} -\sqrt{\ell}<\dfrac{b_{m-1}}{a_{m-1}} -\sqrt{\ell}
\end{align*} 
for $m\in\bb{Z}_{>0}$.
We also have the following sequence of inequalities:
\begin{multline*}
 -\infty=\dfrac{b_0}{a_0}
<-\dfrac{q}{p}=\dfrac{b_{-1}}{a_{-1}}
<\dfrac{b_{-2}}{a_{-2}}<\cdots<-\sqrt{\ell}<\cdots
<\dfrac{\ell a_{-2}}{b_{-2}}<\dfrac{\ell a_{-1}}{b_{-1}}=-\dfrac{\ell p}{q}
\\
<\dfrac{\ell a_0}{b_1}=0
<\dfrac{\ell a_1}{b_1}=\dfrac{\ell p}{q}
<\dfrac{\ell a_2}{b_2}
<\cdots<\sqrt{\ell}
<\cdots<\dfrac{b_2}{a_2}<\dfrac{b_1}{a_1}=\dfrac{q}{p}
<\dfrac{b_0}{a_0}=\infty.
\end{multline*}

\begin{NB}
Set $z_m := b_m + \sqrt{\ell} a_m$ and 
$\overline{z}_m :=b_m - \sqrt{\ell} a_m$.
Then
$z=z_1=q+\sqrt{\ell}p$, 
$\overline{z}_1=q-\sqrt{\ell}p$, and 
\begin{align*}
&z_{m} = z_{m-1} z_1=z_1^m,\quad
 \overline{z}_m = \overline{z}_{m-1} \overline{z}_1 
 =\overline{z}_1^m
\quad
 (m \in \bb{Z}_{\ge0}).
\end{align*}
If $\ep=q^2-\ell p^2=+1$,
then $\overline{z}_1=q-\sqrt{\ell}p>0$.
Thus $\overline{z}_m>0$ for any $m\in\bb{Z}_{\ge0}$,
so that
\begin{align*}
\sqrt{\ell}<\dfrac{b_{m}}{a_{m}}
\quad
(m \in \bb{Z}_{>0}).
\end{align*}
We also have
\begin{align*}
& a_m b_{m-1}-a_{m-1}b_m 
=(q^2-\ell p^2) (a_{m-1}b_{m-2}-a_{m-2}b_{m-1})
\\
&\Longleftrightarrow
\dfrac{b_{m-1}}{a_{m-1}}-\dfrac{b_m}{a_m}
=\dfrac{a_{m-2} }{a_m}(q^2-\ell p^2)
(\dfrac{b_{m-2}}{a_{m-2}}-\dfrac{b_{m-1}}{a_{m-1}})
\quad
(m\in \bb{Z}_{\ge2}).
\end{align*}
Note also that 
$\tfrac{b_1}{a_1}-\tfrac{b_2}{a_2}
=\tfrac{q}{p}-\tfrac{q^2+\ell p^2}{2pq}
=\tfrac{q^2-\ell p^2}{2pq}>0$.
Thus we ahve
\begin{align*}
\dfrac{b_m}{a_m}<\dfrac{b_{m-1}}{a_{m-1}}
\quad
(m\in\bb{Z}_{\ge2}).
\end{align*}
\end{NB}

\begin{defn}
We set
\begin{align*}
\begin{array}{l l l l l l}
I_1     &:=&[0,\tfrac{\ell a_1}{b_1}) \cup [\tfrac{b_1}{a_1},\infty),
&
I_0  &:=&[-\infty,-\tfrac{b_1}{a_1}) \cup [-\tfrac{\ell a_1}{b_1},0),
\\[4pt]
I_{m+1}   &:=&[\tfrac{\ell a_m}{b_m},\tfrac{\ell a_{m+1}}{b_{m+1}})
                \cup
                [\tfrac{b_{m+1}}{a_{m+1}},\tfrac{b_{m} }{a_{m} }),
&
I_{-m}&:=&[-\tfrac{b_m}{a_m},-\tfrac{b_{m+1}}{a_{m+1}})
                \cup
               [-\tfrac{\ell a_{m+1}}{b_{m+1}},-\tfrac{\ell a_{m}}{b_{m}})\quad
                m \geq 1.
\end{array}
\end{align*}
\end{defn}
Then we have
$\bb{P}^1({\bb{R}}) \setminus \{\pm\sqrt{\ell}\}=
\coprod_{m \in \bb{Z}} I_m =\coprod_{m \in \bb{Z}} I_m^*$.

\begin{thm}\label{thm:csv:2}
For a Fourier-Mukai transform 
$\Phi_{X \to X'}^{{\bl{G}^{\vee}}}\colon\bl{D}(X) \to \bl{D}(X')$,
we write
$c_1({\bl{G}}_{x'})/\rk {\bl{G}}_{x'}=(\lambda/\sqrt{n}) H$.

(1)
If $\lambda \in I_m$ $(m \leq 0)$, then
$\Phi_{X \to X'}^{\bl{G}^{\vee}}(F_m^\bullet)$ is a stable sheaf
up to shift.

(2)
If $\lambda \in I_m^*$ $(m \leq 0)$, then
$\cal{D}_{X'}\, \Phi_{X \to X'}^{\bl{G}^{\vee}}(F_m^\bullet)=
\Phi_{X \to X'}^{{\bl{G}}[2]}\, \cal{D}_{X}(F_m^\bullet)$ 
is a stable sheaf.
\end{thm}

The proof is based on the calculation 
in this subsection and is similar to that of Theorem \ref{thm:csv:1}. 
We omit the detail.

\section{Examples}\label{sect:example}

Let $X$ be a principally
polarized abelian surface with $\NS(X)={\Bbb Z}H$. 
We shall study walls for $v=1-\ell \varrho_X$.
We note that $n:=\frac{(H^2)}{2}=1$.
By using Corollary \ref{cor:square},
it is easy to see that $s=0$ is the unique wall for $\ell=1$.
So we assume that $\ell \geq 2$.
We use the notations in section \ref{sect:rho=1}.

(1)
Assume that $\ell=2$.
In this case,
$S_{1,2}/\{\pm 1 \}$ is generated by
\begin{equation*}
A_2:=
\begin{pmatrix}
1 & 2\\
1 & 1
\end{pmatrix}.
\end{equation*}
Hence we have a numerical solution $v=2(1,-H,1)-(1,-2H,4)$.
We have $u_{-1}=(1,-H,1)$ and
\begin{NB}
Then $C_{-1}=W_{u_{-1}}$ is the circle in the $z$-plane 
$$
\left |z+\frac{3}{2} \right|=\frac{1}{2}.
$$
\end{NB}
$C_{-1}=W_{u_{-1}}$ is the circle in the $(s,t)$-plane 
$$
\left(s+\frac{3}{2} \right)^2+t^2=\frac{1}{2^2}.
$$
For $s=-1$, we get $c_1(v e^{-sH})=H$.
Thus there is no wall intersecting with $s=-1$.
Then we see that there is no wall between $C_{-1}$ and $C_0:s=0$
(see Figure 1).
Hence we get the following result.
\begin{lem}
$C_n$ $(n \in {\Bbb Z})$ are all the walls for $v=1-2 \varrho_X$.
\end{lem}

\begin{prop}
Let $v$ be a positive and primitive
Mukai vector with $\langle v^2 \rangle=4$.
Then $M_H(v)$ is isomorphic to
$\Hilb{2}{X} \times X$.
\end{prop}

\begin{proof}
There is a Fourier-Mukai transform 
$\Phi_{X \to X}^{{\bf E}^{\vee}}:{\bf D}(X) \to {\bf D}(X)$
such that $\Phi_{X \to X}^{{\bf E}^{\vee}}((1,0,-2))=v$.
Then we have an isomorphism
$M_{(\beta,\omega)}(1,0,-2) \to M_H(v)$,
where $\beta=c_1({\bf E}_{|\{ x\} \times X})/\rk {\bf E}_{|\{ x\} \times X}$
and $(\omega^2) \ll 1$.
By Theorem \ref{thm:B:3-10.3} and
Proposition \ref{prop:fund},
$M_{(\beta,\omega)}(1,0,-2) \cong 
\Hilb{2}{X} \times X$, which implies the claim.
\end{proof}

\begin{figure}[htbp]
\unitlength 0.1in
\begin{picture}( 51.0000, 23.1000)(  6.0000,-32.0000)
%
{\color[named]{Black}{%
\special{pn 8}%
\special{pa 600 2400}%
\special{pa 5600 2400}%
\special{fp}%
\special{sh 1}%
\special{pa 5600 2400}%
\special{pa 5534 2380}%
\special{pa 5548 2400}%
\special{pa 5534 2420}%
\special{pa 5600 2400}%
\special{fp}%
}}%
%
{\color[named]{Black}{%
\special{pn 25}%
\special{ar 2400 2400 800 800  3.1415927  6.2831853}%
}}%
%
{\color[named]{Black}{%
\special{pn 8}%
\special{pa 1738 2849}%
\special{pa 1734 2843}%
\special{pa 1734 2843}%
\special{fp}%
\special{pa 1714 2811}%
\special{pa 1710 2805}%
\special{fp}%
\special{pa 1692 2772}%
\special{pa 1692 2772}%
\special{pa 1690 2769}%
\special{pa 1689 2766}%
\special{pa 1689 2766}%
\special{fp}%
\special{pa 1673 2732}%
\special{pa 1671 2730}%
\special{pa 1670 2727}%
\special{pa 1669 2726}%
\special{fp}%
\special{pa 1655 2692}%
\special{pa 1654 2690}%
\special{pa 1652 2685}%
\special{fp}%
\special{pa 1640 2649}%
\special{pa 1638 2642}%
\special{fp}%
\special{pa 1627 2606}%
\special{pa 1627 2604}%
\special{pa 1625 2599}%
\special{fp}%
\special{pa 1617 2563}%
\special{pa 1617 2562}%
\special{pa 1615 2556}%
\special{pa 1615 2556}%
\special{fp}%
\special{pa 1609 2519}%
\special{pa 1609 2516}%
\special{pa 1608 2513}%
\special{pa 1608 2512}%
\special{fp}%
\special{pa 1604 2475}%
\special{pa 1603 2473}%
\special{pa 1603 2468}%
\special{fp}%
\special{pa 1601 2431}%
\special{pa 1601 2430}%
\special{pa 1600 2427}%
\special{pa 1600 2423}%
\special{fp}%
\special{pa 1600 2386}%
\special{pa 1600 2378}%
\special{fp}%
\special{pa 1602 2341}%
\special{pa 1602 2340}%
\special{pa 1603 2337}%
\special{pa 1603 2334}%
\special{fp}%
\special{pa 1607 2297}%
\special{pa 1607 2297}%
\special{pa 1607 2294}%
\special{pa 1608 2290}%
\special{pa 1608 2290}%
\special{fp}%
\special{pa 1613 2253}%
\special{pa 1615 2248}%
\special{pa 1615 2246}%
\special{fp}%
\special{pa 1623 2210}%
\special{pa 1623 2209}%
\special{pa 1624 2205}%
\special{pa 1625 2202}%
\special{fp}%
\special{pa 1635 2167}%
\special{pa 1636 2164}%
\special{pa 1637 2160}%
\special{pa 1637 2159}%
\special{fp}%
\special{pa 1649 2124}%
\special{pa 1650 2123}%
\special{pa 1651 2119}%
\special{pa 1652 2117}%
\special{fp}%
\special{pa 1666 2083}%
\special{pa 1666 2082}%
\special{pa 1668 2076}%
\special{pa 1668 2076}%
\special{fp}%
\special{pa 1684 2042}%
\special{pa 1686 2040}%
\special{pa 1687 2037}%
\special{pa 1688 2036}%
\special{fp}%
\special{pa 1705 2003}%
\special{pa 1708 1999}%
\special{pa 1709 1996}%
\special{fp}%
\special{pa 1729 1965}%
\special{pa 1729 1965}%
\special{pa 1730 1962}%
\special{pa 1732 1959}%
\special{pa 1733 1958}%
\special{fp}%
\special{pa 1754 1928}%
\special{pa 1755 1927}%
\special{pa 1758 1922}%
\special{fp}%
\special{pa 1781 1893}%
\special{pa 1782 1892}%
\special{pa 1784 1890}%
\special{pa 1786 1887}%
\special{pa 1786 1887}%
\special{fp}%
\special{pa 1811 1859}%
\special{pa 1812 1857}%
\special{pa 1815 1855}%
\special{pa 1816 1854}%
\special{fp}%
\special{pa 1841 1827}%
\special{pa 1843 1826}%
\special{pa 1847 1822}%
\special{fp}%
\special{pa 1874 1797}%
\special{pa 1874 1797}%
\special{pa 1877 1795}%
\special{pa 1880 1792}%
\special{fp}%
\special{pa 1909 1769}%
\special{pa 1911 1767}%
\special{pa 1913 1765}%
\special{pa 1915 1764}%
\special{fp}%
\special{pa 1945 1743}%
\special{pa 1946 1742}%
\special{pa 1948 1740}%
\special{pa 1951 1738}%
\special{fp}%
\special{pa 1982 1718}%
\special{pa 1985 1716}%
\special{pa 1988 1715}%
\special{pa 1989 1714}%
\special{fp}%
\special{pa 2020 1696}%
\special{pa 2022 1695}%
\special{pa 2025 1693}%
\special{pa 2027 1692}%
\special{fp}%
\special{pa 2060 1676}%
\special{pa 2061 1675}%
\special{pa 2067 1673}%
\special{pa 2067 1673}%
\special{fp}%
\special{pa 2101 1658}%
\special{pa 2107 1656}%
\special{pa 2108 1655}%
\special{fp}%
\special{pa 2143 1643}%
\special{pa 2145 1642}%
\special{pa 2150 1640}%
\special{fp}%
\special{pa 2186 1629}%
\special{pa 2192 1627}%
\special{pa 2193 1627}%
\special{fp}%
\special{pa 2229 1619}%
\special{pa 2234 1617}%
\special{pa 2237 1617}%
\special{fp}%
\special{pa 2273 1610}%
\special{pa 2277 1609}%
\special{pa 2280 1609}%
\special{pa 2280 1609}%
\special{fp}%
\special{pa 2317 1604}%
\special{pa 2317 1604}%
\special{pa 2323 1604}%
\special{pa 2324 1604}%
\special{fp}%
\special{pa 2361 1601}%
\special{pa 2369 1601}%
\special{fp}%
\special{pa 2406 1600}%
\special{pa 2414 1600}%
\special{fp}%
\special{pa 2451 1602}%
\special{pa 2458 1602}%
\special{fp}%
\special{pa 2495 1606}%
\special{pa 2496 1606}%
\special{pa 2500 1606}%
\special{pa 2503 1607}%
\special{fp}%
\special{pa 2539 1612}%
\special{pa 2543 1613}%
\special{pa 2546 1613}%
\special{pa 2547 1613}%
\special{fp}%
\special{pa 2583 1621}%
\special{pa 2585 1622}%
\special{pa 2588 1622}%
\special{pa 2590 1623}%
\special{fp}%
\special{pa 2626 1633}%
\special{pa 2633 1635}%
\special{fp}%
\special{pa 2669 1646}%
\special{pa 2674 1648}%
\special{pa 2676 1649}%
\special{fp}%
\special{pa 2710 1663}%
\special{pa 2711 1663}%
\special{pa 2715 1664}%
\special{pa 2717 1665}%
\special{fp}%
\special{pa 2751 1681}%
\special{pa 2751 1681}%
\special{pa 2754 1683}%
\special{pa 2757 1684}%
\special{pa 2757 1684}%
\special{fp}%
\special{pa 2790 1702}%
\special{pa 2792 1703}%
\special{pa 2795 1704}%
\special{pa 2797 1705}%
\special{fp}%
\special{pa 2828 1724}%
\special{pa 2829 1725}%
\special{pa 2835 1729}%
\special{fp}%
\special{pa 2865 1749}%
\special{pa 2871 1753}%
\special{pa 2872 1754}%
\special{fp}%
\special{pa 2901 1776}%
\special{pa 2905 1780}%
\special{pa 2907 1781}%
\special{fp}%
\special{pa 2935 1806}%
\special{pa 2935 1806}%
\special{pa 2938 1808}%
\special{pa 2940 1810}%
\special{pa 2941 1810}%
\special{fp}%
\special{pa 2967 1836}%
\special{pa 2969 1838}%
\special{pa 2972 1840}%
\special{pa 2973 1841}%
\special{fp}%
\special{pa 2998 1868}%
\special{pa 2999 1869}%
\special{pa 3001 1872}%
\special{pa 3003 1874}%
\special{fp}%
\special{pa 3026 1902}%
\special{pa 3027 1903}%
\special{pa 3029 1905}%
\special{pa 3031 1908}%
\special{fp}%
\special{pa 3054 1938}%
\special{pa 3055 1940}%
\special{pa 3056 1943}%
\special{pa 3057 1945}%
\special{fp}%
\special{pa 3078 1975}%
\special{pa 3079 1976}%
\special{pa 3080 1979}%
\special{pa 3082 1982}%
\special{fp}%
\special{pa 3100 2014}%
\special{pa 3100 2014}%
\special{pa 3102 2016}%
\special{pa 3104 2019}%
\special{pa 3104 2020}%
\special{fp}%
\special{pa 3121 2053}%
\special{pa 3122 2055}%
\special{pa 3123 2058}%
\special{pa 3124 2060}%
\special{fp}%
\special{pa 3139 2094}%
\special{pa 3139 2095}%
\special{pa 3141 2098}%
\special{pa 3142 2101}%
\special{fp}%
\special{pa 3155 2135}%
\special{pa 3157 2141}%
\special{pa 3157 2143}%
\special{fp}%
\special{pa 3169 2178}%
\special{pa 3169 2180}%
\special{pa 3171 2185}%
\special{fp}%
\special{pa 3180 2221}%
\special{pa 3180 2222}%
\special{pa 3181 2225}%
\special{pa 3181 2228}%
\special{pa 3181 2229}%
\special{fp}%
\special{pa 3188 2265}%
\special{pa 3189 2267}%
\special{pa 3189 2271}%
\special{pa 3189 2272}%
\special{fp}%
\special{pa 3195 2309}%
\special{pa 3195 2313}%
\special{pa 3196 2316}%
\special{fp}%
\special{pa 3199 2353}%
\special{pa 3199 2361}%
\special{fp}%
\special{pa 3200 2398}%
\special{pa 3200 2406}%
\special{fp}%
\special{pa 3199 2443}%
\special{pa 3199 2447}%
\special{pa 3198 2450}%
\special{pa 3198 2450}%
\special{fp}%
\special{pa 3195 2487}%
\special{pa 3195 2493}%
\special{pa 3194 2495}%
\special{fp}%
\special{pa 3189 2531}%
\special{pa 3189 2533}%
\special{pa 3188 2536}%
\special{pa 3188 2539}%
\special{fp}%
\special{pa 3181 2575}%
\special{pa 3181 2575}%
\special{pa 3180 2578}%
\special{pa 3179 2582}%
\special{pa 3179 2582}%
\special{fp}%
\special{pa 3170 2618}%
\special{pa 3169 2620}%
\special{pa 3168 2624}%
\special{pa 3168 2625}%
\special{fp}%
\special{pa 3156 2661}%
\special{pa 3154 2668}%
\special{fp}%
\special{pa 3140 2703}%
\special{pa 3139 2705}%
\special{pa 3138 2708}%
\special{pa 3138 2710}%
\special{fp}%
\special{pa 3123 2743}%
\special{pa 3122 2745}%
\special{pa 3120 2748}%
\special{pa 3119 2750}%
\special{fp}%
\special{pa 3102 2783}%
\special{pa 3102 2784}%
\special{pa 3100 2786}%
\special{pa 3099 2789}%
\special{pa 3099 2790}%
\special{fp}%
\special{pa 3080 2822}%
\special{pa 3079 2824}%
\special{pa 3077 2827}%
\special{pa 3076 2828}%
\special{fp}%
}}%
\put(49.000,-26.000){\makebox(0,0)[lb]{$O$}}%
\put(57.0000,-25.6000){\makebox(0,0)[lb]{$s$}}%
\put(49.4000,-10.2000){\makebox(0,0)[lb]{$t$}}%
%
{\color[named]{Black}{%
\special{pn 8}%
\special{pa 2400 2500}%
\special{pa 2400 2300}%
\special{fp}%
}}%
\put(23.6000,-27.2000){\makebox(0,0)[lb]{$-3/2$}}%
\put(32.8000,-11.4000){\makebox(0,0)[lb]{$s=-1$}}%
%
{\color[named]{Black}{%
\special{pn 8}%
\special{pa 4000 2500}%
\special{pa 4000 2300}%
\special{fp}%
}}%
%
{\color[named]{Black}{%
\special{pn 8}%
\special{pa 4800 3200}%
\special{pa 4800 1000}%
\special{fp}%
\special{sh 1}%
\special{pa 4800 1000}%
\special{pa 4780 1068}%
\special{pa 4800 1054}%
\special{pa 4820 1068}%
\special{pa 4800 1000}%
\special{fp}%
}}%
%
{\color[named]{Black}{%
\special{pn 8}%
\special{pa 4900 1600}%
\special{pa 4700 1600}%
\special{fp}%
}}%
\put(44.7000,-15.1000){\makebox(0,0)[lb]{$1/2$}}%
%
{\color[named]{Black}{%
\special{pn 25}%
\special{pa 4800 2400}%
\special{pa 4800 1100}%
\special{fp}%
}}%
\put(40.4000,-26.7000){\makebox(0,0)[lb]{$-1/2$}}%
%
{\color[named]{Black}{%
\special{pn 8}%
\special{pa 2400 2400}%
\special{pa 2400 1600}%
\special{dt 0.045}%
}}%
%
{\color[named]{Black}{%
\special{pn 8}%
\special{pa 2400 1600}%
\special{pa 4800 1600}%
\special{dt 0.045}%
}}%
%
{\color[named]{Black}{%
\special{pn 8}%
\special{pa 3200 1100}%
\special{pa 3200 3200}%
\special{dt 0.045}%
}}%
\put(13.2000,-18.2000){\makebox(0,0)[lb]{$W_{u_{-1}}$}}%
\put(48.7000,-21.8000){\makebox(0,0)[lb]{$C_0$}}%
\end{picture}
\caption{Walls for $v=1-2\varrho_X$.}
\end{figure}

(2)
Assume that $\ell=3$.
In this case,
$S_{1,3}/\{\pm 1 \}$ is generated by
\begin{equation*}
A_3:=
\begin{pmatrix}
2 & 3\\
1 & 2
\end{pmatrix}.
\end{equation*}
Hence we have a numerical solution $v=(4,-6H,9)-3(1,-2H,4)$.
We have $u_{-1}=(1,-2H,4)$ and 
\begin{NB}
 $W_{u_{-1}}$ is the circle in $z$-plane defined by
$$
\left|z+\frac{7}{4} \right|=\frac{1}{4}.
$$
\end{NB}
 $W_{u_{-1}}$ is the circle in $(s,t)$-plane defined by
$$
\left(s+\frac{7}{4} \right)^2+t^2=\frac{1}{4^2}.
$$
We set $w_{-1}:=(1,-H,1)$.
Then $v=w_{-1}+(0,H,-4)$ and $w_{-2}$ defines a wall
$W_{w_{-1}}$ while the defining equation is 
\begin{NB}
$$
\left|z+2 \right|=1.
$$
\end{NB}
$$
\left(s+2 \right)^2+t^2=1.
$$

\begin{lem}
$W_{w_{-1}}$ is the unique wall between $C_0$ and $C_{-1}=W_{u_{-1}}$
(see Figure 2).
\end{lem}

\begin{proof}
Since $W_{u_{-1}}$ passes the point $(s,t)=(-2,0)$,
it is sufficient to classify walls for $s=-2$.
Assume that $w e^{2H}=(r,dH,a)$ defines a wall for $s=-2$.
Since $v e^{2H}=(1,2H,1)$,
 we have $d=1$. We set $w':=(v-w)$ and write
$w' e^{2H}=(r',H,a')$.
By the definition of walls,
we have $\langle w^2 \rangle \geq 0, \langle {w'}^2 \rangle \geq 0$
and $\langle w,w' \rangle>0$.
Thus $ra \leq 1$, $r'a'  \leq 1$ and $1-r a'-r' a>0$.
Since $r+r'=1$, we may assume that $r>0$ and $r' \leq 0$.
Since $a+a'=1$, we have
$0<1-r a'-r' a=(2a-1)r+(1-a)$. 
If $a \leq 0$, then $(2a-1)r+(1-a) \leq (2a-1)+(1-a)=a \leq 0$.
Hence $a \geq 1$.
Since $ra \leq 1$, we have $r=a=1$. 
Therefore $w e^{2H}=(1,H,1)$, which implies that
$w=(1,-H,1)$.
\end{proof}

\begin{NB}
Old orgument:
\begin{proof}
For $w=(r,dH,a)$, $c_1(w e^{-sH})=(d-rs)H$.
If $s=-\frac{3}{2}$, then
$$
\{c_1(w e^{-sH})| w \in H^*(X,{\Bbb Z})_{\alg} \}=\frac{1}{2}{\Bbb Z}H.
$$
If $w=(r,dH,a)$ defines a wall for $s=-3/2$, then
$c_1(w e^{-sH})=\frac{1}{2}H, H$.
Since $c_1((v-w)e^{-sH})=\frac{3}{2}H-c_1(w e^{-sH})$, we may assume that
$c_1(w e^{-sH})=\frac{1}{2}H$.
Thus $w=(1+2\lambda,-(1+3\lambda)H,a)$, where
$\lambda \in {\Bbb Z}$.
We set $w':=v-w=(-2\lambda,(1+3\lambda)H,-3-a)$.
Then
$c_1(w e^{-sH})=-((1+s)+(2s+3)\lambda)H$ and
$c_1(w' e^{-sH})=(1+(2s+3)\lambda)H$.
Since $W_w$ contains $W_{u_1}$,
$-((1+s)+(2s+3)\lambda)>0$ and 
$(1+(2s+3)\lambda)>0$ for $-2 \leq s \leq -3/2$.
For $s=-2$, we have $1-\lambda>0$ and $1+\lambda>0$.
Hence $\lambda=0$.
Then $W_w$ is 
$$
\left(s+\frac{a+3}{2} \right)^2+t^2=\left(\frac{a+3}{2} \right)^2-3.
$$
Since $a \leq 1$, $a=0$ if $a+3 \geq 0$.
If $a+3<0$, then $W_w$ is contained in $s>0$.
Therefore $w=(1,-H,1)$.
\end{proof}
\end{NB}
We have 
$u_0=(0,0,1)$, $u_{-1}=(1,-2H,4)$, $u_{-2}=(4^2,-28H,7^2)$
and so on.
We define $w_n \in H^*(X,{\Bbb Z})_{\alg}$ by
$$
w_n:=(a_n^2,a_n b_n H,b_n^2),\;
(a_n,b_n)=(1,-1)A_3^{n-1}.
$$
Thus 
$w_{-1}=(1,-H,1)$, $w_{-2}=(9,-15H,5^2)$ and so on. 
By Proposition \ref{prop:wall:isometry}
and Proposition \ref{prop:fund}, we get the following.
\begin{lem}
$W_{u_n}$ and $W_{w_n}$ are all the walls for $v$.
\end{lem}
Therefore all moduli spaces are isomorphic
to $\Hilb{3}{X} \times X$.

\begin{NB}
By the action of $A$,
the anulus $W_{u_n} \setminus W_{w_{n+1}}$
and 
$W_{w_{n+1}} \setminus W_{u_{n+1}}$
are transformed to
the anulus
$W_{u_{n+1}} \setminus W_{w_{n+2}}$
and 
$W_{w_{n+2}} \setminus W_{u_{n+2}}$
respectively.
We set
$$
B_n:=A^{-n}
\begin{pmatrix}
1 & 0\\
0 & -1 
\end{pmatrix}
A^n.
$$
Then
$\theta(\Psi_n)=B_n$ and 
by the action of $B_n$,
$W_{w_{n}} \setminus W_{u_{n}}$
are transformed to
$W_{u_{n}} \setminus W_{w_{n+1}}$.
\begin{NB2}
\begin{equation*}
A
\begin{pmatrix}
1 & 0\\
0 & -1 
\end{pmatrix}
=\begin{pmatrix}
1 & 0\\
0 & -1 
\end{pmatrix}
A^{-1}.
\end{equation*}

\begin{equation*}
(0,1)A^{n-k} B_n=(0,1)A^{-k}
\begin{pmatrix}
1 & 0\\
0 & -1 
\end{pmatrix}
A^n
=(0,1)
\begin{pmatrix}
1 & 0\\
0 & -1 
\end{pmatrix}
A^{n+k}
=-(0,1)A^{n+k}.
\end{equation*}

\begin{equation*}
(1,-1)A^{n-k} B_n=(1,-1)A^{-k}
\begin{pmatrix}
1 & 0\\
0 & -1 
\end{pmatrix}
A^n
=(1,-1)
\begin{pmatrix}
1 & 0\\
0 & -1 
\end{pmatrix}
A^{n+k}
=(1,-1)A^{n+k-1}.
\end{equation*}

\end{NB2}
\end{NB}

\begin{prop}
Let $v$ be a positive and primitive
Mukai vector with $\langle v^2 \rangle=6$.
Then $M_H(v)$ is isomorphic to
$\Hilb{3}{X} \times X$.
\end{prop}

\begin{proof}
There is a Fourier-Mukai transform 
$\Phi_{X \to X}^{{\bf E}^{\vee}}:{\bf D}(X) \to {\bf D}(X)$
such that $\Phi_{X \to X}^{{\bf E}^{\vee}}((1,0,-3))=v$.
Then we have an isomorphism
$M_{(\beta,\omega)}(1,0,-3) \to M_H(v)$,
where $\beta=c_1({\bf E}_{|\{ x\} \times X})/\rk {\bf E}_{|\{ x\} \times X}$
and $(\omega^2) \ll 1$.
By Theorem \ref{thm:B:3-10.3} and
Proposition \ref{prop:fund},
$M_{(\beta,\omega)}(1,0,-3) \cong 
\Hilb{3}{X} \times X$, which implies the claim.
\end{proof}

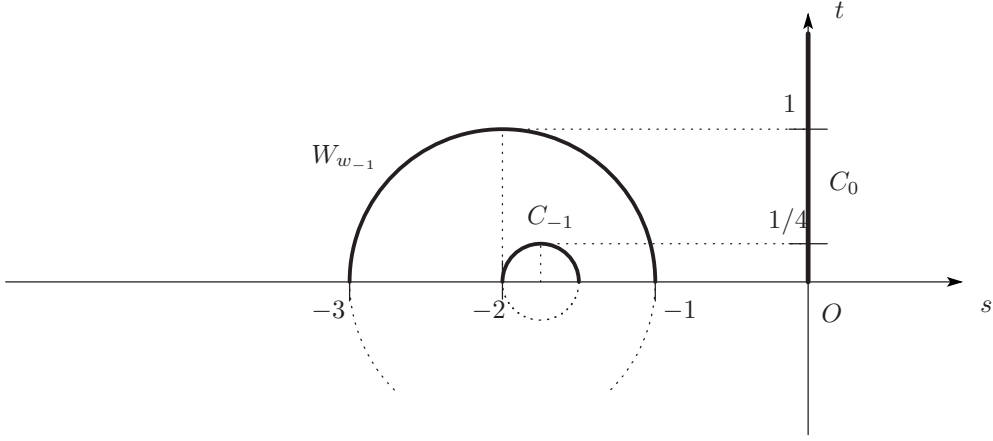
\begin{figure}
\unitlength 0.1in
\begin{picture}( 51.0000, 23.1000)(  6.0000,-32.0000)
%
{\color[named]{Black}{%
\special{pn 8}%
\special{pa 600 2400}%
\special{pa 5600 2400}%
\special{fp}%
\special{sh 1}%
\special{pa 5600 2400}%
\special{pa 5534 2380}%
\special{pa 5548 2400}%
\special{pa 5534 2420}%
\special{pa 5600 2400}%
\special{fp}%
}}%
\put(48.7000,-26.1000){\makebox(0,0)[lb]{$O$}}%
\put(57.0000,-25.6000){\makebox(0,0)[lb]{$s$}}%
\put(49.4000,-10.2000){\makebox(0,0)[lb]{$t$}}%
%
{\color[named]{Black}{%
\special{pn 8}%
\special{pa 3200 2490}%
\special{pa 3200 2290}%
\special{fp}%
}}%
\put(22.000,-26.000){\makebox(0,0)[lb]{$-3$}}%
%
{\color[named]{Black}{%
\special{pn 8}%
\special{pa 4000 2500}%
\special{pa 4000 2300}%
\special{fp}%
}}%
%
{\color[named]{Black}{%
\special{pn 8}%
\special{pa 4800 3200}%
\special{pa 4800 1000}%
\special{fp}%
\special{sh 1}%
\special{pa 4800 1000}%
\special{pa 4780 1068}%
\special{pa 4800 1054}%
\special{pa 4820 1068}%
\special{pa 4800 1000}%
\special{fp}%
}}%
%
{\color[named]{Black}{%
\special{pn 8}%
\special{pa 4900 1600}%
\special{pa 4700 1600}%
\special{fp}%
}}%
\put(46.7000,-15.1000){\makebox(0,0)[lb]{$1$}}%
%
{\color[named]{Black}{%
\special{pn 25}%
\special{pa 4800 2400}%
\special{pa 4800 1100}%
\special{fp}%
}}%
\put(40.4000,-26.000){\makebox(0,0)[lb]{$-1$}}%
\put(22.000,-18.2000){\makebox(0,0)[lb]{$W_{w_{-1}}$}}%
\put(49.1000,-19.4000){\makebox(0,0)[lb]{$C_0$}}%
%
{\color[named]{Black}{%
\special{pn 8}%
\special{pa 2400 2500}%
\special{pa 2400 2300}%
\special{fp}%
}}%
%
{\color[named]{Black}{%
\special{pn 20}%
\special{ar 3200 2400 800 800  3.1415927  6.2831853}%
}}%
%
{\color[named]{Black}{%
\special{pn 8}%
\special{pa 2634 2966}%
\special{pa 2633 2964}%
\special{pa 2629 2960}%
\special{fp}%
\special{pa 2603 2933}%
\special{pa 2603 2933}%
\special{pa 2602 2931}%
\special{pa 2599 2928}%
\special{fp}%
\special{pa 2574 2899}%
\special{pa 2574 2899}%
\special{pa 2573 2897}%
\special{pa 2571 2895}%
\special{pa 2570 2893}%
\special{fp}%
\special{pa 2548 2864}%
\special{pa 2547 2863}%
\special{pa 2546 2861}%
\special{pa 2544 2859}%
\special{pa 2543 2857}%
\special{fp}%
\special{pa 2523 2826}%
\special{pa 2522 2825}%
\special{pa 2520 2821}%
\special{pa 2519 2820}%
\special{fp}%
\special{pa 2501 2788}%
\special{pa 2500 2786}%
\special{pa 2498 2784}%
\special{pa 2497 2781}%
\special{fp}%
\special{pa 2480 2748}%
\special{pa 2479 2746}%
\special{pa 2476 2741}%
\special{fp}%
\special{pa 2462 2707}%
\special{pa 2461 2705}%
\special{pa 2459 2701}%
\special{pa 2459 2700}%
\special{fp}%
\special{pa 2445 2665}%
\special{pa 2445 2663}%
\special{pa 2443 2659}%
\special{pa 2443 2658}%
\special{fp}%
\special{pa 2432 2622}%
\special{pa 2431 2620}%
\special{pa 2430 2618}%
\special{pa 2430 2616}%
\special{pa 2430 2615}%
\special{fp}%
\special{pa 2420 2579}%
\special{pa 2420 2579}%
\special{pa 2420 2577}%
\special{pa 2419 2574}%
\special{pa 2419 2572}%
\special{pa 2419 2572}%
\special{fp}%
\special{pa 2412 2536}%
\special{pa 2412 2535}%
\special{pa 2411 2533}%
\special{pa 2411 2530}%
\special{pa 2410 2528}%
\special{fp}%
\special{pa 2405 2492}%
\special{pa 2405 2486}%
\special{pa 2404 2484}%
\special{fp}%
\special{pa 2401 2447}%
\special{pa 2401 2439}%
\special{fp}%
\special{pa 2400 2402}%
\special{pa 2400 2394}%
\special{fp}%
\special{pa 2401 2357}%
\special{pa 2401 2353}%
\special{pa 2402 2351}%
\special{pa 2402 2349}%
\special{fp}%
\special{pa 2405 2312}%
\special{pa 2405 2309}%
\special{pa 2406 2306}%
\special{pa 2406 2305}%
\special{fp}%
\special{pa 2411 2268}%
\special{pa 2411 2266}%
\special{pa 2412 2264}%
\special{pa 2412 2262}%
\special{pa 2413 2260}%
\special{fp}%
\special{pa 2419 2224}%
\special{pa 2420 2222}%
\special{pa 2421 2220}%
\special{pa 2421 2217}%
\special{pa 2421 2217}%
\special{fp}%
\special{pa 2431 2181}%
\special{pa 2431 2181}%
\special{pa 2431 2179}%
\special{pa 2432 2176}%
\special{pa 2433 2174}%
\special{pa 2433 2174}%
\special{fp}%
\special{pa 2444 2138}%
\special{pa 2444 2138}%
\special{pa 2446 2134}%
\special{pa 2447 2131}%
\special{fp}%
\special{pa 2460 2096}%
\special{pa 2462 2092}%
\special{pa 2463 2089}%
\special{fp}%
\special{pa 2478 2055}%
\special{pa 2480 2051}%
\special{pa 2481 2048}%
\special{pa 2481 2048}%
\special{fp}%
\special{pa 2499 2015}%
\special{pa 2501 2011}%
\special{pa 2502 2008}%
\special{fp}%
\special{pa 2521 1976}%
\special{pa 2521 1976}%
\special{pa 2523 1974}%
\special{pa 2525 1970}%
\special{fp}%
\special{pa 2546 1939}%
\special{pa 2546 1939}%
\special{pa 2548 1937}%
\special{pa 2549 1935}%
\special{pa 2551 1933}%
\special{fp}%
\special{pa 2573 1904}%
\special{pa 2573 1903}%
\special{pa 2575 1901}%
\special{pa 2576 1899}%
\special{pa 2577 1898}%
\special{fp}%
\special{pa 2601 1869}%
\special{pa 2602 1868}%
\special{pa 2604 1866}%
\special{pa 2606 1865}%
\special{pa 2607 1864}%
\special{fp}%
\special{pa 2631 1837}%
\special{pa 2633 1835}%
\special{pa 2635 1834}%
\special{pa 2637 1832}%
\special{pa 2637 1832}%
\special{fp}%
\special{pa 2664 1806}%
\special{pa 2666 1804}%
\special{pa 2668 1803}%
\special{pa 2670 1801}%
\special{fp}%
\special{pa 2698 1777}%
\special{pa 2700 1775}%
\special{pa 2702 1774}%
\special{pa 2704 1772}%
\special{fp}%
\special{pa 2734 1750}%
\special{pa 2734 1750}%
\special{pa 2736 1748}%
\special{pa 2738 1747}%
\special{pa 2740 1745}%
\special{fp}%
\special{pa 2771 1725}%
\special{pa 2771 1725}%
\special{pa 2773 1723}%
\special{pa 2776 1722}%
\special{pa 2777 1721}%
\special{fp}%
\special{pa 2809 1702}%
\special{pa 2810 1702}%
\special{pa 2812 1700}%
\special{pa 2816 1698}%
\special{fp}%
\special{pa 2849 1681}%
\special{pa 2854 1679}%
\special{pa 2856 1678}%
\special{fp}%
\special{pa 2890 1662}%
\special{pa 2895 1660}%
\special{pa 2897 1659}%
\special{fp}%
\special{pa 2932 1646}%
\special{pa 2933 1646}%
\special{pa 2937 1644}%
\special{pa 2939 1643}%
\special{fp}%
\special{pa 2974 1633}%
\special{pa 2976 1632}%
\special{pa 2978 1631}%
\special{pa 2980 1631}%
\special{pa 2982 1630}%
\special{fp}%
\special{pa 3018 1621}%
\special{pa 3019 1621}%
\special{pa 3022 1620}%
\special{pa 3024 1620}%
\special{pa 3025 1620}%
\special{fp}%
\special{pa 3061 1612}%
\special{pa 3063 1612}%
\special{pa 3066 1611}%
\special{pa 3069 1611}%
\special{fp}%
\special{pa 3105 1606}%
\special{pa 3108 1605}%
\special{pa 3113 1605}%
\special{fp}%
\special{pa 3150 1602}%
\special{pa 3153 1601}%
\special{pa 3158 1601}%
\special{fp}%
\special{pa 3195 1600}%
\special{pa 3203 1600}%
\special{fp}%
\special{pa 3240 1601}%
\special{pa 3248 1601}%
\special{fp}%
\special{pa 3285 1605}%
\special{pa 3285 1605}%
\special{pa 3292 1605}%
\special{pa 3293 1605}%
\special{fp}%
\special{pa 3329 1611}%
\special{pa 3334 1611}%
\special{pa 3337 1612}%
\special{fp}%
\special{pa 3373 1619}%
\special{pa 3374 1619}%
\special{pa 3376 1620}%
\special{pa 3379 1620}%
\special{pa 3380 1621}%
\special{fp}%
\special{pa 3416 1629}%
\special{pa 3417 1630}%
\special{pa 3420 1631}%
\special{pa 3422 1631}%
\special{pa 3423 1632}%
\special{fp}%
\special{pa 3458 1643}%
\special{pa 3460 1644}%
\special{pa 3463 1644}%
\special{pa 3466 1645}%
\special{fp}%
\special{pa 3500 1658}%
\special{pa 3502 1659}%
\special{pa 3505 1660}%
\special{pa 3508 1661}%
\special{fp}%
\special{pa 3542 1677}%
\special{pa 3544 1678}%
\special{pa 3549 1680}%
\special{fp}%
\special{pa 3582 1697}%
\special{pa 3584 1698}%
\special{pa 3588 1700}%
\special{pa 3589 1701}%
\special{fp}%
\special{pa 3621 1720}%
\special{pa 3622 1721}%
\special{pa 3625 1722}%
\special{pa 3627 1723}%
\special{pa 3627 1723}%
\special{fp}%
\special{pa 3658 1744}%
\special{pa 3660 1745}%
\special{pa 3662 1747}%
\special{pa 3664 1748}%
\special{pa 3664 1748}%
\special{fp}%
\special{pa 3694 1771}%
\special{pa 3694 1771}%
\special{pa 3696 1772}%
\special{pa 3698 1774}%
\special{pa 3700 1775}%
\special{pa 3700 1775}%
\special{fp}%
\special{pa 3728 1799}%
\special{pa 3729 1800}%
\special{pa 3731 1801}%
\special{pa 3732 1803}%
\special{pa 3734 1805}%
\special{fp}%
\special{pa 3761 1829}%
\special{pa 3762 1830}%
\special{pa 3763 1832}%
\special{pa 3766 1835}%
\special{fp}%
\special{pa 3792 1862}%
\special{pa 3793 1863}%
\special{pa 3794 1865}%
\special{pa 3796 1866}%
\special{pa 3797 1867}%
\special{fp}%
\special{pa 3821 1896}%
\special{pa 3824 1899}%
\special{pa 3825 1901}%
\special{pa 3826 1902}%
\special{fp}%
\special{pa 3848 1931}%
\special{pa 3849 1933}%
\special{pa 3851 1935}%
\special{pa 3852 1937}%
\special{pa 3852 1937}%
\special{fp}%
\special{pa 3873 1968}%
\special{pa 3875 1970}%
\special{pa 3877 1974}%
\special{pa 3877 1974}%
\special{fp}%
\special{pa 3896 2006}%
\special{pa 3898 2009}%
\special{pa 3900 2013}%
\special{fp}%
\special{pa 3917 2046}%
\special{pa 3919 2048}%
\special{pa 3920 2051}%
\special{pa 3921 2053}%
\special{fp}%
\special{pa 3936 2087}%
\special{pa 3937 2089}%
\special{pa 3938 2092}%
\special{pa 3939 2094}%
\special{fp}%
\special{pa 3953 2129}%
\special{pa 3954 2131}%
\special{pa 3954 2134}%
\special{pa 3955 2136}%
\special{fp}%
\special{pa 3967 2171}%
\special{pa 3967 2172}%
\special{pa 3967 2174}%
\special{pa 3968 2176}%
\special{pa 3969 2179}%
\special{fp}%
\special{pa 3978 2214}%
\special{pa 3978 2215}%
\special{pa 3979 2218}%
\special{pa 3979 2220}%
\special{pa 3980 2222}%
\special{fp}%
\special{pa 3987 2258}%
\special{pa 3988 2259}%
\special{pa 3988 2264}%
\special{pa 3988 2265}%
\special{fp}%
\special{pa 3994 2302}%
\special{pa 3994 2306}%
\special{pa 3995 2309}%
\special{pa 3995 2310}%
\special{fp}%
\special{pa 3998 2347}%
\special{pa 3998 2351}%
\special{pa 3999 2353}%
\special{pa 3999 2354}%
\special{fp}%
\special{pa 4000 2391}%
\special{pa 4000 2399}%
\special{fp}%
\special{pa 3999 2437}%
\special{pa 3999 2444}%
\special{fp}%
\special{pa 3996 2481}%
\special{pa 3996 2483}%
\special{pa 3995 2486}%
\special{pa 3995 2489}%
\special{fp}%
\special{pa 3990 2526}%
\special{pa 3990 2528}%
\special{pa 3989 2530}%
\special{pa 3989 2533}%
\special{pa 3989 2533}%
\special{fp}%
\special{pa 3982 2569}%
\special{pa 3982 2570}%
\special{pa 3981 2572}%
\special{pa 3981 2574}%
\special{pa 3980 2577}%
\special{fp}%
\special{pa 3971 2612}%
\special{pa 3971 2613}%
\special{pa 3970 2616}%
\special{pa 3970 2618}%
\special{pa 3969 2620}%
\special{fp}%
\special{pa 3958 2655}%
\special{pa 3958 2656}%
\special{pa 3957 2659}%
\special{pa 3955 2662}%
\special{fp}%
\special{pa 3943 2697}%
\special{pa 3942 2699}%
\special{pa 3939 2704}%
\special{fp}%
\special{pa 3925 2738}%
\special{pa 3924 2740}%
\special{pa 3922 2744}%
\special{pa 3922 2745}%
\special{fp}%
\special{pa 3905 2779}%
\special{pa 3902 2784}%
\special{pa 3901 2785}%
\special{fp}%
\special{pa 3883 2818}%
\special{pa 3882 2819}%
\special{pa 3880 2821}%
\special{pa 3878 2824}%
\special{fp}%
\special{pa 3858 2855}%
\special{pa 3857 2857}%
\special{pa 3855 2859}%
\special{pa 3854 2862}%
\special{fp}%
\special{pa 3831 2892}%
\special{pa 3830 2893}%
\special{pa 3829 2895}%
\special{pa 3827 2897}%
\special{pa 3827 2897}%
\special{fp}%
\special{pa 3803 2926}%
\special{pa 3798 2931}%
\special{pa 3798 2932}%
\special{fp}%
\special{pa 3773 2959}%
\special{pa 3771 2960}%
\special{pa 3767 2964}%
\special{fp}%
}}%
%
{\color[named]{Black}{%
\special{pn 8}%
\special{pa 3200 2400}%
\special{pa 3200 1600}%
\special{dt 0.045}%
}}%
%
{\color[named]{Black}{%
\special{pn 8}%
\special{pa 3200 1600}%
\special{pa 4800 1600}%
\special{dt 0.045}%
}}%
%
{\color[named]{Black}{%
\special{pn 20}%
\special{ar 3400 2400 200 200  3.1415927  6.2831853}%
}}%
%
{\color[named]{Black}{%
\special{pn 8}%
\special{pa 3600 2400}%
\special{pa 3600 2408}%
\special{fp}%
\special{pa 3595 2443}%
\special{pa 3595 2446}%
\special{pa 3594 2447}%
\special{pa 3594 2450}%
\special{pa 3593 2451}%
\special{fp}%
\special{pa 3583 2481}%
\special{pa 3583 2482}%
\special{pa 3582 2482}%
\special{pa 3582 2484}%
\special{pa 3581 2484}%
\special{pa 3581 2486}%
\special{pa 3580 2486}%
\special{fp}%
\special{pa 3565 2513}%
\special{pa 3564 2514}%
\special{pa 3564 2515}%
\special{pa 3563 2515}%
\special{pa 3563 2516}%
\special{pa 3562 2517}%
\special{pa 3562 2518}%
\special{pa 3561 2518}%
\special{pa 3561 2519}%
\special{fp}%
\special{pa 3541 2541}%
\special{pa 3541 2541}%
\special{pa 3541 2542}%
\special{pa 3540 2543}%
\special{pa 3539 2543}%
\special{pa 3539 2544}%
\special{pa 3538 2544}%
\special{pa 3538 2545}%
\special{pa 3537 2546}%
\special{fp}%
\special{pa 3514 2564}%
\special{pa 3513 2565}%
\special{pa 3512 2565}%
\special{pa 3512 2566}%
\special{pa 3511 2566}%
\special{pa 3511 2567}%
\special{pa 3510 2567}%
\special{pa 3509 2568}%
\special{pa 3509 2568}%
\special{fp}%
\special{pa 3480 2583}%
\special{pa 3480 2583}%
\special{pa 3480 2584}%
\special{pa 3477 2584}%
\special{pa 3477 2585}%
\special{pa 3475 2585}%
\special{pa 3475 2586}%
\special{fp}%
\special{pa 3443 2595}%
\special{pa 3442 2596}%
\special{pa 3437 2596}%
\special{pa 3437 2597}%
\special{pa 3436 2597}%
\special{fp}%
\special{pa 3400 2600}%
\special{pa 3392 2600}%
\special{fp}%
\special{pa 3357 2595}%
\special{pa 3354 2595}%
\special{pa 3353 2594}%
\special{pa 3350 2594}%
\special{pa 3349 2593}%
\special{fp}%
\special{pa 3318 2583}%
\special{pa 3318 2583}%
\special{pa 3318 2582}%
\special{pa 3316 2582}%
\special{pa 3316 2581}%
\special{pa 3314 2581}%
\special{pa 3314 2580}%
\special{pa 3314 2580}%
\special{fp}%
\special{pa 3286 2564}%
\special{pa 3286 2564}%
\special{pa 3285 2564}%
\special{pa 3285 2563}%
\special{pa 3284 2563}%
\special{pa 3283 2562}%
\special{pa 3282 2562}%
\special{pa 3282 2561}%
\special{pa 3281 2561}%
\special{pa 3281 2561}%
\special{fp}%
\special{pa 3258 2541}%
\special{pa 3257 2540}%
\special{pa 3257 2539}%
\special{pa 3256 2539}%
\special{pa 3255 2538}%
\special{pa 3255 2537}%
\special{pa 3254 2537}%
\special{pa 3254 2536}%
\special{pa 3254 2536}%
\special{fp}%
\special{pa 3235 2513}%
\special{pa 3234 2512}%
\special{pa 3234 2511}%
\special{pa 3233 2511}%
\special{pa 3233 2510}%
\special{pa 3232 2509}%
\special{pa 3232 2508}%
\special{pa 3231 2508}%
\special{pa 3231 2508}%
\special{fp}%
\special{pa 3217 2480}%
\special{pa 3216 2479}%
\special{pa 3216 2477}%
\special{pa 3215 2477}%
\special{pa 3215 2475}%
\special{pa 3214 2475}%
\special{pa 3214 2474}%
\special{fp}%
\special{pa 3204 2441}%
\special{pa 3204 2438}%
\special{pa 3203 2437}%
\special{pa 3203 2434}%
\special{fp}%
\special{pa 3200 2397}%
\special{pa 3200 2390}%
\special{fp}%
\special{pa 3205 2354}%
\special{pa 3205 2353}%
\special{pa 3206 2353}%
\special{pa 3206 2349}%
\special{pa 3207 2349}%
\special{pa 3207 2348}%
\special{fp}%
\special{pa 3217 2318}%
\special{pa 3218 2318}%
\special{pa 3218 2316}%
\special{pa 3219 2316}%
\special{pa 3219 2314}%
\special{pa 3220 2314}%
\special{pa 3220 2313}%
\special{fp}%
\special{pa 3236 2285}%
\special{pa 3236 2285}%
\special{pa 3237 2285}%
\special{pa 3237 2284}%
\special{pa 3238 2283}%
\special{pa 3238 2282}%
\special{pa 3239 2282}%
\special{pa 3239 2281}%
\special{pa 3240 2280}%
\special{fp}%
\special{pa 3260 2257}%
\special{pa 3260 2257}%
\special{pa 3261 2257}%
\special{pa 3261 2256}%
\special{pa 3262 2255}%
\special{pa 3263 2255}%
\special{pa 3263 2254}%
\special{pa 3264 2254}%
\special{pa 3264 2253}%
\special{fp}%
\special{pa 3287 2235}%
\special{pa 3288 2234}%
\special{pa 3289 2234}%
\special{pa 3289 2233}%
\special{pa 3290 2233}%
\special{pa 3291 2232}%
\special{pa 3292 2232}%
\special{pa 3292 2231}%
\special{pa 3292 2231}%
\special{fp}%
\special{pa 3321 2216}%
\special{pa 3321 2216}%
\special{pa 3323 2216}%
\special{pa 3323 2215}%
\special{pa 3325 2215}%
\special{pa 3325 2214}%
\special{pa 3326 2214}%
\special{fp}%
\special{pa 3358 2204}%
\special{pa 3362 2204}%
\special{pa 3363 2203}%
\special{pa 3366 2203}%
\special{fp}%
\special{pa 3402 2200}%
\special{pa 3410 2200}%
\special{fp}%
\special{pa 3444 2205}%
\special{pa 3447 2205}%
\special{pa 3447 2206}%
\special{pa 3450 2206}%
\special{pa 3451 2207}%
\special{fp}%
\special{pa 3481 2217}%
\special{pa 3482 2217}%
\special{pa 3482 2218}%
\special{pa 3484 2218}%
\special{pa 3484 2219}%
\special{pa 3486 2219}%
\special{pa 3486 2220}%
\special{fp}%
\special{pa 3513 2235}%
\special{pa 3514 2236}%
\special{pa 3515 2236}%
\special{pa 3515 2237}%
\special{pa 3516 2237}%
\special{pa 3517 2238}%
\special{pa 3518 2238}%
\special{pa 3518 2239}%
\special{pa 3518 2239}%
\special{fp}%
\special{pa 3541 2258}%
\special{pa 3543 2260}%
\special{pa 3543 2261}%
\special{pa 3544 2261}%
\special{pa 3545 2262}%
\special{pa 3545 2263}%
\special{pa 3546 2263}%
\special{pa 3546 2263}%
\special{fp}%
\special{pa 3565 2286}%
\special{pa 3565 2287}%
\special{pa 3566 2288}%
\special{pa 3566 2289}%
\special{pa 3567 2289}%
\special{pa 3567 2290}%
\special{pa 3568 2291}%
\special{pa 3568 2292}%
\special{pa 3568 2292}%
\special{fp}%
\special{pa 3583 2319}%
\special{pa 3583 2320}%
\special{pa 3584 2321}%
\special{pa 3584 2323}%
\special{pa 3585 2323}%
\special{pa 3585 2325}%
\special{fp}%
\special{pa 3595 2357}%
\special{pa 3595 2358}%
\special{pa 3596 2358}%
\special{pa 3596 2362}%
\special{pa 3597 2363}%
\special{pa 3597 2364}%
\special{fp}%
\special{pa 3600 2400}%
\special{pa 3600 2400}%
\special{fp}%
}}%
\put(33.3000,-21.3000){\makebox(0,0)[lb]{$C_{-1}$}}%
\put(30.4000,-26.000){\makebox(0,0)[lb]{$-2$}}%
%
{\color[named]{Black}{%
\special{pn 8}%
\special{pa 3400 2400}%
\special{pa 3400 2200}%
\special{dt 0.045}%
}}%
%
{\color[named]{Black}{%
\special{pn 8}%
\special{pa 3400 2200}%
\special{pa 4800 2200}%
\special{dt 0.045}%
}}%
%
{\color[named]{Black}{%
\special{pn 8}%
\special{pa 4900 2200}%
\special{pa 4700 2200}%
\special{fp}%
}}%
\put(45.9000,-21.5000){\makebox(0,0)[lb]{$1/4$}}%
\end{picture}%

\caption{Walls for $v=1-3\varrho_X$.}
\end{figure}

(3)
Assume that $\ell=4$.
By Corollary \ref{cor:square}, 
all walls in $s<0$ intersect with the line $s=-\sqrt{4}=-2$.
Assume that $v_1$ defines a wall for $v$.
Since $v e^{2H}=(1,2H,0)$, $u_1 e^{2H}=(r,H,a)$.
Hence $v_1=(r,(1-2r)H,a-4+4r)$ and
$v-v_1=(1-r,(2r-1)H,-a-4r)$. Replacing $v_1$ by
$v-v_1$ if necessary,
we may assume that $r> 0$.
We have $2n:=\langle v_1^2 \rangle \geq 0$, $\langle (v-v_1)^2 \rangle \geq 0$
and $\langle v_1,v-v_1 \rangle >0$.
Hence $n=1-ra$, $4-2n>a \geq -n$.
Then $n=0,1,2,3$, which implies that $ra=1,0,-1,-2$.
We also have $\left|\frac{a-4+4r+4r}{2(1-2r)}\right| >2$,
which implies that $a(8(2r-1)+a)>0$. In particular, $a \ne 0$.
Then $ra=1,0,-1,-2$ and $r> 0$ implies that
$r=1,2$.
\begin{NB}
We don't need the following argument, since $r>0$:
If $r=0$, then $n=1$ and $2>a \geq -1$. Since $a(a-8)>0$,
$a=-1$. Thus $v_1=(0,H,-5)$.
\end{NB}
If $r=1$, then $a(a+8)>0$ and $ra=1,0,-1,-2$ imply that
$a=1$. Thus $v_1=(1,-H,1)$.
If $r=2$, then $a=-1$ and $a(a+24)>0$, which is impossible.
Therefore $v_1=(1,-H,1)$.
Thus $W_{v_1}$ is the unique wall which is defined by
\begin{NB}
$$
\left|z+\frac{5}{2} \right|=\frac{3}{2}.
$$
\end{NB}
$$
\left(s+\frac{5}{2} \right)^2+t^2=\frac{3^2}{2^2}.
$$

\begin{prop}
Let $v$ be a positive and primitive
Mukai vector with $\langle v^2 \rangle=8$.
Then $M_H(v)$ is isomorphic to
$\Hilb{4}{X} \times X$ or $M_H(0,2H,-1)$.
\end{prop}

\begin{proof}
We first prove that $M_H(1,0,-4) \not \cong
M_H(0,2H,-1)$.
We note that the Hilbert-Chow morphism
of $\Hilb{4}{X}$ induces
a divisorial contraction
of $M_H(1,0,-4)$.
We note that $M_H(3,H,-1)$ has a morphism
to the Uhlenbeck compactification of
the moduli of stable vector bundles,
which contracts a ${\Bbb P}^2$-bundle
over $M_H(3,H,0)$.
Let ${\bf P}$ be the Poincar\'{e} line bundle on
$X \times X$, where we identify
$\Piczero{X}$ with $X$.
Then we have an isomorphism
$M_H(1,0,-4) \cong M_H(1,H,-3) \cong 
M_H(3,H,-1)$ by sending $E$ to
$\Phi_{X \to X}^{{\bf P}[1]}(E(H))$
\cite[Prop. 3.5]{Y:7}.
Hence $M_H(1,0,-4)$ has another contraction.
There is no other contraction by a similar
argument in \cite[Example 7.2]{Y:7}.
On the other hand,
$M_H(0,2H,-1)$ has a Lagrangian fibration.
Therefore $M_H(1,0,-4) \not \cong M_H(0,2H,-1)$.

For $s=-2$ and $v=(1,0,-4)$, we have two moduli spaces
$M_{(-2H,t_1 H)}(1,0,-4)$ and
$M_{(-2H,t_2 H)}(1,0,-4)$, where $t_1 >\sqrt{2}$ and
$t_2<\sqrt{2}$.
For $E \in M_{(-2H,t_2 H)}(1,0,-4)$,
$\Phi_{X \to X}^{{\bf P}[1]}(E(2H))
\in M_H(0,2H,-1)$ and we have an isomorphism
$M_{(-2H,t_2 H)}(1,0,-4) \cong M_H(0,2H,-1)$.

By \cite[sect. 7.3]{YY}, every quadratic form
$r x^2+2d xy+a y^2$ is equivalent to
$x^2-4y^2$.
Since $(1,0,-4)^{\vee}=(1,0,-4)$,
there is an auto-equivalence
$\Phi$ such that $\Phi(v)=(1,0,-4)$.
Then $\Phi(M_H(v)) \cong M_{(sH,tH)}(1,0,-4)$
for a suitable $(s,t)$.
Therefore the claim holds.
\end{proof}

(4)
Assume that $\ell=5$.
In this case,
$S_{1,5}/\{ \pm 1 \}$ is generated by
\begin{equation*}
A_5:=
\begin{pmatrix}
2 & 5\\
1 & 2
\end{pmatrix}.
\end{equation*}
Hence we have a numerical solution
 $v=5(1,-2H,4)-(4,-10H,25)$.
Then we have $u_{-1}=(1,-2H,4)$ and
$W_{u_{-1}}$ is the circle defined by
$$
\left(s+\frac{9}{4} \right)^2+t^2=\frac{1}{4^2}.
$$

(5)
Assume that $\ell=6$.
In this case,
$S_{1,6}/\{ \pm 1 \}$ is generated by
\begin{equation*}
A_6 :=
\begin{pmatrix}
5 & 12\\
2 & 5
\end{pmatrix}.
\end{equation*}
Hence we have a numerical solution
$v=(5^2,-60H,12^2)-6(4,-10H,5^2)$.
Then we have $u_{-1}=(5^2,-10H,5^2)$ and
 $W_{u_{-1}}$ is the circle defined by
$$
\left(s+\frac{49}{20} \right)^2+t^2=\frac{1}{20^2}.
$$

\begin{rem}
Let $X$ be an arbitrary abelian surface and $H$ an ample
divisor on $X$.
Let $v$ be a primitive Mukai vector with 
$\langle v^2 \rangle/2=1$.
Then we have an isomorphism
$X \times \Piczero{X} \to M_H(v)$ by sending
$(x, L) \in X \times \Piczero{X}$ to
$T_x^*(E_0) \otimes L$, where $E_0$ is an element 
of $M_H(v)$ and $T_x$ is the translation by $x$
(\cite[Cor. 4.3]{Y:nagoya}).
\end{rem}

\section{Appendix}

\subsection{The action of Fourier-Mukai transforms on ${\Bbb H}$.}
Assume that $\NS(X)={\Bbb Z}H$.
Then $\beta+\sqrt{-1}\omega=\frac{z}{ \sqrt{n}}H$ with  
$z \in {\Bbb H}$, where
$z:=x+\sqrt{-1}y$ with $x \in {\Bbb R}$ and $y \in {\Bbb R}_{>0}$.
Thus we have an identification of
$\NS(X)_{\Bbb R} \times \Amp(X)_{\Bbb R}$ with 
${\Bbb H}$.
We set $Z_z:=Z_{(\beta,\omega)}$.

We study the action of $\Phi$ with $\gamma \in {\Bbb Q}H$.
We write 
$$
\gamma'+\widehat{\xi}+\sqrt{-1} \widehat{\eta}=
\frac{z'}{\sqrt{n}}H,\; z' \in {\Bbb H}.
$$

Let $\Phi_{X \to X_1}^{\bf E}:{\bf D}(X) \to
{\bf D}(X_1)$
be a Fourier-Mukai transform
such that ${\bf E}$ is a coherent sheaf.
Then there is 
\begin{equation*}
A=\begin{pmatrix}
a & b \\
c & d 
\end{pmatrix}
\in G
\end{equation*}
such that 
$\mu(\Phi_{X \to X_1}^{\bf E}({\frak k}_x))=\frac{a}{c} \sqrt{n}$
and
$\mu(\Phi_{X \to X_1}^{\bf E}({\cal O}_X))=\frac{b}{d}\sqrt{n}$.
$X_1=M_H(c^2 e^{\frac{d}{c}\frac{H}{\sqrt{n}}})$ and 
${\bf E}$ is unique up to the action of $X \times \Piczero{X}$.
Then 
$$
\theta(\Phi_{X \to X_1}^{{\bf E}})=
\begin{pmatrix}
d & b \\
c & a
\end{pmatrix}
\in G.
$$ 

\begin{defn}
For $\Phi_{X \to X_1}^{\bf E}$,
we set 
$$
\varphi(\Phi_{X \to X_1}^{\bf E}):=
\pm \begin{pmatrix}
a & b\\
c & d 
\end{pmatrix} \in G/\{\pm 1\}.
$$
\end{defn}

For $\Phi=\Phi_{X \to X_1}^{{\bf E}}$,
we have
\begin{equation*}
\begin{split}
r_1 e^{\gamma}= &c^2-\frac{cd}{\sqrt{n}}H+d^2 \varrho_X, \\
r_1 e^{\gamma'}= & c^2+\frac{ac}{\sqrt{n}}\widehat{H}+a^2 \varrho_{X_1}.
\end{split}
\end{equation*}
Hence
\begin{equation*}
\frac{z}{\sqrt{n}}+\frac{d}{c \sqrt{n}}=-\lambda+\sqrt{-1} t,\;
\frac{z'}{\sqrt{n}}-\frac{a}{c \sqrt{n}}=
\frac{\lambda+\sqrt{-1} t}{(\lambda^2+t^2)nc^2}.
\end{equation*}
Thus we get
\begin{equation*}
\begin{split}
\frac{z'}{\sqrt{n}}=& \frac{a}{c \sqrt{n}}-\frac{1}{c^2 n 
(\frac{z}{\sqrt{n}}+\frac{d}{c \sqrt{n}})}\\
=& \frac{ac z-(1-ad)}{c(cz+d)\sqrt{n}}
=\frac{a z+b}{\sqrt{n}(c z+d)}.
\end{split}
\end{equation*}
Therefore the action of $A$ on ${\Bbb H}$ is the natural action of
$\SL(2,{\Bbb R})$.

$$
\zeta=-c^2 \frac{(\lambda-t \sqrt{-1})^2 (H^2)}{2}=
-c^2 \frac{(\frac{z}{\sqrt{n}}+\frac{d}{c\sqrt{n}})^2 (H^2)}{2}
=-(cz+d)^2.
$$

\begin{prop}
For $\Phi_{X \to X_1}^{\bf E}$ with
$\varphi(\Phi_{X \to X_1}^{\bf E})=
\begin{pmatrix}
a & b\\
c & d
\end{pmatrix}
 \in G$,
$$
-(cz+d)^2 Z_{\frac{az+b}{cz+d}\frac{\widehat{H}}{\sqrt{n}}}
=Z_{z \frac{H}{\sqrt{n}}} \circ (\Phi_{X \to X_1}^{\bf E})^{-1}.
$$
We also have
$$
\Phi_{X \to X_1}^{\bf E}({\Bbb Q}e^{\lambda \frac{H}{{\sqrt{n}}}})
={\Bbb Q}e^{\frac{a\lambda+b}{c\lambda+d}\frac{\widehat{H}}{\sqrt{n}}}.
$$
\end{prop}

We now extend the action of $G$ to $\widehat{G}$.
We set 
$$
\Delta:=\begin{pmatrix}
1 & 0\\
0 & -1
\end{pmatrix}
$$
We note that
$$
\widehat{G}=G \rtimes  \left\langle
\Delta
\right\rangle
$$
with
$$
\begin{pmatrix}
1 & 0\\
0 & -1
\end{pmatrix}
\begin{pmatrix}
a & b\\
c & d
\end{pmatrix}
\begin{pmatrix}
1 & 0\\
0 & -1
\end{pmatrix}
=
\begin{pmatrix}
a & -b\\
-c & d
\end{pmatrix}.
$$
We define the action of
$\Delta$ on ${\Bbb H}$ as $\Delta(z):=-\overline{z}$.
Then we have
$$
\Delta(A(\Delta(z)))=\frac{az-b}{-cz+d},\;
A=\begin{pmatrix}
a & b\\
c & d
\end{pmatrix}.
$$
Thus we have an action of $\widehat{G}$ on 
${\Bbb H}$.

\begin{prop}
We can extend the action of $G$ to
the action of $\widehat{G}$ by
$$
(g,\Delta^n) \cdot z:=
\begin{cases}
g \cdot z, & 2|n,\\
-\overline{g \cdot z},& 2 \not | n,
\end{cases} 
$$
where $g \in G$.
\end{prop}

\begin{rem}
In \cite{YY}, we showed that
the cohomological action of
$\Eq_0({\bf D}(X),{\bf D}(X))$
defines a normal subgroup of $G$
which is a conjugate of $\Gamma_0(n)$ in 
$\GL(2,{\Bbb R})$.
More precisely,
we set 
$G_0:=\theta(\Eq_0({\bf D}(X),{\bf D}(X)))$.
Then 
\begin{equation*}
\begin{pmatrix}
\sqrt{n} & 0 \\
0 & 1 
\end{pmatrix}^{-1}
G_0
\begin{pmatrix}
\sqrt{n} & 0 \\
0 & 1 
\end{pmatrix}
=\Gamma_0(n).
\end{equation*}
We set $\beta+\sqrt{-1}\omega=w H$.
Then $\Eq_0({\bf D}(X),{\bf D}(X))$ acts on $w$-plane
as the action of $\Gamma_0(n)$ on $w$-plane.
\begin{NB}
We shall explain known description of Mukai lattices
\cite[sect. 5]{Ma}.

We set
\begin{equation*}
L_n:=
\left\{
\left.
\begin{pmatrix}
np & q \\
nr & -np 
\end{pmatrix}
\in {\frak {sl}}_2({\Bbb R}) \right|
p,q,r \in {\Bbb Z}
\right\}.
\end{equation*}
For $B_1,B_2 \in L_n$, we set
$(B_1,B_2):=\tr(B_1 B_2)/(2n) \in {\Bbb Z}$.
Then we have an isometry
$$
\begin{matrix}
\zeta:&
\Sym_2(n,{\Bbb Z}) & \to & L_n\\
& 
\begin{pmatrix}
x & y \sqrt{n} \\
y\sqrt{n} & z  
\end{pmatrix}
& \mapsto &
\begin{pmatrix}
ny & -x \\
nz & -ny 
\end{pmatrix}.
\end{matrix}
$$
It is easy to see that
the induced action of $G_0$ on $L_n$ is the adjoint action of 
$\Gamma_0(n)$. Thus
$$
g
\begin{pmatrix}
ny & -x \\
nz & -ny 
\end{pmatrix}
g^{-1},\;
g=
\begin{pmatrix}
a & b \\
c & d 
\end{pmatrix}
\in \Gamma_0(n)
$$
corresponds to 
$$
{{}^t g_1}
\begin{pmatrix}
x & y \sqrt{n} \\
y \sqrt{n} & z 
\end{pmatrix}
g_1,\;
g_1=
\begin{pmatrix}
a & \frac{c}{\sqrt{n}} \\
b \sqrt{n} & d 
\end{pmatrix}
\in G_0
$$
via $\zeta^{-1}$.
\end{NB}

\end{rem}

\begin{NB}
We set $w:=\frac{z}{\sqrt{n}}$ and $w':=\frac{z'}{\sqrt{n}}$.
Then 
$$
-(c\sqrt{n} w+d)^2 Z_{\frac{aw+\frac{b}{\sqrt{n}}}{c\sqrt{n}w+d}\widehat{H}}
=Z_{w H} \circ \widehat{\Phi}.
$$

\begin{equation*}
\begin{pmatrix}
a & \frac{b}{\sqrt{n}} \\
c\sqrt{n} & d 
\end{pmatrix}=
\begin{pmatrix}
\sqrt{n} & 0 \\
0 & 1 
\end{pmatrix}^{-1}
\begin{pmatrix}
a & b \\
c & d 
\end{pmatrix}
\begin{pmatrix}
\sqrt{n} & 0 \\
0 & 1 
\end{pmatrix}
\end{equation*}
and 
$$
\Gamma_0(n)=
\begin{pmatrix}
\sqrt{n} & 0 \\
0 & 1 
\end{pmatrix}^{-1}
\ker \phi
\begin{pmatrix}
\sqrt{n} & 0 \\
0 & 1 
\end{pmatrix}
$$
is the group of the auto-equivalences.
\end{NB}

\begin{NB}
For a rational number $\frac{a}{b}$,
there is an isotropic Mukai vector
$(p^2 n,pq H,q^2)$ such that
$(p n,q)=1$ and $\frac{q}{pn}$ is sufficiently close to
$\frac{a}{b}$.

We take $p_1, q_1 \in {\Bbb Z}$ such that
$(p_1 n,q_1)=1$.
We set $(p,q)=(p_1(1+b q_1 m),q_1(1+an p_1 m))$.
Then $(n,q_1(1+an p_1 m))=1$ for all $m$.
We shall prove that 
there are sufficiently large integer $m$ 
such that $(1+b q_1 m,1+an p_1 m)=1$.
We first assume that $|b q_1| < |an p_1|$.
Then $|1+b q_1 m| < |1+an p_1 m|$ for $|m| \gg 0$.
We can choose $m$ such that $1+an p_1 m$ is prime.
Then $(1+b q_1 m, 1+an p_1 m)=1$.
If $|b q_1| > |an p_1|$, then we also get the claim.
Assume that $|b q_1| = |an p_1|$. Then
$a=q_1 x$ and $b= \pm p_1 n x$.
Hence $\frac{a}{b}=\pm \frac{q_1}{n p_1}$.
So if $(p_1,q_1)$ satisfies 
$\frac{a}{b} \ne \pm \frac{q_1}{n p_1}$,
then we can take a suitable $m$.
  
\end{NB}

\begin{NB}

\subsection{Modular groups.}

\begin{equation}
\begin{pmatrix}
x & y \\
y & z 
\end{pmatrix}
\begin{pmatrix}
0 & -1 \\
1 & 0 
\end{pmatrix}
=
\begin{pmatrix}
y & -x \\
z & -y 
\end{pmatrix}
 \in {\frak {sl}}_2({\Bbb R})
\end{equation}
Since
$$
\begin{pmatrix}
d & -c \\
-b & a 
\end{pmatrix}=
\begin{pmatrix}
0 & 1 \\
-1 & 0 
\end{pmatrix}
\begin{pmatrix}
a & b \\
c & d 
\end{pmatrix}\begin{pmatrix}
0 & -1 \\
1 & 0 
\end{pmatrix},
$$

$$
^t P 
\begin{pmatrix}
x & y \\
y & z 
\end{pmatrix}
 P
\begin{pmatrix}
0 & 1 \\
-1 & 0 
\end{pmatrix}
=
(\det P) {^t P}
\begin{pmatrix}
y & -x \\
z & -y 
\end{pmatrix}
{^t P^{-1}} 
$$

We define an anti-homomorphism
\begin{equation}
\begin{matrix}
\varphi: & \GL_2({\Bbb Q})_+ & \to & \SL_2({\Bbb R})\\
& Q & \mapsto & \frac{^t Q}{\sqrt{\det Q}}
\end{matrix}
\end{equation}
Then
$$
^t \varphi(Q) 
\begin{pmatrix}
x & y \\
y & z 
\end{pmatrix}
\varphi(Q)
\begin{pmatrix}
0 & 1 \\
-1 & 0 
\end{pmatrix}
=
Q
\begin{pmatrix}
y & -x \\
z & -y 
\end{pmatrix}
{Q^{-1}}. 
$$

Let $M_\sigma, N_\sigma$ be positive integers
such that $M_\sigma N_\sigma=n$ and 
$\gcd(M_\sigma, N_\sigma)=1$.
We take integers $a_\sigma, b_\sigma$ 
such that $a_\sigma N_\sigma-b_\sigma M_\sigma=1$.
For 
$$
\gamma_\sigma:=
\begin{pmatrix}
a_\sigma N_\sigma & b_\sigma \\
n & N_\sigma
\end{pmatrix},
$$

$$
\begin{pmatrix}
1 & 0\\
0 & \sqrt{n}
\end{pmatrix}
\varphi(\gamma_\sigma)
\begin{pmatrix}
1 & 0\\
0 & \sqrt{n}
\end{pmatrix}^{-1}
=
\frac{1}{\sqrt{N_\sigma}}
^t \begin{pmatrix}
a_\sigma N_\sigma & b_\sigma \sqrt{n} \\
\sqrt{n} & N_\sigma
\end{pmatrix}
=
{^t \begin{pmatrix}
a_\sigma \sqrt{N_\sigma} & b_\sigma \sqrt{M_\sigma} \\
\sqrt{M_\sigma} & \sqrt{N_\sigma}
\end{pmatrix}}.
$$

$$
\begin{pmatrix}
1 & 0\\
0 & \sqrt{n}
\end{pmatrix}
\varphi
\left(
\begin{pmatrix}
a & b\\
nc & d \\
\end{pmatrix}
\right)
\begin{pmatrix}
1 & 0\\
0 & \sqrt{n}
\end{pmatrix}^{-1}
=
^t \begin{pmatrix}
a & c\sqrt{n} \\
b\sqrt{n} & d
\end{pmatrix}.
$$

We have equivalriant morphisms
\begin{equation}
\begin{matrix}
\Sym_2(n,{\Bbb Z}) & \to & \Sym_2({\Bbb R}) & \to & \frak{sl}_2({\Bbb R})  \\
\begin{pmatrix}
x & \sqrt{n} y\\
\sqrt{n} y & z
\end{pmatrix} 
& \mapsto &
\begin{pmatrix}
x & n y\\
n y & nz
\end{pmatrix} 
& \mapsto &
\begin{pmatrix}
ny & -x\\
nz & -ny
\end{pmatrix} 
\end{matrix}
\end{equation}

\begin{equation}
\begin{pmatrix}
1 & 0\\
0 & \sqrt{n} 
\end{pmatrix}
\begin{pmatrix}
x & \sqrt{n} y\\
\sqrt{n} y & z
\end{pmatrix}
\begin{pmatrix}
1 & 0\\
0 & \sqrt{n} 
\end{pmatrix}
=
\begin{pmatrix}
x & n y\\
n y & nz
\end{pmatrix}
\end{equation}

\begin{equation}
\begin{pmatrix}
1 & 0\\
0 & \sqrt{n} 
\end{pmatrix}^{-1}
\begin{pmatrix}
a & b\\
c & d
\end{pmatrix}
\begin{pmatrix}
1 & 0\\
0 & \sqrt{n} 
\end{pmatrix}
=
\begin{pmatrix}
a & b \sqrt{n}\\
\frac{c}{\sqrt{n}}  & d
\end{pmatrix}
\end{equation}

\begin{equation}
\begin{pmatrix}
ny & -x\\
n z & -ny
\end{pmatrix}
\begin{pmatrix}
d & -\frac{c}{\sqrt{n}}\\
-b \sqrt{n}  & a
\end{pmatrix}
=
\begin{pmatrix}
1 & 0\\
0 & \sqrt{n} 
\end{pmatrix}
\begin{pmatrix}
x & \sqrt{n} y\\
\sqrt{n} y & z
\end{pmatrix}
\begin{pmatrix}
a & b\\
c & d
\end{pmatrix}
\begin{pmatrix}
1 & 0\\
0 & \sqrt{n} 
\end{pmatrix}
\begin{pmatrix}
0 & -1\\
1 & 0 
\end{pmatrix}
\end{equation}

\begin{equation}
\begin{split}
&
\begin{pmatrix}
1 & 0 \\
0 & \sqrt{n} 
\end{pmatrix}
{^t \varphi \left(
\begin{pmatrix}
a & b \sqrt{n} \\
c \sqrt{n} & d 
\end{pmatrix}
\right)}
\begin{pmatrix}
x & y \sqrt{n} \\
y \sqrt{n} & z 
\end{pmatrix}
\varphi \left(
\begin{pmatrix}
a & b \sqrt{n} \\
c \sqrt{n} & d 
\end{pmatrix}
\right)
\begin{pmatrix}
1 & 0 \\
0 & \sqrt{n} 
\end{pmatrix}
\begin{pmatrix}
0 & -1 \\
1 & 0 
\end{pmatrix}\\
=& ^t 
\begin{pmatrix}
a & b  \\
cn & d 
\end{pmatrix}
\begin{pmatrix}
ny & -x \\
nz & -ny 
\end{pmatrix}
\begin{pmatrix}
a & b  \\
cn & d 
\end{pmatrix}^{-1}
\end{split}
\end{equation}

It seems that the cusps of $\Gamma_0(n)$ corresponds to
the Fourier-Mukai partner of $X$.

\end{NB}

\end{document}